\documentclass[12pt]{amsart}
\usepackage[margin=1in]{geometry}
\usepackage{latexsym}
\usepackage{float}
\usepackage{amssymb}
\usepackage[all]{xy}
\usepackage{epsfig}
\usepackage{color}
\usepackage{graphics}

\newtheorem{Thm}[equation]{Theorem}
\newtheorem{Lem}[equation]{Lemma}
\newtheorem{Cor}[equation]{Corollary}

\theoremstyle{remark}
\newtheorem{Rem}[equation]{Remark}

\theoremstyle{definition}

\numberwithin{equation}{section}

\newcommand{\be}{\begin{equation}}
\newcommand{\beu}{\begin{equation*}}

\renewcommand{\bar}[1]{\overline{#1}}

\begin{document}

\title{An Algorithm to classify rational 3-tangles}
\author{Bo-hyun Kwon}
\footnote{The subject classification code: 57M27}

\begin{abstract}
 A $3$-$tangle$ $T$  is the disjoint union of $3$ properly embedded arcs in the unit 3-ball; it is called rational if there is a homeomorphism of pairs from
$(B^3,T)$ to $(D^2\times I,\{x_1,x_2,x_3\}\times I)$. Two rational 3-tangles $T$ and $T'$ are isotopic if there is an orientation-preserving self-homeomorphism $h: (B^3, T)\rightarrow (B^3,T')$ that is the identity map on the boundary.
In this paper, we give an algorithm to check whether or not two rational 3-tangles are isotopic  by using a modified version of Dehn's method for classifying simple closed curves on surfaces.
\end{abstract}

\maketitle
\section{Introduction}

Tangles were introduced by J. Conway.  In 1970, he proved that every rational 2-tangle defines a rational number and two rational 2-tangles are isotopic if and only if they have the same rational number.
However, there is no similar invariant known which classifies rational $3$-tangles. In this paper, I describe an algorithm to check whether  or not two rational $3$-tangles are isotopic.\\

A $\emph{n-tangle}$ is the disjoint union of $n$ properly embedded arcs in the unit 3-ball; the embedding must send the endpoints of the arcs to $2n$ marked (fixed) points on the ball's boundary. Without loss of generality, consider the marked points on the 3-ball boundary to lie on a great circle. The tangle can be arranged to be in general position with respect to the projection onto the flat disk in the $xy$-plane bounded by the great circle. The projection then gives us a $\emph{tangle diagram}$, where we make note of over and undercrossings as with knot diagrams.
  A $rational$ $n$-$tangle$ is a $n$-tangle $\alpha_1\cup\alpha_2\cup\cdot\cdot\cdot \cup\alpha_n$ in a 3-ball $B^3$
 such that there exists a homeomorphism of pairs
 $\bar{H}: (B^3,\alpha_1\cup\alpha_2\cup\cdot\cdot\cdot \cup\alpha_n)\longrightarrow
(D^2\times I,\{p_1,p_2,\cdot\cdot\cdot ,p_n\}\times I)$, where $I=[0,1]$.\\

We note that there exists a homeomorphism $K:
(D^2\times I,\{p_1,p_2, \cdot\cdot\cdot,p_n\}\times I)\rightarrow (B^3,\epsilon_1\cup\epsilon_2 \cup\cdot\cdot\cdot\cup\epsilon_n)$, where $\epsilon_1\cup\epsilon_2\cup\cdot\cdot\cdot\cup\epsilon_n$ is the $\infty$ tangle as in Figure~\ref{A1}.\\

Therefore, alternatively, a $n$-tangle $\alpha_1\cup\alpha_2\cup\cdot\cdot\cdot\cup\alpha_n$ is $\emph{rational}$ if there exists a homeomorphism of pairs:
$\widehat{H}=(\bar{H})^{-1}K^{-1}: (B^3,\epsilon_1\cup\epsilon_2\cup\cdot\cdot\cdot \cup\epsilon_n)\rightarrow (B^3,\alpha_1\cup\alpha_2\cup\cdot\cdot\cdot \cup\alpha_n)$.\\

Two rational $n$-tangles, $T,T'$, in $B^3$ are $\emph{isotopic}$, denoted by $T\approx T'$, if there is an orientation-preserving self-homeomorphism
$h: (B^3, T)\rightarrow (B^3,T')$ that is the identity map on the boundary.\\

\begin{figure}[htb]
\includegraphics[scale=.2]{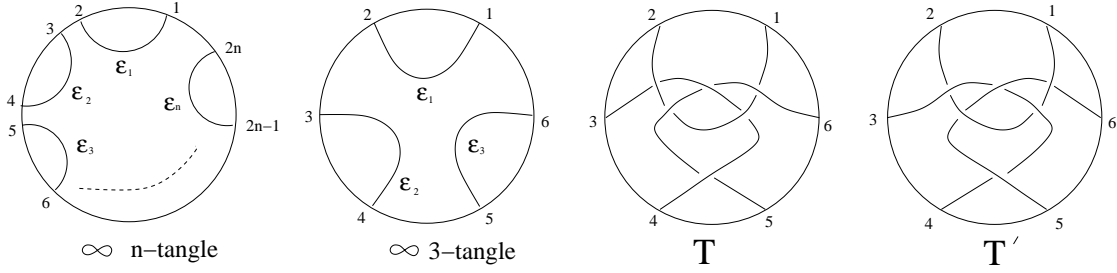}
\caption{Examples of rational 3-tangles}
\label{A1}
\end{figure}

Let $\Sigma_{0,6}$ be the six punctured sphere and let $\epsilon=\epsilon_1\cup\epsilon_2\cup\epsilon_3$ be the $\infty$ 3-tangle as in Figure~\ref{A2}.
Then, for two orientation preserving homeomorphisms $f$ and $g$ from $\Sigma_{0,n}$ to $\Sigma_{0,n}$,
we say they are $\textit{isotopic}$, denoted by $f\sim g$,  if there is a continuous map $H: \Sigma_{0,n}\times I\rightarrow
\Sigma_{0,n}$ so that $H(x,1)=f(x)$ and $H(x,0)=g(x)$ and $h_t(x)=H(x,t)$ is a homeomorphism for all $t$.
Also, we say that a subset $C_1$ of $\Sigma_{0,n}$ is $\textit{isotopic}$ to $C_2$, denoted by $C_1\sim C_2$, if there is a homeomorphism $h$ of $\Sigma_{0,6}$ with $h(C_1)=C_2$ such that
$h\sim id$.\\

To demonstrate the effectiveness of this algorithm, we will show that the rational $3$-tangles in Figure~\ref{A1} are not isotopic to each other.
Note that for every string of $T$, if we consider the other two strings then they are isotopic to a trivial rational 2-tangle in $B^3$. However, we will show that $T$ is not isotopic to the $\infty$ tangle. So, $T$ is similar to the Borromean rings. We will also show that $T$ is not isotopic to the tangle $T'$ which is obtained from $T$ by reversing  all the crossings in $T$. \\

The algorithm is based on the following facts, which are proved in Section 2.\\

Up to isotopy, orientation preserving homeomorphisms $f$ and $g$ from $\Sigma_{0,6}$ to $\Sigma_{0,6}$ which fix the puncture 1 can be obtained
by four half Dehn twists $\sigma_1,\sigma_2,\sigma_3$ and $\sigma_4$  which are the generators of the braid group $\mathbb{B}_5$. (Refer to~\cite{2}.)\\

Then we can get  extensions $F,G:B^3\rightarrow B^3$ of $f$ and $g$ which fix the set of six points,  $\epsilon\cap \partial B^3$, setwise
giving  two rational 3-tangles $T_{F}:=F({\epsilon})$ and $T_{G}:=G({\epsilon})$.\\

 We will show that a tangle can be ``presented" by an element of $\mathbb{B}_5$ and our algorithm will decide whether or not two elements of $\mathbb{B}_5$ present equivalent tangles. We will discuss this in Section 2.\\

 We say that a disk $D$ is $\textit{essential}$ in $B^3-T$ for a rational 3-tangle $T$ if
$D$ is a properly embedded disk in $B^3-T$ but it is not boundary parallel in $B^3-T$. \\

\begin{figure}[htb]
\begin{center}
\includegraphics[scale=.4]{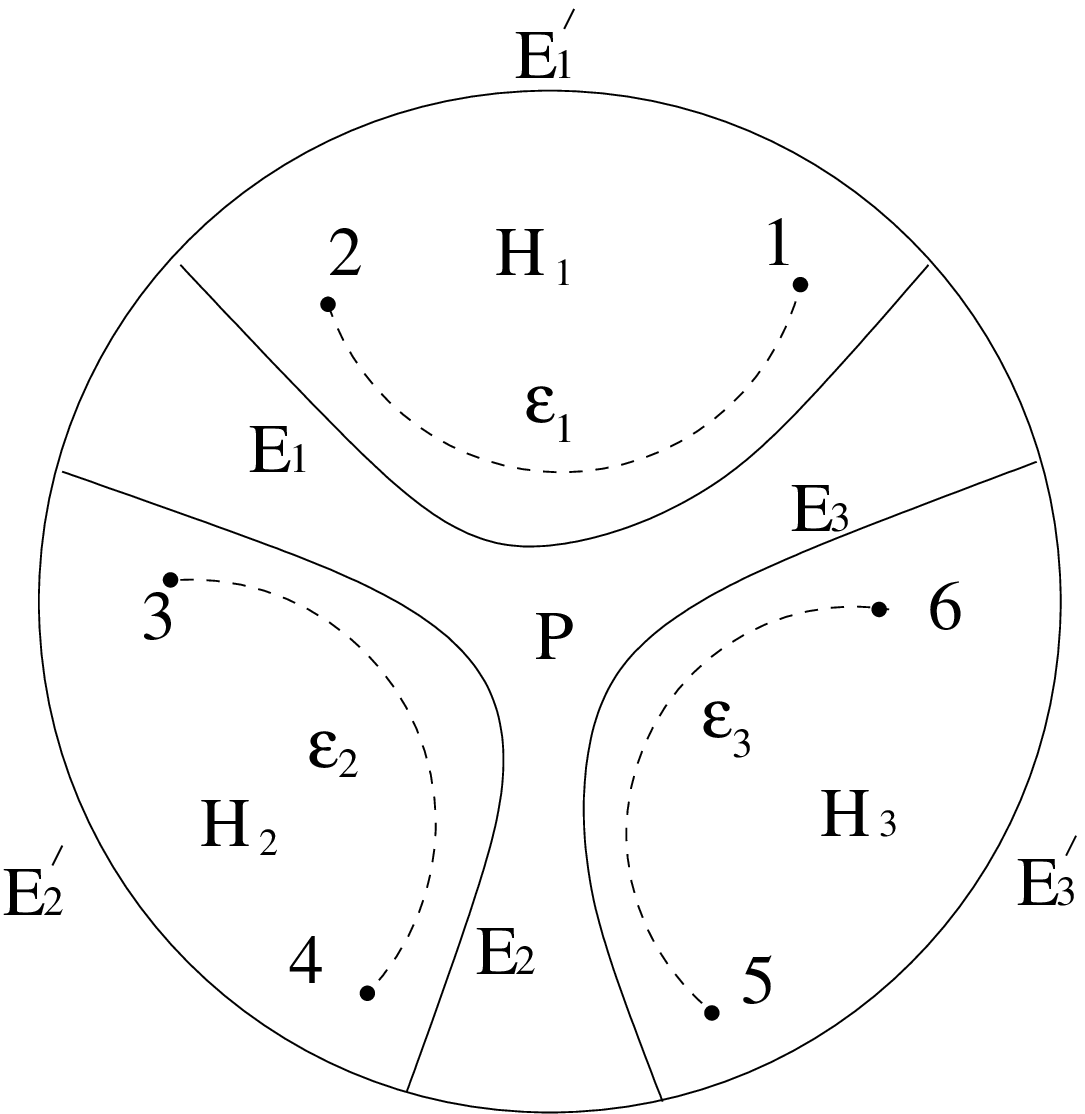}
\end{center}
\caption{A trivial rational 3-tangle, $\epsilon$}
\label{A2}
\end{figure}

$\bf{Theorem~2.3}$
\emph{For two rational 3-tangles $T_F$ and $T_G$, $T_F\approx T_G$ if and only if $G^{-1}F(\partial E)$ bounds  essential disks in $B^3-\epsilon$, where $E=E_1\cup E_2\cup E_3$ is a fixed union of ``standard essential disks" in $B^3-\epsilon$. (Refer to Figure~\ref{A2}.) }\\

In fact, if $G^{-1}F(\partial E_i)$ and $G^{-1}F(\partial E_j)$ bound essential disks in $B^3-\epsilon$ then $G^{-1}F(\partial E_k)$ also bounds an essential disk in $B^3-\epsilon$ for $\{i,j,k\}=\{1,2,3\}$.
Therefore,  if two of $G^{-1}F(\partial E)$ bound essential disks then $T_F\approx T_G$.\\

We note that there is another way to check whether $G^{-1}F(\partial E_i)$ bounds an essential disk in $B^3-\epsilon$ or not by using a fundamental group argument.
There is an induced map $i_*:\pi_1(\Sigma_{0,6})\rightarrow \pi_1(B^3-\epsilon)$ from the inclusion map $i:\Sigma_{0,6}\rightarrow B^3-\epsilon$. Then we know that if $i_*([G^{-1}F(\partial E_i)])=1$ then $G^{-1}F(\partial E_i)$ bounds an essential disk in $B^3-\epsilon$ by Dehn's Lemma.
This method is conceptually simple but it appears to be awkward to implement due to having to deal with arbitrarily long words in a free group.
 However, if one uses the algorithm given below to check whether or not $G^{-1}F(\partial E_i)$ bounds an essential disk in $B^3-\epsilon$  then we will be dealing with integer vectors of fixed dimension.
I will give an example to compare the two algorithms later. \\

Let $\mathcal{C}$ be the set of isotopy classes of closed essential simple closed curves in $\Sigma_{0,6}$. A simple closed curve $\gamma$ is \emph{essential} in  $\Sigma_{0,6}$ if $\gamma$ does not bound a disk in $\Sigma_{0,6}$ and $\gamma$ does not enclose a single puncture of $\Sigma_{0,6}$.\\

The algorithm is as follows:\\

Step 1: We represent $\Sigma_{0,6}$ as the union of two hexagons $H$ and $H^c$ so that the  vertices are the punctures of $\Sigma_{0,6}$. Then we use a variation of  normal curve theory to parameterize $\mathcal{C}$.
Each curve has a ``hexagon diagram". A set of ``weights" $w_{ij}$ and $w^{kl}$ for the hexagon diagram parameterizes the set of isotopy classes $[G^{-1}F(\partial E_s)]$.
Using certain formulas the weights can be obtained easily from the words in $\sigma_1,\cdot\cdot\cdot ,\sigma_4$ which describe $F$ and $G$, but the weights are hard to use directly to decide whether or not $G^{-1}F(\partial E_s)$ bounds an essential disk.\\

Step 2: We find a simple closed curve $\gamma'$, possibly not isotopic to $\gamma$, which bounds an essential disk in $B^3-\epsilon$ if and only if the component $\gamma$ of  $G^{-1}F(\partial E_i)$ does. We take a decomposition of  $\Sigma_{0,6}$ into three $2$-punctured disks $E_i'$ and one pair of pants $I$, where each $2$-punctured disk contains one component of $H\cap H^c$. We specify the isotopy class $[\gamma']$ by using a modified version of Dehn's method. (See [5].) We define the Dehn parameters $p_i,q_i$ and $t_i$ ($1\leq i\leq 3$) of $[\gamma']$ in $E_i'$ and the weights $x_{jk}$ ($1\leq j,k\leq 3$) of $[\gamma']$ in $I$. The $x_{jk}$ are determined by $p_i,q_i$ and $t_i$. We note that the Dehn parameters $p_i,q_i$ and $t_i$ ($1\leq i\leq 3$) of $[\gamma']$ are obtained from the weights $w_{ij}$ and $w^{kl}$ for the hexagon diagram, where $i,j,k,l\in\{1,2,3,4,5,6\}$.\\

Step 3: We modify $\gamma'$ into $\gamma_0$, possibly not isotopic to $\gamma'$ or $\gamma$, which is in ``standard position" and  bounds an essential disk in $B^3-\epsilon$ if and only if $\gamma'$ does. Then we get ``standard weights'' $m_i\geq 0$ ($1\leq i\leq 11$) of $\gamma_0$ from the Dehn parameters. Standard position is slightly reminiscent of train track theory, but involves fewer diagrams.\\

Step 4: We define three homeomorphisms $\delta_1, \delta_2$ and $\delta_3$ so that $\gamma_0$ bounds an essential disk in $B^3-\epsilon$ if and only if both $\delta_1\delta_2^{-1}(\gamma_0)$ and $\delta_3(\gamma_0)$ bound essential disks in $B^3-\epsilon$. Then, we repeatedly  apply Theorem~\ref{T93} below  to check whether $\gamma_0$ bounds an essential disk in $B^3-\epsilon$, where $I'$ is certain regular heighborhood of $I$.\\

$\bold{Theorem~\ref{T93}}$
Suppose that $\gamma_0$ bounds an essential disk in $B^3-\epsilon$ and $\gamma_0$  is in standard position in $I'$ and $m_3>0$.
Then applying one of the homeomorphisms  $(\delta_1\delta_2^{-1})^{\pm 1}$ and $\delta_3^{\pm 1}$ reduces the sum of the $p_i$ for the image of $\gamma_0$.\\

Suppose that a simple closed curve $\gamma_0$ is in standard position and has $m_3>0$. If we can reduce the sum of the standard weights of $\gamma_0$ by using one of the four homeomorphisms then we take the new simple closed curve $\gamma_1$ which is obtained by applying  one of the four homeomorphisms. If not, then $\gamma_0$ does not bound an essential disk. If  $\gamma_1$  still has $m_3>0$, then we will go on. Suppose $m_3=0$.  Then $\gamma_1$   is isotopic to one of the $\partial E_k$ if $m_i=0$ for all $i$. It does not bound an essential disk in $B^3-\epsilon$ if $m_i\neq 0$ for some $i$. Since the sum of the standard weights is finite, the algorithm will end in a finite number of steps.\\

Recall that $\gamma_0$ bounds an essential disk in $B^3-\epsilon$ if and only if $G^{-1}F(\partial E_i)$ does.
So, the given procedures form an algorithm to classify rational 3-tangles.\\

The author would like to thank his advisor Robert Myers for his consistent encouragement and sharing his enlightening ideas on the foundations of this topic.

\section{Presentations of rational 3-tangles}

We recall that a rational 3-tangle $T$ can be arranged to be in general position with respect to the projection onto the flat disk $Q$ in the $xy$-plane bounded by the great circle $C$.
Then we will have a tangle diagram $TD$ of the rational 3-tangle $T$. Let $p$ be the number of crossings of the tangle $T$ in the diagram. \\
\begin{figure}[htb]
\begin{center}
\includegraphics[scale=.3]{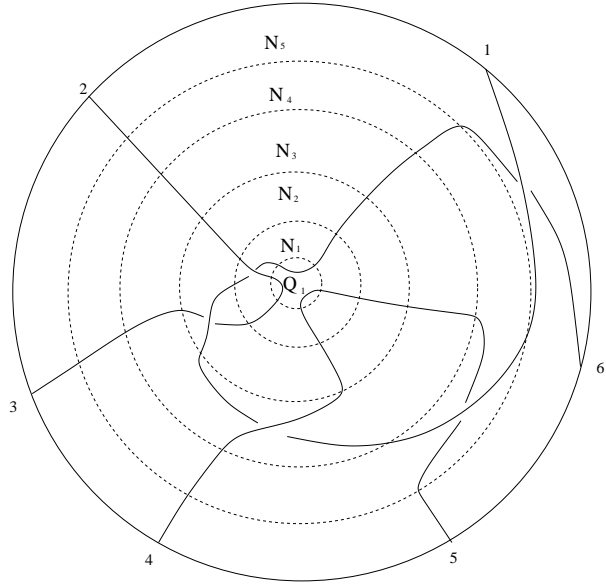}
\end{center}
\caption{A standard diagram of a rational 3-tangle expressed by  $w=\tau_5^{-1}\tau_4^{-1}\tau_3\tau_2\tau_1$.}
\label{B1}
\end{figure}

Now, we say that a tangle diagram $TD$ is $\emph{standard}$ if for the nested disks $Q_1\subset Q_2\subset\cdot\cdot\cdot\subset Q_{p+1}$, $Q_1$ contains the $\infty$ tangle and each annulus $N_j=Q_{j+1}-Q_{j}^\circ$ contains exactly one crossing of the crossings of $T$ as in Figure~\ref{B1}.\\

Then we define  a rational 3-tangle $T$ to be in $\emph{standard position}$ if the projection of $T$  onto the flat disk in the $xy$-plane bounded by $C$ is a standard diagram.\\

Let $\sigma_i$ be the half Dehn twist supported on the twice punctured disk $K_i$ as in Figure~\ref{B2}.\\

\begin{figure}[htb]
\includegraphics[scale=.4]{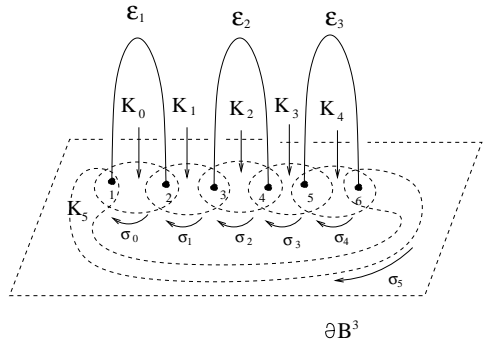}
\caption{Generators of the mapping class group of $\Sigma_{0,6}$}
\label{B2}
\end{figure}

Then we have an extension $\tau_i$ of $\sigma_i$ to $B^3$ as follows.\\

\begin{figure}[htb]
\begin{center}
\includegraphics[scale=.4]{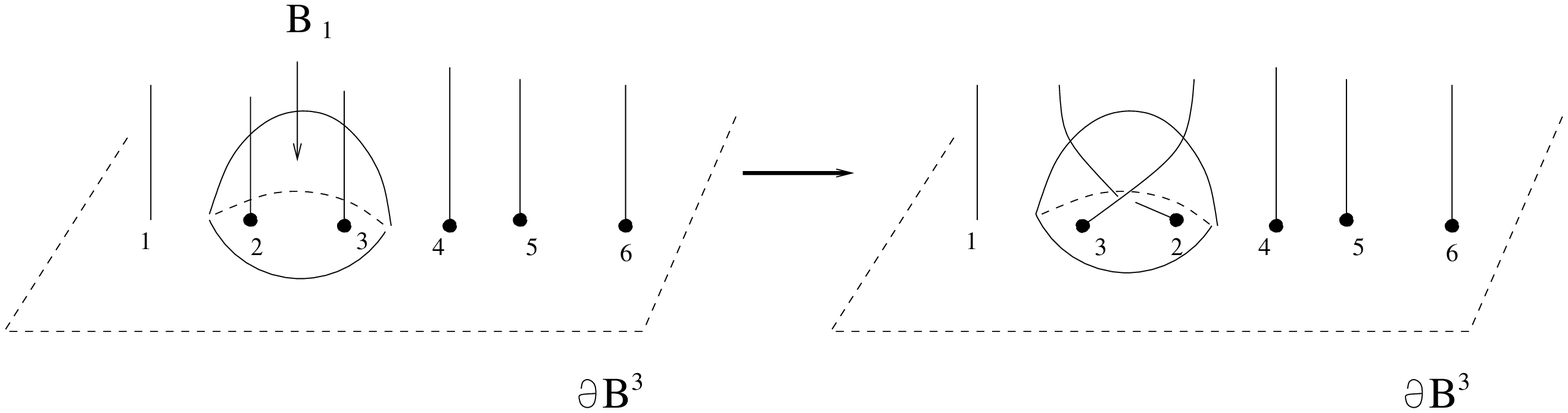}
\end{center}
\caption{The extension $\tau_i$ of a half Dehn twist $\sigma_i$ to $B^3$}
\label{B3}
\end{figure}
Take a  ball $B_i$ in $B^3$ so that $B_i\cap \partial B^3=\overline{K_i}$ and $B_i\cap \epsilon$ is two trivial  subarcs of  $\epsilon$ as in Figure~\ref{B3}.\\

Then, we define $\tau_i$ so that $\tau_i|_{B^3-B_i}=id$ and $\tau_i|_{B_i}$ is an extension of $\sigma_i$ to $B_i$ which twists the two trivial simple subarcs in $B_i$ to have a positive crossing  as in Figure~\ref{B3}. 

\begin{Lem}
\label{T21}
Suppose that $F$  is an orientation preserving homeomorphism from $B^3$ to $B^3$ so that $F(\{1,2,3,4,5,6\})=\{1,2,3,4,5,6\}$. Then, there exists an orientation preserving homeomorphism $F_1:B^3\rightarrow B^3$ so that  $F_1(1)=1$ and $F_1(\epsilon)=F(\epsilon)$.
\end{Lem}

\begin{proof}

First, we claim that there is a homeomorphism $R_i:(B^3,\epsilon)\rightarrow (B^3,\epsilon)$ so that $R_i(1)=i$ for $i\in\{1,2,3,4,5,6\}$.\\

We note that there is a homeomorphism $\rho_j:(B^3,\epsilon)\rightarrow(B^3,\epsilon)$  so that $\rho_j(1)=j$ for $j=1,3$ or $5$  by using $0^\circ$, $120^\circ$ or $240^\circ$ counterclockwise rotation in the plane. Then we remark that it preserves $\epsilon$ setwise. So, $R_i=\rho_i$ for $i=1,3,5$.
\\

We note that $\tau_k$ switches the two endpoints of $\epsilon_{1+{k\over 2}}$ for $k=0,2,4$.\\

  We let  $R_i:(B^3,\epsilon)\rightarrow (B^3,\epsilon)$ be $R_i=\tau_{i-2}\circ\rho_{i-1}$ for $i=2,4,6$. We check that $R_i(1)=i$ for $i=2,4,6$.\\

Suppose that $F(p)=1$.
Then we let $F_1=F\circ R_p$. So, we have $F_1(1)=F( R_p(1))=F(p)=1$.  Also, we know that $F_1(\epsilon)=F(R_p(\epsilon))=F(\epsilon)$ since $R_p(\epsilon)=\epsilon$.

\end{proof}

Now, by Lemma~\ref{T21}, we may assume $F(1)=1$.
 \emph{This will be assumed throughtout the rest of the paper.}\\

Let $f=F|_{\Sigma_{0,6}}$. Then it is an orientation preserving homeomorphism from $\Sigma_{0,6}$ to $\Sigma_{0,6}.$ 
The mapping class group of $\Sigma_{0,6}$ is  $\mathcal{MCG}(\Sigma_{0,6})=Homeo^+(\Sigma_{0,6})/\sim$.\\

Let $\mathcal{MCG}_{1}(\Sigma_{0,6})=\{[h]\in \mathcal{MCG}(\Sigma_{0,6})~|~h(1)=1\}$, where $[h]$ is the isotopy class of $h$. 
Then, $[f]\in\mathcal{MCG}_{1}(\Sigma_{0,6})$.\\

Recall the two rational 3-tangles $T$ and $T'$ in Figure\ref{A1} which can be arranged as standard diagrams.\\

In Figure~\ref{A1}, we see that $T\approx\tau_5\circ\tau_1\circ\tau_0^{-1}\circ\tau_3\circ\tau_1\circ\tau_5(\epsilon)$ and $T'\approx\tau_5^{-1}\circ\tau_1^{-1}\circ\tau_0\circ\tau_3^{-1}\circ\tau_1^{-1}\circ\tau_5^{-1}(\epsilon)$, where $\tau_i$ is an extension of $\sigma_i$ to $B^3$.\\

We say that a rational 3-tangle $F(\epsilon)$ is $\emph{presented}$ by an element of $\mathbb{B}_5$ if $F(\epsilon)$ is isotopic to $G(\epsilon)$ so that $G$ is the composition of a sequence  of extensions $\tau_i^{\pm 1}$  for $i\in\{1,2,3,4\}$. We note that the four generators of $\mathbb{B}_5$ are associated with the four isotopy classes $[\tau_i]$ ($1\leq i\leq 4)$. The later of this section, we will show that every rational 3-tangle can be presented by an element of $\mathbb{B}_5$. For example, $T\approx\tau_5\circ\tau_1\circ\tau_0^{-1}\circ\tau_3\circ\tau_1\circ\tau_5(\epsilon)\approx \tau_1\circ\tau_3\circ\tau_2^{-1}\circ\tau_1\circ\tau_2^{-1}\tau_1\circ\tau_2\circ\tau_3(\epsilon)$. \\

The following Lemma~\ref{T22} and Theorem~\ref{T23} appear as  Lemma 4.4.1 and Theorem 4.5 of~\cite{3}.

\begin{Lem}[Alexander~\cite{2}]\label{T22}
If  $g:D^n\rightarrow D^n$ is a homeomorphism from the unit $n$-ball to itself which fixes
the $(n-1)$-sphere $S^{n-1}=\partial D^n$ pointwise, then $g$ is isotopic to the identity under an isotopy
 which fixes $S^{n-1}$ pointwise. If $g(0)=0$, then the isotopy may be chosen to fix $0$.
\end{Lem}

\begin{Thm}[Birman~\cite{2}]\label{T23}
If $n\geq 2$, then $\mathcal{MCG}(\Sigma_{0,n})$ admits a presentation with generators $\sigma_0,\cdot\cdot\cdot,\sigma_4$.
\end{Thm}

\begin{Cor}\label{T24}
Suppose that $F$ is a homeomorphism of $B^3$ so that $F(\Sigma_{0,6})=\Sigma_{0,6}$. Then there exists a homeomorphism $G$ of $B^3$ so that  $G(\epsilon)\approx F(\epsilon)$ and $G$ is the composition of a sequence of extensions $\tau_i^{\pm 1}$ of $\sigma_i^{\pm 1}$ for $i\in\{0,1,2,3,4\}.$
\end{Cor}
\begin{proof}
By Theorem~\ref{T23}, $F|_{\Sigma_{0,6}}$ is isotopic to $g$ in $\Sigma_{0,6}$ which is the composition of a sequence of $\sigma_i^{\pm 1}$ for $i\in\{0,1,2,3,4\}$. Then, By Lemma~\ref{T22}, the extension $G$ of $g$ which is the composition of the sequence of $\tau_i^{\pm 1}$ is isotopic to $F$.
\end{proof}
\begin{Lem}\label{T25}
If two homeomorphisms $f$ and $g$ of $\Sigma_{0,6}$ are isotopic, then for any two extensions $\mathcal{F}$ and $\mathcal{G}$ of $f$ and $g$ to $B^3$, $\mathcal{F}(\epsilon)\approx\mathcal{G}(\epsilon)$.
\end{Lem}

\begin{proof}
First, take a collar $N(\partial B^3)=S^2\times[0,1]$ in $B^3$ so that $S^2\times\{0\}=\partial B^3$, $S^2\times\{1\}$ is a properly embedded sphere in $B^3$ and $\epsilon\cap N(\partial B^3)=\{1,2,3,4,5,6\}\times[0,1]$.\\

We note that there exists a homeomorphism $\phi:\Sigma_{0,6}\times[0,1]\rightarrow\Sigma_{0,6}\times[0,1]$ so that $\phi(x,0)=(f(x),0)$ and $\phi(x,1)=(g(x),1)$ since $f\sim g$. Then we define $\bar{\phi}:S^2\times[0,1]\rightarrow S^2\times[0,1]$ by filling in the six punctures of $\Sigma_{0,6}$ for each time $t$.\\

Let $\bar{f}$ and $\bar{g}$ be the extensions of $f$ and $g$ to $S^2$ by filling in the six punctures of $\Sigma_{0,6}$.\\

Also, we know that there exists a homeomorphism $\psi:\Sigma_{0,6}\times[0,1]\rightarrow\Sigma_{0,6}\times [0,1]$ so that $\psi(x,t)=(g(x),t)$ for all $t$.
Then we define $\bar{\psi}:S^2\times[0,1]\rightarrow S^2\times[0,1]$ by filling in the six punctures of $\Sigma_{0,6}$ for each time $t$.\\

Now, we define a homeomorphism $F:B^3\rightarrow B^3$ so that $F|_{S^2\times[0,1]}=\bar{\phi}$ and $F|_{B^3-(S^2\times[0,1])}$ is a homeomorphism of ${B^3-(S^2\times[0,1])}$ which  extends $\bar{\phi}|_{S^2\times \{1\}}$.\\

Also, we define a homeomorphism $G:B^3\rightarrow B^3$ so that $G|_{S^2\times[0,1]}=\bar{\psi}$ and $G|_{B^3-(S^2\times[0,1])}=F|_{B^3-(S^2\times[0,1])}$. 
We remark that $\bar{\phi}(x,1)=(\bar{g}(x),1)=\bar{\psi}(x,1).$\\

We see that $F(\epsilon)=G(\epsilon)$.\\

We remark that for any extension $\mathcal{F}$ of $f$ to $B^3$, $\mathcal{F}(\epsilon)\approx F(\epsilon)$ and any extension $\mathcal{G}$ of $g$ to $B^3$, $\mathcal{G}(\epsilon)\approx G(\epsilon)$ by Lemma~\ref{T22} since $F$ is the extension of $f$ to $B^3$ and $G$ is the extension of $g$ to $B^3$.\\

This implies that $\mathcal{F}(\epsilon)\approx\mathcal{G}(\epsilon)$ since $F(\epsilon)=G(\epsilon)$.

\end{proof}

\begin{Lem}\label{T26}
For a rational 3-tangle $F(\epsilon)$, there exists a rational 3-tangle $G(\epsilon)$ so that $G(\epsilon)\approx F(\epsilon)$ and $G(\epsilon)$  is in standard position.

\end{Lem}

\begin{proof}

By Theorem~\ref{T23},  there exists a homeomorphism $g$ of $\Sigma_{0,6}$ which is isotopic to $F|_{\Sigma_{0,6}}$ and $g$ is a composition of a sequence of $\sigma_i^{\pm 1}$ for $0\leq i\leq 4$. \\

Now, we construct an extension of $g$ to $B^3$ as follows:\\

Suppose that $g=\sigma_{j_1}^{\alpha_{1}}\sigma_{j_2}^{\alpha_{2}}\cdot\cdot\cdot\sigma_{j_m}^{\alpha_{m}}$ for some $\sigma_{j_k}\in\{\sigma_0,\sigma_1,\cdot\cdot\cdot,\sigma_5\}$ and integers $\alpha_{k}$. Let $p=|\alpha_{1}|+|\alpha_{2}|+\cdot\cdot\cdot+|\alpha_{m}|$.
 Now, consider the projection of $B^3$ onto the flat disk $Q$ in the $xy$-plane bounded by $C$ and having the $\infty$ tangle diagram in $Q_1\subset Q$.
 Then take nested disks $Q_2,\cdot\cdot\cdot Q_{p+1}$ so that $Q_1\subset Q_2\subset\cdot\cdot\cdot\subset Q_{p+1}$. Let $N_l=Q_{l+1}-Q_l$. \\
 
We know that the extension $\tau_{j_k}^{\pm 1}$ of $\sigma_{j_k}^{\pm 1}$ generates the crossing which may be in $N_p$. 
We note that the extension $\tau_i$ of a half Dehn tiwst $\sigma_{i}$ in Figure~\ref{B3} makes a positive crossing as in the last diagram of Figure~\ref{B3}. Then, we isotope the crossing into $N_1$. After this,  we generate the next crossing  by the extension of the next element either $\sigma_{j_k}^{\pm 1}$ or $\sigma_{j_{k-1}}^{\pm 1}.$ Then we  isotope the crossing into $N_2$ while we fix $Q_2$. By reading off the sequence of the composition from the right to the left and doing this procedure repeatedly, we can construct an extension $G$ of $g$ so that $G$ is in standard position.\\

Finally, by using Lemma~\ref{T22}, we complete the proof of this lemma.
\end{proof}

We say that a crossing in a standard diagram is $\emph{expressed}$ by an extension $\tau_i^{\pm 1}$ of $\sigma_i^{\pm 1}$ if the crossing is obtained by applying  $\tau_i^{\pm 1}$ as above.\\

\begin{figure}[htb]
\begin{center}
\includegraphics[scale=.32]{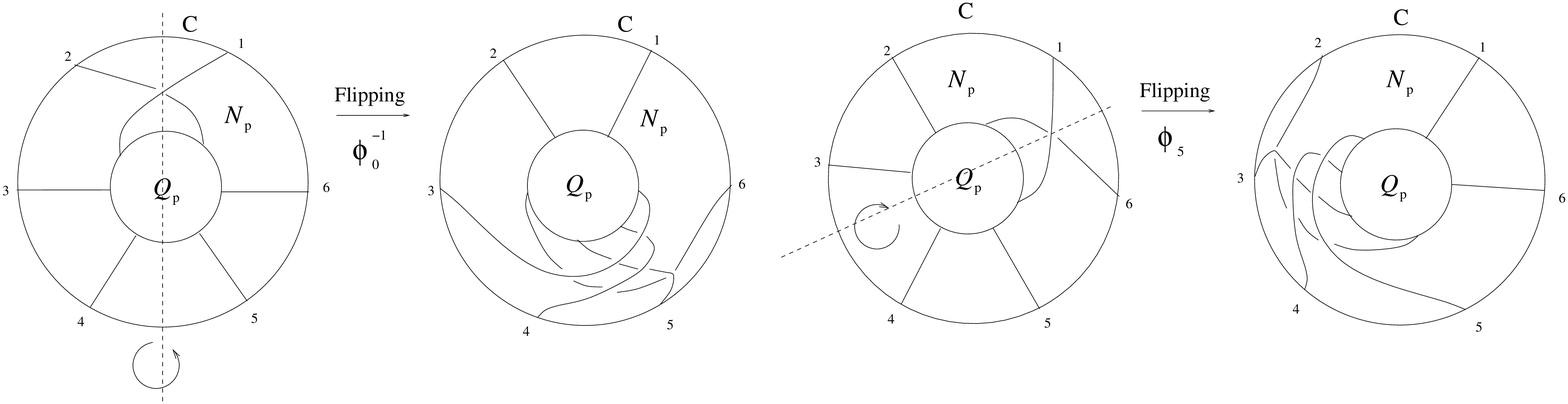}
\end{center}
\caption{Flippings}
\label{B4}
\end{figure}

Now, we will prove that every rational 3-tangle can be presented by an element of $\mathbb{B}_5$ with generators $\sigma_1,\sigma_2,\sigma_3$ and $\sigma_4$. So, our algorithm will decide whether or not two elements of $\mathbb{B}_5$ present equivalent tangles.

\begin{Lem}\label{T27}
Suppose that $G(\epsilon)$ is in standard position and the crossing in $N_p$ is expressed by $\tau_0^{\pm 1}$  as in the first diagram, or $\tau_5^{\pm 1}$ as in the third diagram  in Figure~\ref{B4}.
Then  $(\tau_2\tau_3\tau_2\tau_4\tau_3\tau_2)^{\pm 1}$ or $(\tau_1\tau_2\tau_1\tau_3\tau_2\tau_1)^{\pm 1}$  can replace  $\tau_0^{\pm 1}$ or $\tau_5^{\pm 1}$, respectively, so that the diagram of the new expression is still standard as in the second or fourth diagram in Figure ~\ref{B4}.\\

Especially, the number of crossings in $Q_p$ is fixed, where $Q_p$ is the disk inside of $C$ as in Figure~\ref{B4}.

\end{Lem}

\begin{proof}

Consider the dotted line which passes through the center of $C$ for each case as in Figure~\ref{B4}.\\

Flip the disk $Q_p$ about the dotted line to eliminate the crossing associated to $\tau_0^{\pm 1}$ or $\tau_5^{\pm 1}$.
Then, this procedure shows the lemma.

\end{proof}

\begin{Rem}\label{T28}
In Lemma~\ref{T27}, let $\phi_0^{-1}$ be the isotopy move to flip the disk $Q_p$ to eliminate the crossing associated to $\tau_0$ as in Figure~\ref{B4}.
Then let $\phi_0$ be the isotopy move to flip the disk $Q_p$ counter clockwise to eliminate the crossing associated to $\tau_0^{-1}$.
Similarly, let $\phi_5$ be the isotopy move to flip the disk $Q_p$ to eliminate the crossing associated to $\tau_5^{-1}$ as in Figure~\ref{B4}.
Also, let $\phi_5^{-1}$ be the isotopy move to flip the disk $Q_p$ clockwise to eliminate the crossing associated to $\tau_5$.\\

Suppose that $G=\tau_{j_1}^{\alpha_{1}}\tau_{j_2}^{\alpha_{2}}\cdot\cdot\cdot\tau_{j_m}^{\alpha_{m}}$  for some $\tau_{j_k}\in\{\sigma_0,\sigma_1,\cdot\cdot\cdot,\tau_5\}$ and integers $\alpha_{k}$ which expresses the crossings in $Q_{p+1}$.\\

Then, the crossings in $Q_p$ are expressed by $\tau_{j_1}^{\alpha_{1}}\tau_{j_2}^{\alpha_{2}}\cdot\cdot\cdot\tau_{j_m}^{\alpha_{m}\pm 1}$.\\

Then we note that the crossings of $\phi_0^{\pm 1}(Q_p)$ are expressed by $\tau_{\bar{j_1}}^{\alpha_{1}}\tau_{\bar{j_2}}^{\alpha_{2}}\cdot\cdot\cdot\tau_{\bar{j_m}}^{\alpha_{m}\pm 1} $,  where $\bar{j_i}\equiv -{j_i}$ (mod 6). \\

Similarly, we note that the crossings of $\phi_5^{\pm 1}(Q_p)$ are expressed by $\tau_{\bar{j_1}}^{\alpha_{1}}\tau_{\bar{j_2}}^{\alpha_{2}}\cdot\cdot\cdot\tau_{\bar{j_m}}^{\alpha_{m}\pm 1} $, where $\bar{j_i}\equiv 4-{j_i}$ (mod 6).

\end{Rem}

\begin{Thm}\label{T29}
A rational 3-tangle can be presented by an element of  $\mathbb{B}_5$.
\end{Thm}

\begin{proof}
First, we assume that a rational 3-tangle is in standard position. So, the projection onto the plat disk $Q$ in the $xy$-plane is a standard diagram. Let $p$ be the number of crossings in $Q$.\\

We remark that the two flippings in Figure~\ref{B4} will not change the tangle type in $Q_1$. i.e., $Q_1$ still contains the $\infty$ tangle after flippings.\\

Also, we know that  $\tau_0^{\pm 1}$ is replaced by  $(\tau_2\tau_3\tau_2\tau_4\tau_3\tau_2)^{\pm 1}$ and $\tau_5^{\pm 1}$ is replaced by $(\tau_1\tau_2\tau_1\tau_3\tau_2\tau_1)^{\pm 1}$
after flipping.\\

We note that the expression in terms of the crossings in $Q_{p}$ will be changed after flipping as in Remark 2.8, but the number of crossings in $Q_p$ is fixed.\\

If the  crossing  in $N_{p}$ is not expressed by either $\tau_0^{\pm 1}$ or $\tau_5^{\pm 1}$, then we consider the next  crossing  in $N_{p-1}$.\\

If the  crossing in $N_p$ is  expressed by either $\tau_0^{\pm 1}$ or $\tau_5^{\pm 1}$, then we flip $Q_p$ to eliminate the crossing associated to either $\tau_0^{\pm 1}$ or $\tau_5^{\pm 1}$.\\

Then, we also know that  the number of crossings in $Q$ of the original diagram is more than the number of crossings in $Q_p$.\\

We remark that  $(\tau_1\tau_2\tau_1\tau_3\tau_2\tau_1)^{\pm 1}$ and $(\tau_2\tau_3\tau_2\tau_4\tau_3\tau_2)^{\pm 1}$ do not contain  $\tau_0^{\pm 1}$ or $\tau_5^{\pm 1}$ factors.\\

By repeating this procedure, we can have another expression of $G(\epsilon)$ which involves only $\sigma_1^{\pm 1}$,...,$\sigma_4^{\pm 1}$.

\end{proof}

Now, I  give an example about Theorem~\ref{T29}.\\

\begin{figure}[htb]
\begin{center}
\includegraphics[scale=.32]{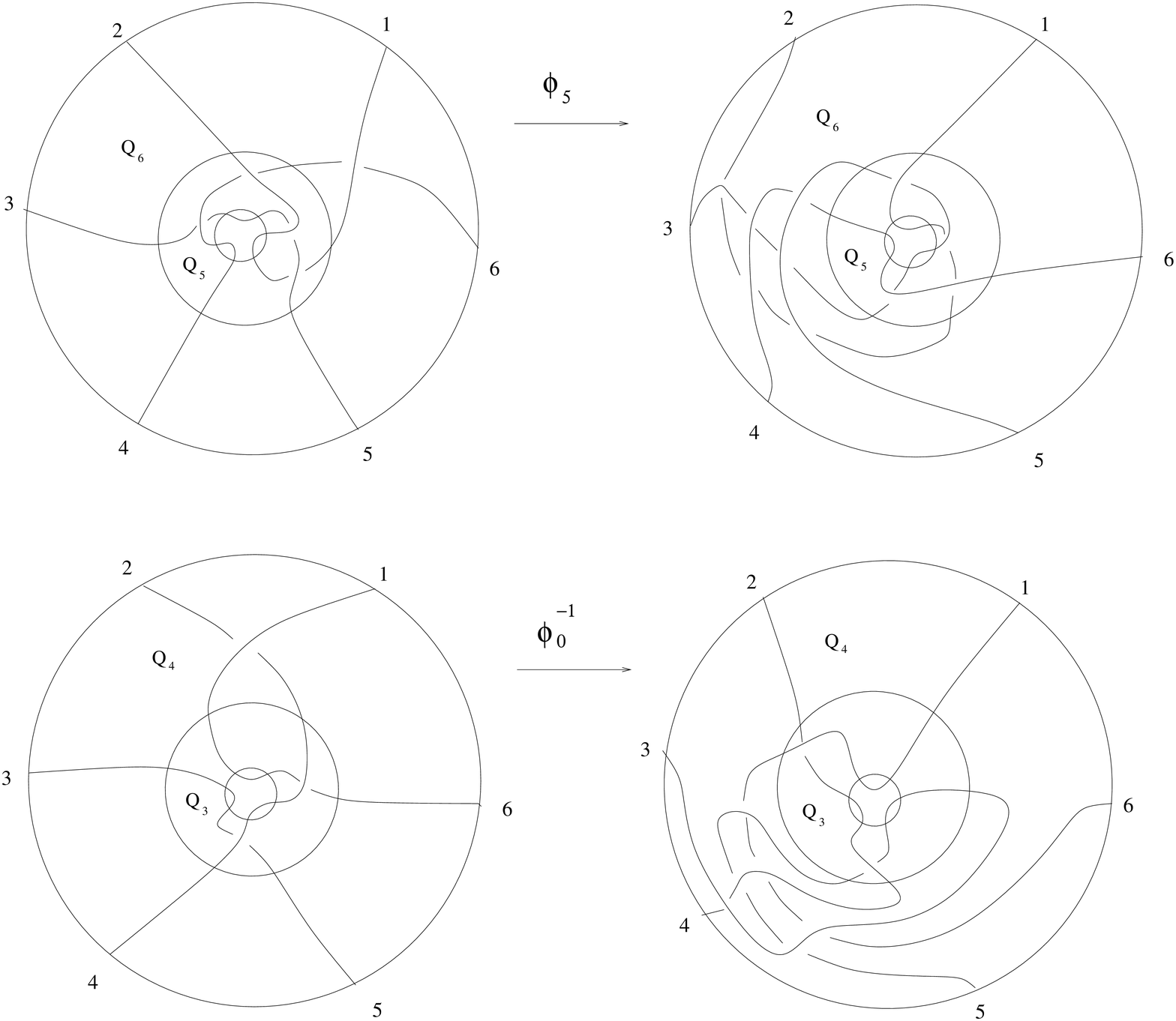}
\end{center}
\caption{A procedure to find a presentation which involves only $\sigma_1^{\pm 1},\cdot\cdot\cdot,\sigma_4^{\pm 1}$}
\label{B5}
\end{figure}

Consider the rational 3-tangle expressed by $w=\tau_5^{-1}\tau_0^{-1}\tau_4\tau_5^{-1}\tau_1$ as in the first diagram of Figure~\ref{B5}.\\

Flip the disk $Q_5$ to have a new expression $w_1=(\tau_1\tau_2\tau_1\tau_3\tau_2\tau_1)^{-1}\tau_{\bar{0}}^{-1}\tau_{\bar{4}}\tau_{\bar{5}}^{-1}\tau_{\bar{1}}$ as in Figure~\ref{B5}.\\

Then, we note that $w_1=(\tau_1\tau_2\tau_1\tau_3\tau_2\tau_1)^{-1}\tau_{\bar{0}}^{-1}\tau_{\bar{4}}\tau_{\bar{5}}^{-1}\tau_{\bar{1}}=
(\tau_1\tau_2\tau_1\tau_3\tau_2\tau_1)^{-1}\tau_{4}^{-1}\tau_{0}\tau_{5}^{-1}\tau_{3}$.\\

 So $w_1=(\tau_1\tau_2\tau_1\tau_3\tau_2\tau_1)^{-1}\tau_4^{-1}w_1'$, where $w_1'=\tau_{0}\tau_{5}^{-1}\tau_{3}$.\\
 
 We note that $Q_4$ contains the crossings which are expressed by $\tau_{0}\tau_{5}^{-1}\tau_{3}$ and $Q_3$ contains the crossings which are expressed by 
 $\tau_{5}^{-1}\tau_{3}$.\\

Now, flip the disk $Q_3$ to have a new expression $w_2'=(\tau_2\tau_3\tau_2\tau_4\tau_3\tau_2)\tau_{\bar{5}}^{-1}\tau_{\bar{3}}$ of $w_1'$. \\

We note that $w_2'=(\tau_2\tau_3\tau_2\tau_4\tau_3\tau_2)\tau_{\bar{5}}^{-1}\tau_{\bar{3}}=(\tau_2\tau_3\tau_2\tau_4\tau_3\tau_2)\tau_{1}^{-1}\tau_{3}.$ \\

Therefore, we have a new expression $w_3=(\tau_1\tau_2\tau_1\tau_3\tau_2\tau_1)^{-1}\tau_4^{-1}(\tau_2\tau_3\tau_2\tau_4\tau_3\tau_2)\tau_{1}^{-1}\tau_{3}$ of $w$ which involves only $\sigma_1^{\pm 1},\sigma_2^{\pm 1},\sigma_3^{\pm 1}$ and $\sigma_4^{\pm 1}$.

\section{ Equivalence of rational 3-tangles}

In this section, we will prove Theorem~\ref{T32} which tells us alternative method to decide whether or not two rational 3-tangles are isotopic.\\
 \begin{figure}[htb]
 \begin{center}
\includegraphics[scale=.4]{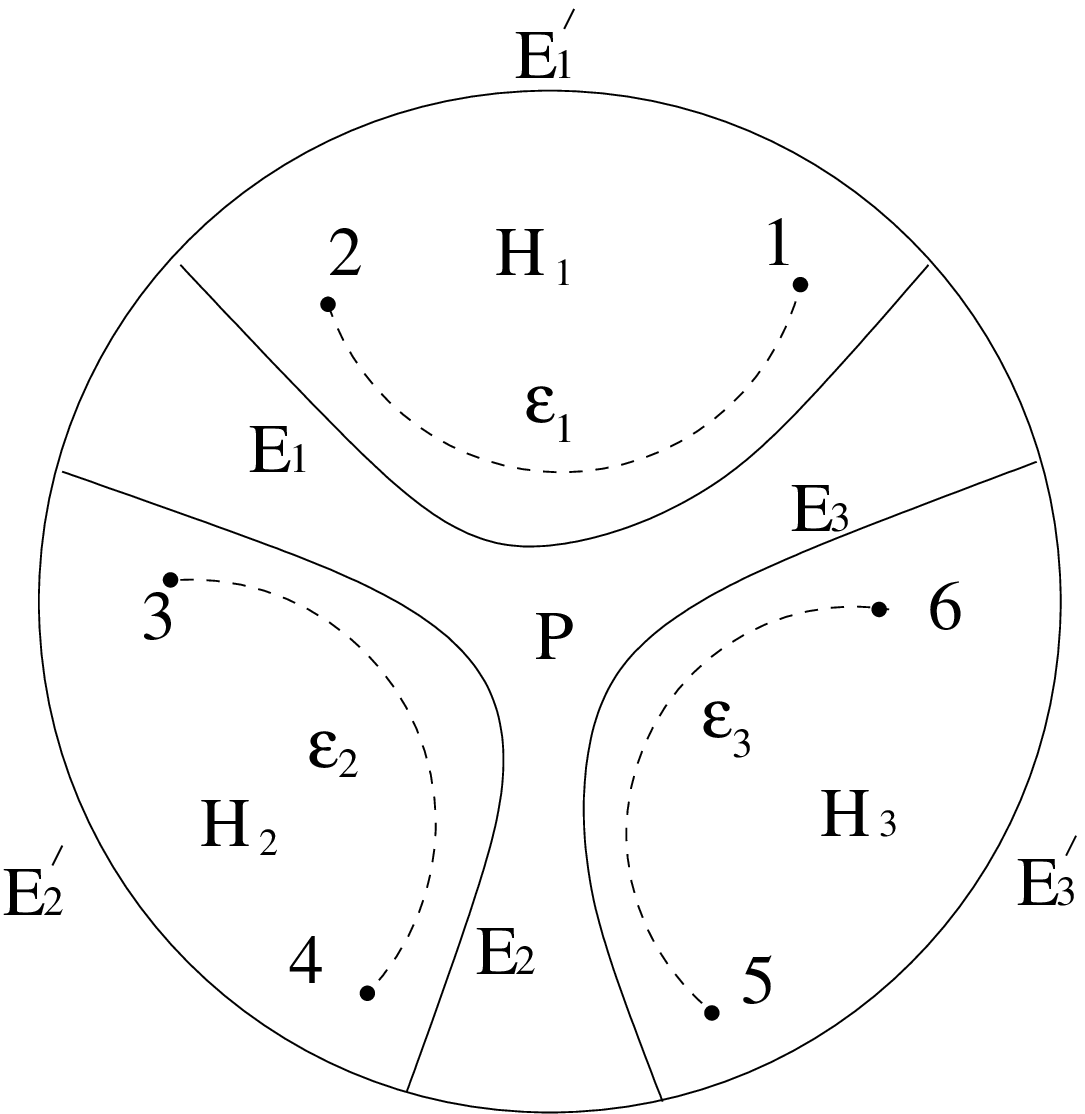}
\end{center}
\caption{ Three essential disks in the $\infty$  tangle}
\label{C1}
\end{figure}

Let $E_1$,$E_2$ and $E_3$ be the three disjoint essential disks as in Figure~\ref{C1}.
Then $E_1$,$E_2$ and $E_3$ separate $B^3$ into four components. Let $H_i$  be the component which contains $\epsilon_i$ and $P=cl(B^3-(H_1\cup H_2\cup H_3))$.\\

 Let $E_i'$ be the disk in $\partial B^3$ so that $\partial E_i'=\partial E_i$ and $E_i\cup E_i'$ bounds the ball $H_i$ in $B^3$.
Let  $E=E_1\cup E_2\cup E_3$,  $E'=E_1'\cup E_2'\cup E_3'$ and $\partial E=\partial E'=\partial E_1\cup \partial E_2\cup \partial E_3$.\\

We say that a properly embedded simple arc $C$ in $B^3$ is $\textit{unknotted}$ if there is an isotopy $\phi_t: B^3\rightarrow B^3$  that is identity on $\partial B^3$ so that $\phi_1(S)=C$, where $S$ is the straight line arc with the endpoints $\partial C$.\\
Then we can prove Lemma~\ref{T22} below.
\begin{Lem}\label{T31}
If $\alpha$ and $\alpha'$ are properly embedded unknotted simple arcs in $B^3$ with $\partial \alpha=\partial \alpha'\subseteq S^2$,
then $\alpha\approx\alpha'$.
\end{Lem}
\begin{proof}
Since $\alpha$ and $\alpha'$ are properly embedded unknotted simple arcs in $B^3$ with $\partial \alpha=\partial \alpha'$, $\alpha\sim \alpha_g\sim\alpha'$ for the straight line arc  $\alpha_g$ in $B^3$ from $a$ to $b$. (where $\partial \alpha=\{a,b\}.)$
We have a path $f_{\alpha}:I\rightarrow \alpha$ so that $f_{\alpha}(0)=a$ and $f_{\alpha}(1)=b$. Similarly, we also have paths $f_{\alpha'}$ and $f_{\alpha_g}$.
Let $H$ and $J$ be the isotopies from $B^3\times I$ to  $B^3$ so that $H(x,0)=f_{\alpha}(x)$ and $H(x,1)=f_{\alpha_g}(x)$, and $J(x,0)=f_{\alpha'}(x)$ and $J(x,1)=f_{\alpha_g}(x)$. Now, we define the isotopy $K:B^3\times I\rightarrow B^3$ so that $K(x,t)=H(x,2t)$ for $0\leq t\leq {1\over 2}$ and 
$K(x,t)=J(x,2-2t)$ for ${1\over 2}\leq t\leq 1$. Then $K$ is an isotopy from $\alpha$ to $\alpha'$ in $B^3$.

\end{proof}
Now consider orientation preserving homeomorphisms $f$ and $g$  from $\Sigma_{0,6}$ to $\Sigma_{0,6}$.
Then we have $F$ and $G$ which are  extensions to $B^3$  of $f$ and $g$ respectively.

\begin{Thm}\label{T32}
For two rational 3-tangles $T_F$ and $T_G$, $T_F\approx T_G$ if and only if $G^{-1}F(\partial E)$ bounds  essential disks in $B^3-\epsilon.$
\end{Thm}

\begin{proof}

($\Rightarrow)$ Suppose that there exists a homeomorphism $H$ from $(B^3,F(\epsilon))$ to $(B^3,G(\epsilon))$ so that $H|_{\partial B^3}=id|_{\partial B^3}$. Then we know that $HF(\partial E)=F(\partial E)$ since $H|_{\partial B^3}=id|_{\partial B^3}$.  Also, $G^{-1}HF(\epsilon)=\epsilon$ since $H(F(\epsilon))=G(\epsilon)$. Also, $G^{-1}HF(\partial E)=G^{-1}F(\partial E)$.
We claim that $G^{-1}HF(E)$ are essential disks in $B^3-\epsilon$. Since $E$ are essential disks in $B^3-\epsilon$, $F(E)$ are  essential disks in $B^3-F(\epsilon).$
Then $H(F(E))$ are  properly embedded disks in $B^3$ which are disjoint with $H(F(\epsilon))=G(\epsilon).$ Therefore, $H(F(E))$ are  essential disks in $B^3-G(\epsilon)$.
Finally, we know that $G^{-1}(H(F(E)))$ are  properly embedded disks in $B^3$ which are disjoint with $G^{-1}(G(\epsilon))=\epsilon.$
So, $G^{-1}HF(E)$ are  disks in $B^3-\epsilon$ and essential since each simple closed curve of $G^{-1}HF(\partial E)$ encloses two punctures in $\Sigma_{0,4}.$
 This implies that $G^{-1}F(\partial E)$ bound essential disks in $B^3-\epsilon.$ \\

($\Leftarrow$) Since $G^{-1}F(\partial E)$ bounds essential disks in $B^3-\epsilon,$  $F(\partial E)$ bounds essential disks in $B^3-G(\epsilon)$. Let $D_i$ be the properly embedded disk in $B^3-G(\epsilon)$ so that $\partial D_i=F(\partial E_i)$.
 We also know that $F(\partial E_i)$ bounds a disk $F(E_i')(=K_i)$ in $\partial B^3$ which contains two punctures.
 Then,  $F(E_i)\cup K_i$ bounds a ball $M_i$ in $B^3$ and $M_i$ contains $F(\epsilon_i)$.
Similarly, $D_i\cup K_i$ bounds a ball $N_i$ in $B^3$ so that $N_i$ contains $G(\epsilon_i)$.
Now, we can define a homeomorphism  $h_i(1\leq i\leq3)$ from $M_i$ to $N_i$ so that $h_i|_{F(E_i')}=id_{F(E_i')}$ and $h_i(F(\epsilon_i))=G(\epsilon_i)$  by using  Lemma~\ref{T22} and the Alexander trick. 
  Also, we can define $h_4$ from $B^3-(M_1\cup M_2\cup M_3)^\circ$ to $B^3-(N_1\cup N_2\cup N_3)^{\circ}$
   so that $h_4|_{\partial B^3-(K_1\cup K_2\cup K_3)}=id$ and $h_4|_{F(E_i)}=h_i|_{F(E_i)}$.
   Then we have a homeomorphism  $H$ from $B^3$ to $B^3$ so that $H|_{\partial B^3}=id$  and $H(F(\epsilon))=G(\epsilon)$.

 \end{proof}
 
 In fact, if two of $G^{-1}F(\partial E)$ bound essential disks in $B^3-\epsilon$ then $T_{F}\approx T_G$ by Lemma~\ref{T33} below.
 So, two disjoint non-parallel  simple closed curves which bound essential disks in $B^3-\epsilon$ determine the $\infty$ tangle.
 \begin{Lem}\label{T33}
 Suppose that two essential simple closed curves $\alpha,\beta~(\not\sim \alpha$) bound disjoint disks in $B^3-\epsilon$. If $\gamma$ is  an essential simple closed curve  which encloses two punctures, is disjoint with $\alpha$ and $\beta$ and is non-parallel to $\alpha$ and $\beta$, then $\gamma$ bounds an essential disk in $B^3-\epsilon$.
 \end{Lem}

 \begin{proof}
 Let $D_1$ and $D_2$ be the two disks in $B^3-\epsilon$ so that $\partial D_1=\alpha$ and $\partial D_2=\beta$.
Cut $B^3-\epsilon$ along the two disks. Then we have three balls $P_i$ which contains $\epsilon_i$.
Suppose that $\gamma\subset P_1$ without loss of generality. Then $\gamma$ divides $\partial P_1$ into two regions $Q$ and $R$. Assume that $Q$ contains the two punctures in $P_1$. To have a disk $D_3$ in $P_1$, push $R^\circ$ from $\partial P_1$ to the interior of $P_1$ a little bit.
 
 \end{proof}

 \section{Step 1: Hexagon prameterization of $\mathcal{C}$}

Recall $\mathcal{C}$ which is the set of isotopy classes of  essential simple closed curves in $\Sigma_{0,6}$.
In this section, we will describe how to parameterize $\mathcal{C}$ by using the hexagon diagram.  
To do this, we  define the hexagon as follows.\\

\begin{figure}[htb]
\begin{center}
\includegraphics[scale=.42]{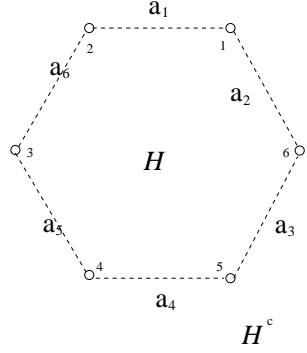}
\end{center}
\caption{Hexagon in $\Sigma_{0,6}$ }
\label{D1}
\end{figure}
Let $\partial\epsilon_1=\{1,2\}$, $\partial\epsilon_2=\{3,4\}$ and $\partial\epsilon_3=\{5,6\}$ as in Figure~\ref{D1}. 
By connecting the punctures in $\Sigma_{0,6}$ as in Figure~\ref{D1}, we can make  the hexagon $H$. Let $a_i$ be the dotted open intervals as in Figure~\ref{D1}. 
Then let $\bar{a_i}$ be the closed interval which is obtained from $a_i$ by adding the two punctures. For example, $\bar{a_1}=a_1\cup\{1,2\}$ and $\bar{a_3}=a_3\cup\{5,6\}$.\\

A family $C$ of smooth simple closed curves disjointly embedded in $\Sigma_{0,6}$ so that no component of $C$ is either null-homotopic or homotopic into a puncture is called a $\textit{multiple curve}$ in $\Sigma_{0,6}$; moreover, we require that two distinct components of $C$ cannot be isotopic to each other.
Define a $\textit{multicurve}$ in  $\Sigma_{0,6}$ to be the isotopy class of a multiple curve in $\Sigma_{0,6}$. Let $\Gamma$ be a graph so that the vertices of $\Gamma$ are  the punctures $1,2,3,4,5,6$ and the edges of $\Gamma$ are $\bar{a_i}$. Then, we define the $\emph{pseudo-graph}$ ${\Gamma}^{\circ}=\Gamma-\{1,2,3,4,5,6\}$. Then a multiple curve $\gamma$ is in $\textit{general position}$ with respect to $\Gamma^\circ$ in $\Sigma_{0,6}$  if $\gamma$ meets $\Gamma^\circ$ transversely.
Also, a multiple curve $\gamma$ is in $\textit{minimal general position}$ with respect to  $\Gamma^\circ$ in $\Sigma_{0,6}$ if $\gamma$ is in general position with respect to $\Gamma^\circ$ and $\gamma$ has a minimal number of intersections with $\Gamma^\circ$ up to isotopy.\\

Now, consider orientation preserving homeomorphisms $f=\sigma_1^{a_1}\sigma_2^{b_1}\sigma_3^{c_1}\sigma_4^{d_1}\cdot\cdot\cdot\sigma_1^{a_k}\sigma_2^{b_k}\sigma_3^{c_k}\sigma_4^{d_k}$ and $g=\sigma_1^{a_1'}\sigma_2^{b_1'}\sigma_3^{c_1'}\sigma_4^{d_1'}
\cdot\cdot\cdot\sigma_1^{a_m'}\sigma_2^{b_m'}\sigma_3^{c_m'}\sigma_4^{d_m'}$ from $\Sigma_{0,6}$ to $\Sigma_{0,6}$  
 for some
 $a_i,b_i,c_i,d_i,a_j',b_j',c_j',d_j'\in \mathbb{Z}$. 
Then by Theorem~\ref{T23}, $g^{-1}f(\partial E)=\sigma_4^{-d_m'}\sigma_3^{-c_m'}\sigma_2^{-b_m'}\sigma_1^{-a_m'}\cdot\cdot\cdot\sigma_4^{-d_1'}\sigma_3^{-c_1'}\sigma_2^{-b_1'}\sigma_1^{-a_1'}\sigma_1^{a_1}\sigma_2^{b_1}\sigma_3^{c_1}\sigma_4^{d_1}\cdot\cdot\cdot\sigma_1^{a_k}\sigma_2^{b_k}\sigma_3^{c_k}\sigma_4^{d_k}(\partial E)$ bounds essential disks in $B^3-\epsilon$ if and only if $T_F\approx T_G$, where $F$ and $G$ are extensions of $f$ and $g$ to $B^3$.\\

Let $\gamma=g^{-1}f(\partial E_p)$, where $p\in\{1,2,3\}$.
Then, we want to know  how each half Dehn twist $\sigma_j$  changes  $\gamma$ in $\Sigma_{0,6}$.
Assume that $\gamma$ is in minimal general position with respect to $\Gamma^\circ$. 
 Let $w_{ij}$ be the number of arcs of $\gamma$ which are from $a_i$ to $a_j$ in the hexagon.
Also, we  define $w^{kl}$ to be the number of arcs which are from $a_k$ to $a_l$ in the complement of the hexagon $H^c$. These are called $\textit{weights}$.
We notice that $w_{ij}=w_{ji}$ and $w^{kl}=w^{lk}$. Also, we know if $w_{ij}\neq 0$ for $i,j$ such that $i=j\pm 1$ (mod~6) then $w^{ij}=0$ and if $w^{kl}\neq 0$ for $k,l$ such that  
 $k=l\pm 1$ (mod~6) then $w_{kl}=0$.  If not, then we have a simple closed curve which is parallel to a puncture.
We notice that $w_{ii}=w^{ii}=0$ for all $i$ since $\gamma$ is in minimal general position with respect to $\Gamma^\circ$.\\

First, we will show that the weights $w_{ij}$ and $w^{ij}$  for the isotopy class $[\gamma]$  are well defined.\\

\begin{figure}[htb]
\begin{center}
\includegraphics[scale=.33]{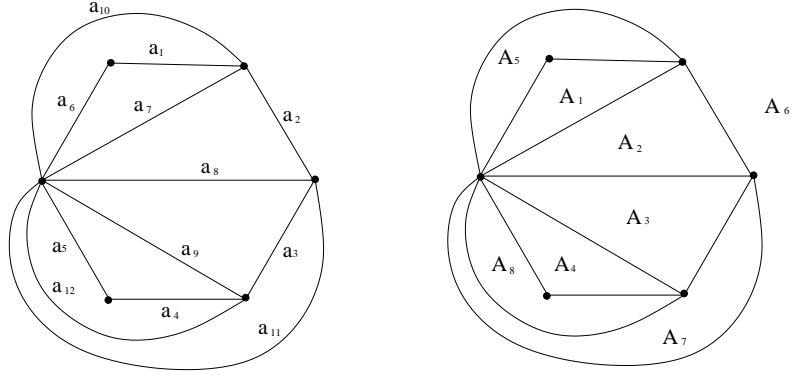}
\end{center}
\caption{A triangulation of the Hexagon diagram}
\label{D2}
\end{figure}

Let $a_7,a_8,...,a_{12}$ be the open arcs which connect two punctures as in Figure~\ref{D2}.
Let $\Gamma_+=\cup_{i=1}^{12}\bar{a_i}$ and $\Gamma^\circ_+=\Gamma_+-\{1,2,3,4,5,6\}$.
Then let $\Gamma_*$ be a subgraph of $\Gamma_+$. Then we define $\Gamma^\circ_*=\Gamma_*-\{1,2,3,4,5,6\}.$\\

 For two simple subarcs $\lambda$ and $\nu$ of a union $K$ of finitely many simple arcs  in a surface $\Sigma$, $(\Delta,\lambda,\nu)$ is a $bigon$ in the surface $\Sigma$ if $ \lambda\cup \nu$ bounds a disk $\Delta$ in $\Sigma$ and $\partial \lambda=\partial \nu=\lambda\cap \nu$
and $(\mathrm{int}~\Delta)\cap K =\emptyset$. Then we say that two unions $A$ and $B$ of simple arcs in a surface $\Sigma$ have a $\emph{bigon}$ if there exist two simple subarcs $\lambda$ and $\nu$ in $A$ and $B$ respectively so that $(\Delta,\lambda,\nu)$ is a bigon.\\

Let $|A\cap B|$ be the number of  intersections between $A$ and $B$.\\

\begin{Lem} \label{T41}
Suppose that $\delta$ is a simple closed curve in  $\Sigma_{0,6}$ so that $\delta$  is in general position with respect to $\Gamma^\circ_*$, but $\delta\cap \Gamma^\circ_*$ is not minimal.
Then $\delta$ and $\Gamma^\circ_*$ have a bigon in $\Sigma_{0,6}$.
\end{Lem}

\begin{proof}
Let $\delta'$ be a simple closed curve in $\Sigma_{0,6}$ so that $\delta'\sim\delta$ and $\delta'$ is in minimal general position with respect to $\Gamma^{\circ}_*$.
Then by the transversality theorem we can choose an isotopy $H:S^1\times[0,1]\rightarrow \Sigma_{0,6}$ so that $H(S^1\times\{0\})=\delta$, $H(S^1\times\{1\})=\delta'$ and $H^{-1}(\Gamma^{\circ}_*)$ is a collection of 1-manifolds in $S^1\times[0,1]$.    Let  $m=|H^{-1}(\Gamma^{\circ}_+)\cap (S^1\times\{0\})|$ and $n=|H^{-1}(\Gamma^{\circ}_*)\cap (S^1\times\{1\})|$. Then we notice that $m>n$ since $\delta'\cap \Gamma^{\circ}_*$ is minimal in $\Sigma_{0,6}$, but $\delta\cap \Gamma^{\circ}_*$ is not minimal. Therefore, there exists a properly embedded arc $\alpha$ in $S^1\times[0,1]$ so that $\alpha$ is parallel to an arc $\beta$ of $S^1\times \{0\}$ and $H(\alpha)\subset \Gamma^{\circ}_*$. So, $\alpha\cup\beta$ bounds a disk $D$ in $S^1\times [0,1]$.
Let $c_1$ and $c_2$ be the common endpoints of $\alpha$ and $\beta$. Let $d_1=H(c_1)$ and $d_2=H(c_2)$.
Now, consider $H|_{D}$. Let  $[d_1,d_2]$ be the segment between $d_1$ and $d_2$ in $ \Gamma^{\circ}_*$. Now, we choose a homeomorphism $K:\alpha\rightarrow [d_1,d_2]$ with $K(c_1)=d_1,~K(c_2)=d_2$. Then we remark that $ K \simeq H|_{\alpha}$ $\mathrm{rel}$ $\{c_1,c_2\}$. We define $\bar{K}:\alpha\cup\beta\rightarrow \Sigma_{0,6}$ so that $\bar{K}(x)=K(x)$ for $x\in\alpha$ and $\bar{K}(y)=H(y)$ for $y\in \beta$.
So, $H|_{\alpha\cup\beta}\simeq \bar{K}|_{\alpha\cup\beta}$ rel $\{c_1\}$. Let $[d_1,d_2]=\alpha'$ and $H(\beta)=\beta'$. Let $p_1$ be a path from $H(c_1)$ to $H(c_2)$ along $\alpha'$ and let $p_2$ be a path from $H(c_2)$ to $H(c_1)$ along $\beta'$. Then $p_1\cdot p_2$ is a loop with base point $H(c_1)$. Then we notice that $p_1\cdot p_2$ is null-homotopic in $\Sigma_{0,6}$. Therefore, $\alpha'\cup \beta'$ bounds a disk $D'$ in $\Sigma_{0,6}$.
This implies that $\delta$ and $\Gamma^{\circ}_*$ have a bigon in $\Sigma_{0,6}$. 

\end{proof}

\begin{Cor}\label{T42}
If $\gamma$ and $\Gamma^\circ_+$ have no bigons then $\gamma$ is in minimal general position with respect to $\Gamma^{\circ}_+$. Moreover, $\gamma\cap a_i$ also has a minimal intersection for all $i\in\{1,2,...,12\}$.
\end{Cor} 

\begin{proof} From Lemma~\ref{T41}, we know that if $\gamma$ and $\Gamma^\circ_+$ have no bigons then $\gamma$ is in minimal general position with respect to $\Gamma^{\circ}_+$. Now, suppose that $\gamma\cap a_i$ does not have a minimal intersection. Then $\gamma$ and $a_i$ have a bigon in $\Sigma_{0,6}$ by Lemma~\ref{T41}. So, we have closed intervals $\lambda\subset a_i$ and $\nu\subset \gamma$ so that $\lambda\cup \nu$ bounds a disk $\Delta$ in $\Sigma_{0,6}$. This implies that $\gamma$ and $\Gamma^\circ_+$ have a bigion since $\nu$ is homotopic to $\lambda$.
This contradicts the fact that $\gamma\cap \Gamma^\circ_+$ has a minimal intersection. Therefore,  $\gamma\cap a_i$ is minimal for all $i\in\{1,2,...,12\}.$
\end{proof}

 Using Corollary~\ref{T42}, we will show the weights of isotopy classes are well defined.\\

Recall that $w_{ij}$ is the number of arcs of $\gamma$ which are from $a_i$ to $a_j$ in the hexagon $H$ and $w^{kl}$ is the number of arcs of $\gamma$ which are from $a_k$ to $a_l$ in the complement of the hexagon $H^c$.
\begin{Lem}\label{T43}
The weights $w_{ij}$ and $w^{ij}$ of  $[\gamma]$  for $i,j\in\{1,2,...,6\}$ are  well defined.
\end{Lem}

\begin{proof}
Suppose that $\gamma$ is in minimal general position with respect to $\Gamma^\circ_+$. 
Let $m_i$ be the number of intersections between $\gamma$ and $a_i$ for $i\in\{1,2,...,12\}$.
Let $A_i$ be the regions as in Figure~\ref{D2}.\\

For the three sides $a_1,a_6$ and $a_7$ of a region $A_1$, let $s_{16},s_{17}, s_{67}$ be the numbers of arcs from $a_{i}$ to $a_{j}$ in $A_1$, where $i,j\in\{1,6,7\}$.
Then we know that $s_{16}+s_{17}=m_1$, $s_{16}+s_{67}=m_6$ and $s_{67}+s_{17}=m_7$. By solving these equations for $s_{ij}$, we have
$\displaystyle{s_{16}={m_1+m_6-m_7\over 2}}$, $\displaystyle{s_{17}={m_1+m_7-m_6\over 2}}$ and $\displaystyle{s_{67}={m_6+m_7-m_1\over 2}}$. So, the weights in $A_1$ are determined by $m_1, ~m_6$ and $m_7$.
Similarly, the weights $t_{27},t_{28},t_{78}$ in $A_2$ are  determined by $m_2$, $m_7$ and $m_8$.\\

Since $\gamma$ is in mimimal general position with respect to $\Gamma^{\circ}_+$, $\gamma$ and 
$\Gamma^{\circ}_+$ have no bigon. By Corollary~\ref{T42}, we know that $m_k$ is unique for $k=1,2,...,12$. This implies that the weights $s_{ij}$ and $t_{kl}$ in $A_1$ and $A_2$ respectively are also unique since $m_k$ is unique.\\

 Now, consider $A_1\cup A_2$. Then for the four sides $a_{1},a_{2},a_{6}$ and $a_{8}$ of the rectange $A_1\cup A_2$, let $y_{pq}$ be the number of arcs from $a_p$ to $a_q$ in $A_1\cup A_2$, where $p,q\in\{1,2,6,8\}$. Then the weights $y_{12},y_{16},y_{18},y_{26},y_{28}$ and $y_{68}$ are determined by $\{s_{16},s_{17},s_{67},t_{27},t_{28},t_{78}\}$ as follows. $y_{12}=\min(s_{17},t_{27})$, $y_{16}=s_{16}$, $y_{18}=s_{17}-y_{12}=s_{17}-\min(s_{17},t_{27})$, $y_{26}=t_{27}-y_{12}=t_{27}-\min(s_{17},t_{27})$, $y_{28}=t_{28}$ and $y_{68}=t_{78}-y_{18}=t_{78}-s_{17}+y_{12}=t_{78}-s_{17}+\min(s_{17},t_{27})$.\\

For four sides $a_3,a_4,a_5$ and $a_8$ of $A_3\cup A_4$, let $z_{uv}$ be the number of arcs from $a_u$ to $a_v$ in $A_3\cup A_4$, where $u,v\in\{3,4,5,8\}$. 
Then we note that $z_{uv}$ are determined by the six weights in $A_3$ and $A_4$.\\

Now, consider $H=A_1\cup A_2\cup A_3\cup A_4$. Then we claim that the weights $w_{ij}$ in $H$ are determined by $y_{pq}$ and $z_{uv}$ as follows.\\

$w_{12}=y_{12}$, $w_{13}=\max(0,y_{18}+z_{38}-\max(z_{38},y_{18}+y_{28}))$, $w_{15}=\max(0,y_{18}+z_{58}-\max(z_{58},y_{18}+y_{68}))$, $w_{14}=y_{18}-w_{13}-w_{15}$, $w_{16}=y_{16}$; $w_{23}=\min(y_{28},z_{38})$, $w_{24}=\max(0,y_{48}+z_{28}-\max(z_{28},y_{48}+y_{38}))$, $w_{25}=y_{28}-w_{23}-w_{24}$, $w_{26}=y_{26}$;  $w_{34}=z_{34}$,
$w_{35}=z_{35}$, $w_{36}=z_{38}-w_{13}-w_{23}$; $w_{45}=z_{45}$, $w_{46}=\max(0,y_{48}+z_{68}-\max(z_{68},y_{48}+y_{58}))$; $w_{56}=\min(y_{68},z_{58})$.\\

In order to get the formula for $w_{13}$, we need to consider the two cases that $z_{38}\geq y_{18}+y_{28}$ and $z_{38}<y_{18}+y_{28}$.\\

If $z_{38}\geq y_{18}+y_{28}$, then we see that $w_{13}=y_{18}$.\\

If $z_{38}<g_{18}+y_{28}$, then we see that $w_{13}=\max(y_{18}-((y_{18}+y_{28})-z_{38}),0)=\max(z_{38}-y_{28},0)$.\\

By combining the two cases, we get $w_{13}=\max(0,y_{18}+z_{38}-\max(z_{38},y_{18}+y_{28}))$.\\

Similarly, we can get the formulas for $w_{15}, w_{24}$ and $w_{46}$.\\

Therefore, $w_{ij}$ of $[\gamma]$ for $i,j\in\{1,2,3,4,5,6\}$ are unique if $\gamma$ is in minimal general position.\\

By using symmetry, we also know $w^{kl}$ are determined by the weights in $A_{i}$ for $i=5,6,7,8$.\\

Therefore, $w_{jk}$ and $w^{jk}$ for $j,k\in\{1,2,...,6\}$  are determined by $m_i$ for $i=1,2,...,12$ and this proves the theorem.
\end{proof}

From Theorem~\ref{T42}, we define that  $w_{ij}$ and $w^{ij}$ are $\textit{the weights}$ for the $\textit{isotopy class}$ $[\gamma]$ in the \emph{hexagon} parameterization if $w_{ij}$ and $w^{ij}$ are the weights for a simple closed curve $\delta$ which is isotopic to $\gamma$ and has no bigons with the hexagon.
Now, we want to calculate the weight changes by a half Dehn twist. \\

\begin{figure}[htb]
\begin{center}
\includegraphics[scale=.4]{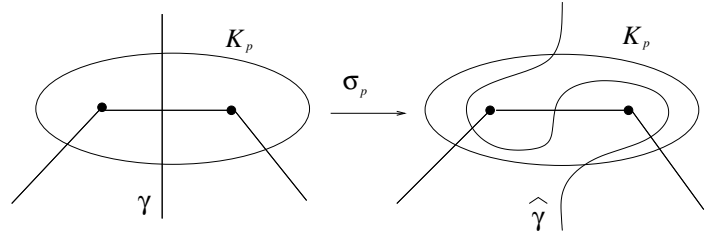}
\end{center}
\caption{The half Dehn twist supported on $K_P$}
\label{D3}
\end{figure}

First, we are getting a new curve $\hat{\gamma}$ as in Figure~\ref{D3} which is a representative of $[\sigma_p(\gamma)]$ that may have a bigon, and so the new weight $w'_{ii}$ may be non-zero. Then $\hat{\gamma}$ will be isotoped to remove all bigons and get the new weights $v_{ij}$ for $[\sigma_p(\gamma)]$.\\

\begin{figure}[htb]
\begin{center}
\includegraphics[scale=.40]{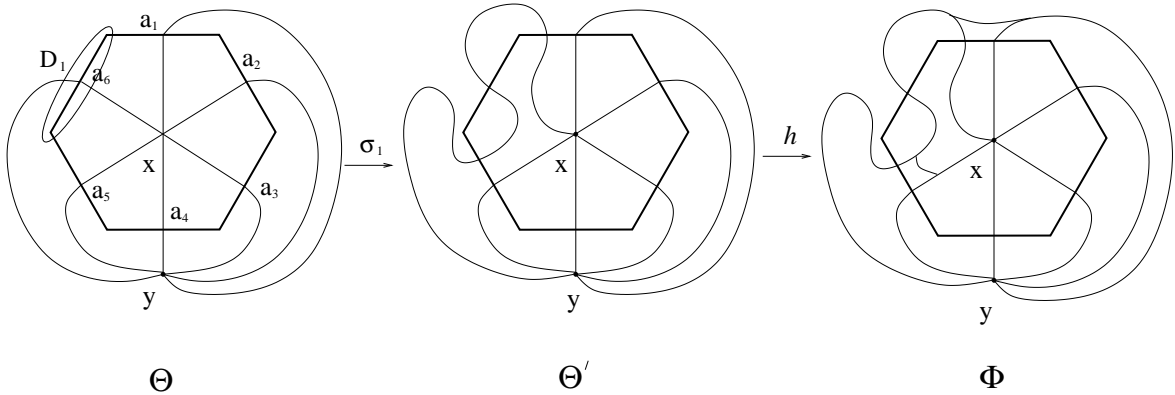}
\end{center}
\caption{The weight changes by the half Dehn twist $\sigma_1$}
\label{D4}
\end{figure}
Let  $w'_{ij}$ and $w'^{ij}$ be the weights for $\sigma_1(\gamma)$ as in the middle diagram of Figure~\ref{D4}. 

\begin{Thm}\label{T44}
Let $w_{ij}$ and $w^{ij}$ be the weights for $[\gamma]$. Then the following formulas give the weights $w'_{ij}$ and $w'^{ij}$ for $\sigma_1(\gamma)$.\\

$w_{12}'=w_{12}+w_{26}$,

$w_{13}'=w_{13}+w_{36}$,

  $w_{14}'=w_{14}+w_{46}$,
  
    $w_{15}'=w_{15}+w_{56}$, 
    
      $w_{16}'=0$;
      
      $w_{23}'=w_{23}$,
      
        $w_{24}'=w_{24}$,
        
          $w_{25}'=w_{25}$,
          
            $w_{26}'=0$;
            
    $w_{34}'=w_{34}$,
    
      $w_{35}'=w_{35}$,
      
         $w_{36}'=0$; 
         
       $w_{45}'=w_{45}$,
       
        $w_{46}'=0$;
        
    $w_{56}'=w_{16}+w_{26}+w_{36}+w_{46}+w_{56}$;
    
    $w'_{11}=w_{16}$;\\

$w'^{52}=w^{52}+w^{26}$, 

$w'^{53}=w^{53}+w^{36}$,

 $w'^{54}=w^{54}+w^{46}$,
 
   $w'^{51}=w^{51}+w^{16}$,
   
     $w'^{56}=0$; 
     
$w'^{23}=w^{23}$, 

$w'^{24}=w^{24}$,

 $w'^{21}=w^{21}$, 
 
  $w'^{26}=0$; 
  
$w'^{34}=w^{34}$,

 $w'^{31}=w^{31}$,
 
   $w'^{36}=0$;
   
 $w'^{41}=w^{41}$,
 
   $w'^{46}=0$;
   
 $w'^{16}=w^{56}+w^{26}+w^{36}+w^{46}+w^{16}$;
 
 $w'^{55}=w^{56}$.

\end{Thm}

\begin{proof}
To see  weight changes by a half Dehn twist $\sigma_1$, consider Figure~\ref{D4}.

From the two points $x,y$, we have 12 arcs which connect one of the two points and the middle of $a_i$. 
Then this diagram shows all possibilities of the weights. For example, there are arcs from $a_1$ to $x$ and from $x$ to $a_2$. These two arcs show the possibilities for $w_{12}$. Let $\Theta$ be the graph with the two vertices and the six edges. Now, take a proper two punctured disk $D_1$ which contains $a_6$ so that every component of $D_1\cap \gamma$ is essential in $D_1$.
We note that $D_1$ contains punctures $2$ and $3$.  Now, apply a half Dehn twist $\sigma_1$ supported on $D_1$ counter clockwise to the first diagram to get the second diagram. Let $\Theta'$ be the graph which is obtained from $\Theta$ by $\sigma_1$.
 Let $w_{ij}'$ and $w'^{kl}$ be the weights  for $\sigma_1(\gamma)$. We point out that $\sigma_1(\gamma)$ is not isotoped to have minimal intersection with the hexagon when
 $w_{ij}'$ and $w'^{kl}$ are computed. That will happen when $v_{ij}$ and $v^{kl}$ are computed. The formulas above give the weights $w_{ij}'$ and $w'^{kl}$.

\end{proof}

We remark that  if we use the transposition $(1,5)$ then we can get the formulas for $w'^{ij}$ from the formulas for $w'_{ij}$. i.e., we switch the indices $1$ and $5$. For example,  we get $w'^{16}=w^{56}+w^{26}+w^{36}+w^{46}+w^{16}$ from $w_{56}'=w_{16}+w_{26}+w_{36}+w_{46}+w_{56}$.\\

We notice that if we have a subarc of $\sigma_1(\gamma)$ for $w'_{ii}$ or $w'^{ii}$ then we can isotope the subarc across $a_i$ so that eventually $w'_{ii}=w'^{ii}=0$. Let $\Phi$ be the graph which is obtained from $\Theta$ by the isotopy to have $w'_{ii}=w'^{ii}=0$. Then we have the following theorem.

\begin{Thm}\label{T45}
Let $w'_{ij}$ and $w'^{ij}$ be the weights for $\sigma_1(\gamma)$. Then the following formulas give the weights $v_{ij}$ and $v^{ij}$ for $[\sigma_1(\gamma)]$ which has $v_{ii}=v^{ii}=0$ for all $i\in\{1,2,3,4,5,6\}$.\\

$v_{12}=w'_{12}$, $v_{13}=w'_{13}$, $v_{14}=w'_{14}$,

 $v_{15}=\max(w'_{15}-w'^{55},0)$, 

 $v_{16}=\min(w'^{55},w'_{15})$;
   
    $v_{23}=w'_{23}$, $v_{24}=w'_{24}$,
    
     $v_{25}=\min(w'_{25},\max(w'_{15}+w'_{25}+w'_{35}+w'_{45}-w'^{55},0))$, 
    
    $v_{26}=\min(w'_{25},\max(w'^{55}-w'_{15}-w'_{45}-w'_{35},0))$;
  
   $v_{34}=w'_{34}$,
   
    $v_{35}=\min(w'_{35},\max(w'_{15}+w'_{35}+w'_{45}-w'^{55},0))$,
       
    $v_{36}=\min(w'_{35},\max(w'^{55}-w'_{15}-w'_{45},0))$; 
    
    $v_{45}=\min(w'_{45},\max(w'_{15}+w'_{45}-w'^{55},0))$,
    
     $v_{46}=\min(w'_{45},\max(w'^{55}-w'_{15},0)$;
     
$v_{56}=w'_{56}-(w'^{55})$.\\

$v^{25}=w'^{25}$, $v^{35}=w'^{35}$, $v^{45}=w'^{45}$,

 $v^{15}=\max(w'^{15}-w'_{11},0)$,

 $v^{56}=\min(w'_{11},w'^{15})$; 

$v^{23}=w'^{23}$, 

$v^{24}=w'^{24}$,

 $v^{12}=\min(w'^{12},\max(w'^{12}+w'^{13}+w'^{14}+w'^{15}-w'_{11},0))$, 

$v^{26}=\min(w'^{12},\max(w'_{11}-w'^{13}-w'^{14}-w'^{15},0))$; 

$v^{34}=w'^{34}$,

 $v^{13}=\min(w'^{13},\max(w'^{13}+w'^{14}+w'^{15}-w'_{11},0))$,

 $v^{36}=\min(w'^{13},\max(w'_{11}-w'^{14}-w'^{15},0))$;

 $v^{14}=\min(w'^{14},\max(w'^{14}+w'^{15}-w'_{11},0))$, 
 
 $v^{46}=\min(w'^{14},\max(w'_{11}-w'^{15},0))$;
 
$v^{16}=w'^{16}-(w'_{11})$.
\end{Thm}

\begin{proof}

\begin{figure}[htb]
\begin{center}
\includegraphics[scale=.4]{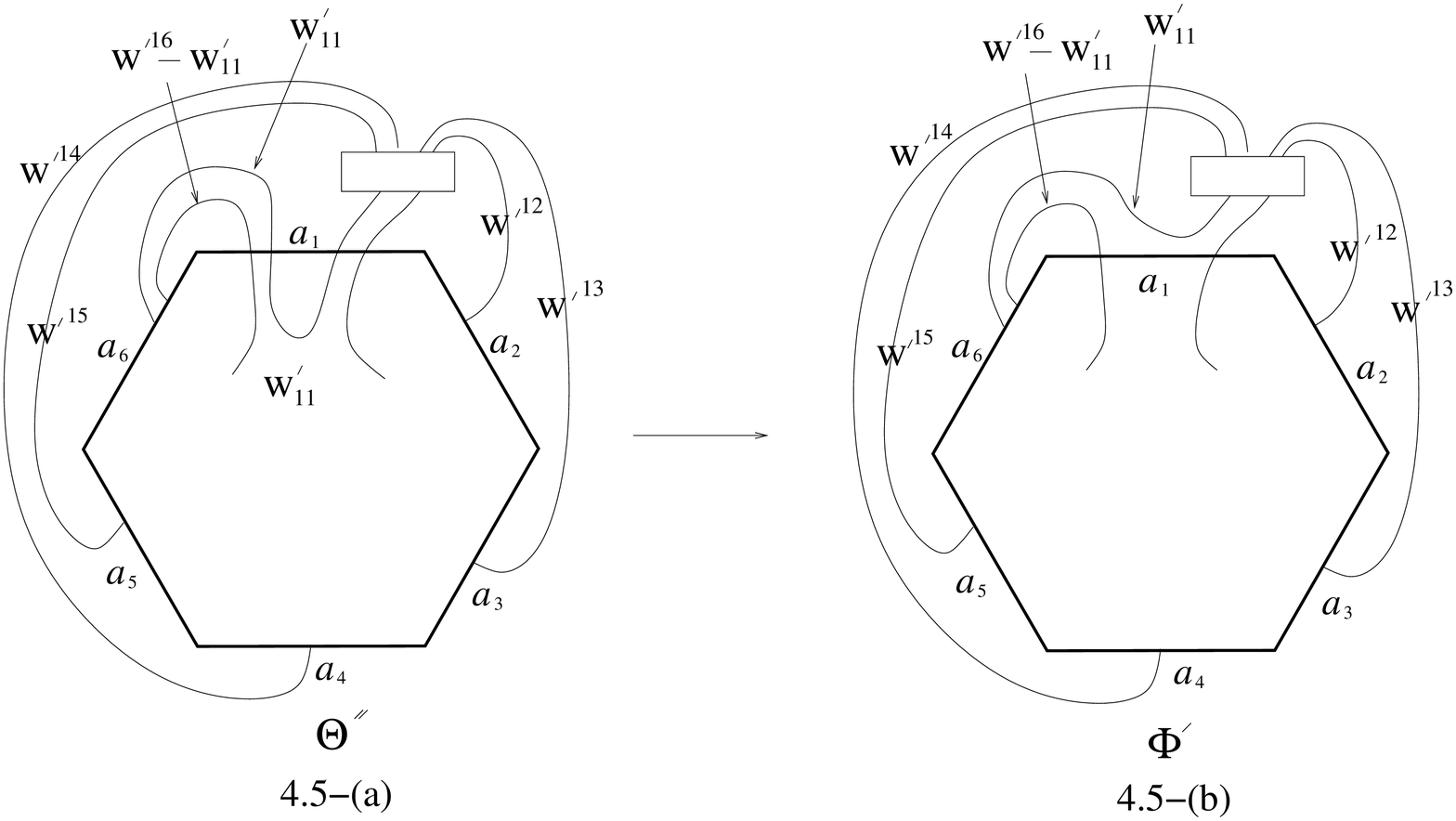}
\end{center}
\caption{The way to have $v_{ii}=0$}
\label{D5}
\end{figure}
Let $\Theta''$ be the union of arcs in the hexagon diagram as in Figure~\ref{D5}-(a) which shows the details of $\Theta'$. Similarly, let $\Phi'$ be the union of arcs in the hexagon diagram as in Figure~\ref{D5}-(b) which shows the details of $\Phi$.\\

 I want to remark that the arcs in $\Theta''$ and $\Phi'$ carry  $\sigma_1(\gamma)$.\\

From the diagram 13-(a), we obtain the diagram 13-(b) by pushing the arcs for $w'_{11}$ across $a_1$ so that we change $w'_{11}$ to $0$. I note that each subarc of $\sigma_1(\gamma)$ from $a_i$ to $a_j$ carries a weight. For example, in 13-(a), assume that the two arcs from $a_1$ to $a_6$ carries the weights $w'^{16}-w'_{11}$ and $w'_{11}$ respectively so that the sum of weights is $w'^{16}$.  
Let $h_1$ be the isotopy move to push the arcs that carring the weight $w'_{11}$  across $a_1$. Also, we know that $h_1\circ\sigma_1$ is isotopic to $\sigma_1$.
Let $h_1(\sigma_1(\gamma))=\gamma'$ and $w''_{ij}$ and $w''^{kl}$ be the weights for $\gamma'$. Then we have weight changes by $\sigma_1$ from $\gamma$ to
$\gamma'$ which has $w''_{11}=0$.\\

We note that $w''_{ij}=w'_{ij}$ for $i\neq j\in\{1,2,3,4,5,6\}$ and $w''_{ii}=0$ for all $i\in\{1,2,3,4,5,6\}$.\\

\begin{figure}[htb]
\begin{center}
\includegraphics[scale=.25]{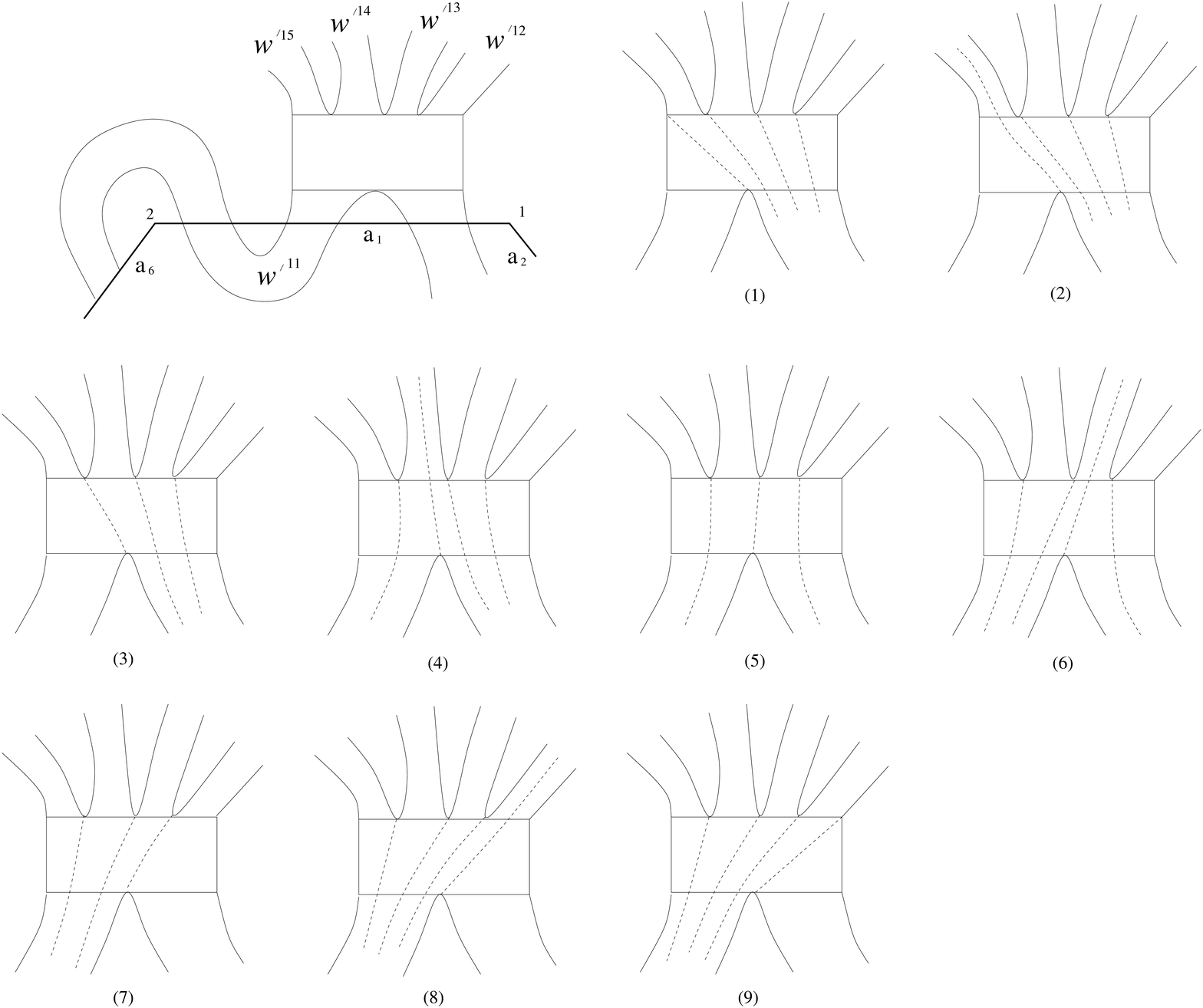}
\end{center}
\caption{Subcases for the weight changes from $w'^{11}$}
\label{D6}
\end{figure}

Now, consider the nine cases to find the formulas for $w''^{kl}$ as in Figure~\ref{D6}.\\

I want to emphasize that  the bands in the diagrams of Figure~\ref{D6} now carries the weight of $\gamma'$.\\

(1) $w_{11}'=0$ (2) $0<w_{11}'<w'^{15}$ (3) $w_{11}'=w'^{15}$ (4) $w'^{15}<w_{11}'<w'^{14}+w'^{15}$\\

 (5) $w_{11}'=w'^{14}+w'^{15}$
(6) $w'^{14}+w'^{15}<w_{11}'<w'^{13}+w'^{14}+w'^{15}$\\

 (7) $w_{11}'=w'^{13}+w'^{14}+w'^{15}$
(8) $w'^{13}+w'^{14}+w'^{15}<w_{11}'<w'^{12}+w'^{13}+w'^{15}+w'^{15}$\\

 (9) $w_{11}'=w'^{12}+w'^{13}+w'^{14}+w'^{15}$\\

Then, we have the following formulas for $w''^{kl}$.\\

$w''^{25}=w'^{25}$, $w''^{35}=w'^{35}$, $w''^{45}=w'^{45}$, \\

$w''^{23}=w'^{23}$, $w''^{24}=w'^{24}$,\\

$w''^{34}=w'^{34}$,\\

$w''^{16}=w'^{16}-(w'_{11})$.\\

Claim (a): $w''^{56}=\min(w'_{11},w'^{15})$.

\begin{proof}
We remind that $w'^{56}=0$.\\

 (i) If $w'_{11}<w'^{15}$ then by (1) and (2) we have $w''^{56}=w'_{11}$. \\

(ii) If $w'_{11}\geq w'^{15}$ then by $(3)-(9)$ we have $w''^{56}=w'^{15}$.
\end{proof}

 Claim (b): $w''^{15}=\max(w'^{15}-w'_{11},0)$.\\
 
 \begin{proof}
 (i) If  $w_{11}'<w'^{15}$ then by (1) and (2) we have $w''^{15}=w'^{15}-w_{11}'$.\\
   
   (ii) If $w_{11}'\geq w'^{15}$ then by $(3)-(9)$ we have $w''^{15}=0$.
   \end{proof}

 Claim (c): $w''^{26}=\min(w'^{12},\max(w'_{11}-w'^{13}-w'^{14}-w'^{15},0))$.

\begin{proof}
Recall that $w'^{26}=0$.\\

(i) If $w'_{11}\leq w'^{13}+w'^{14}+w'^{15}$  then by $(1)-(7)$ we have $w''^{26}=0$. Since $w'_{11}- w'^{13}-w'^{14}-w'^{15}\leq 0$,
 $\max(w'_{11}-w'^{15}-w'^{14}-w'^{13},0))=0$. So, $w''^{26}=\min(w'^{12},0)=0$.\\

(ii) If $w'^{13}+w'^{14}+w'^{15}<w'_{11}<w'^{12}+w'^{13}+w'^{14}+w'^{15}$ then by (8) we have $w''^{26}=w'_{11}-w'^{13}-w'^{14}-w'^{15}$. Since $w'_{11}-w'^{13}-w'^{14}-w'^{15}>0$, $\max(w'_{11}-w'^{13}-w'^{14}-w'^{15},0)=w'_{11}-w'^{13}-w'^{14}-w'^{15}$. Since $w'_{11}-w'^{13}-w'^{14}-w'^{15}<w'^{12}$, we have $w''^{26}=\min(w'^{12},w'_{11}-w'^{13}-w'^{14}-w'^{15})=w'_{11}-w'^{13}-w'^{14}-w'^{15}$. \\

 (iii) If $w'_{11}= w'^{12}+w'^{13}+w'^{14}+w'^{15}$ then by (9) we have $w''^{26}=w'^{12}$. Since $w_{11}'-w'^{13}-w'^{14}-w'^{15}=w'^{12}$, $\max(w'_{11}-w'^{13}-w'^{14}-w'^{15},0)= w'_{11}-w'^{13}-w'^{14}-w'^{15}=w'^{12}$. So, $\min(w'^{12},w'^{12})=w'^{12}$.
 \end{proof}

Claim (d): $w''^{12}=\min(w'^{12}, \max(w'^{12}+w'^{13}+w'^{14}+w'^{15}-w_{11}',0))$.

\begin{proof}
(i) If $w'^{13}+w'^{14}+w'^{15}> w_{11}'$ then by $(1)-(6)$ we have $w''^{12}=w'^{12}$. Since $w'^{13}+w'^{14}+w'^{15}-w_{11}'>0$, we have $\max(w'^{12}+w'^{13}+w'^{14}+w'^{15}-w_{11}',0)=w'^{12}+w'^{13}+w'^{14}+w'^{15}-w_{11}'$. So, $w''^{12}=\min(w'^{12},w'^{12}+w'^{13}+w'^{14}+w'^{15}-w_{11}')=w'^{12}$ since $w'^{13}+w'^{14}+w'^{15}> w_{11}'$. \\

(ii) If $w'^{13}+w'^{14}+w'^{15}\leq w_{11}'<w'^{12}+w'^{13}+w'^{14}+w'^{15}$ then by $(7)$ and $(8)$ we have $w''^{12}=w'^{12}+w'^{13}+w'^{14}+w'^{15}-w_{11}'$. Since $w'^{12}+w'^{13}+w'^{14}+w'^{15}-w_{11}'>0$, $\max(w'^{12}+w'^{13}+w'^{14}+w'^{15}-w_{11}',0)=w'^{12}+w'^{13}+w'^{14}+w'^{15}-w_{11}'$. So, $w''^{12}=\min(w'^{12},w'^{12}+w'^{13}+w'^{14}+w'^{15}-w_{11}')=w'^{12}+w'^{13}+w'^{14}+w'^{15}-w_{11}'$ since  $w'^{13}+w'^{14}+w'^{15}\leq w_{11}'$.\\

 (iii) If $w'^{12}+w'^{13}+w'^{14}+w'^{15}=w_{11}'$ then by $(9)$ we have $w''^{12}=0$. Since $w'^{12}+w'^{13}+w'^{14}+w'^{15}-w_{11}'=0$, we have $\max(w'^{12}+w'^{13}+w'^{14}+w'^{15}-w_{11}',0)=0$. So, $\min(w'^{12},0)=0$.
 \end{proof}

 Claim (e): $w''^{36}=\min(w'^{13},\max(w'_{11}-w'^{14}-w'^{15},0))$.
 
 \begin{proof}
 Recall that $w'^{36}=0$.\\
 
 (i) If $w'_{11}\leq w'^{15}+w'^{14}$ then by $(1)-(5)$ we have $w''^{36}=0$. Since $w'_{11}-w'^{14}-w'^{15}\leq 0$, $\max(w'_{11}-w'^{14}-w'^{15},0)=0$. So, $w''^{36}=\min(w'^{13},0)=0$. \\
  
 (ii) If $w'^{14}+w'^{15}<w'_{11}< w'^{13}+w'^{14}+w'^{15}$ then by (6) we have $w''^{36}=w'_{11}-w'^{14}-w'^{15}$. Since $w'_{11}-w'^{14}-w'^{15}>0$, we have $\max(w'_{11}-w'^{14}-w'^{15},0)= w'_{11}-w'^{14}-w'^{15}$. So, $w''^{36}=\min(w'^{13},w'_{11}-w'^{14}-w'^{15})=w'_{11}-w'^{14}-w'^{15}$ since $w'_{11}< w'^{13}+w'^{14}+w'^{15}$.\\
 
   (iii) If $w'_{11}\geq w'^{13}+w'^{14}+w'^{15}$ then by $(7)-(9)$ we have $w''^{36}=w'^{13}$. Since $w'_{11}-w'^{13}-w'^{14}-w'^{15}\geq 0$, $\max(w'_{11}-w'^{14}-w'^{15},0)=w'_{11}-w'^{14}-w'^{15}$. So, $w''^{36}=\min(w'^{13},w'_{11}-w'^{14}-w'^{15})=w'^{13}$ since $w'_{11}\geq w'^{13}+w'^{14}+w'^{15}$.
   \end{proof}
 
 Claim (f): $w''^{13}=\min(w'^{13},\max(w'^{13}+w'^{14}+w'^{15}-w_{11}',0))$.
\begin{proof}
(i) If  $w'^{14}+w'^{15}\geq w'_{11}$ then by $(1)-(5)$ we have $w''^{13}=w'^{13}$. Since $w'^{14}+w'^{15}- w'_{11}\geq 0$, $\max(w'^{13}+w'^{14}+w'^{15}-w_{11}',0)=w'^{13}+w'^{14}+w'^{15}-w_{11}'$. So, $w''^{13}=\min(w'^{13},w'^{13}+w'^{14}+w'^{15}-w_{11}')=w'^{13}$ since $w'^{14}+w'^{15}\geq w'_{11}$.\\

  (ii) If $w'^{14}+w'^{15}<w'_{11}\leq w'^{13}+w'^{14}+w'^{15}$ then by $(6)$ and $(7)$ we have $w''^{13}=w'^{13}+w'^{14}+w'^{15}-w_{11}'$. Since $w'^{13}+w'^{14}+w'^{15}- w_{11}'>0$, we have $\max(w'^{13}+w'^{14}+w'^{15}-w_{11}',0)=w'^{13}+w'^{14}+w'^{15}-w_{11}'$. So, $w''^{13}=\min(w'^{13},w'^{13}+w'^{14}+w'^{15}-w_{11}')=w'^{13}+w'^{14}+w'^{15}-w_{11}'$ since $w'^{14}+w'^{15}< w_{11}'$.\\
  
  (iii) If $w'^{13}+w'^{14}+w'^{15}<w_{11}'$ then by (8) and (9) we have $w''^{13}=0$.  Since $w'^{13}+w'^{14}+w'^{15}-w_{11}'<0$, $\max(w'^{13}+w'^{14}+w'^{15}-w_{11}',0)=0$. So, $w''^{13}=\min(w'^{13},0)=0$.\end{proof}

 Claim (g): $w''^{46}=\min(w'^{14},\max(w'_{11}-w'^{15},0)$.
 \begin{proof}
  (i) If  $w'_{11}\leq w'^{15}$ then by $(1)-(3)$ we have $w''^{46}=0$. Since $w'_{11}-w'^{15}\leq 0$, $\max(w'_{11}-w'^{15},0)=0$. So, $w''^{46}=\min(w'^{14},0)=0$. \\
  
  (ii) If $w'^{15}<w'_{11}<w'^{14}+w'^{15}$ then by (4) we have $w''^{14}=w'_{11}-w'^{15}$. Since $w'_{11}-w'^{15}>0$, $\max(w'_{11}-w'^{15},0)=w'_{11}-w'^{15}$. So,
  $w''^{46}=\min(w'^{14},w'_{11}-w'^{15})=w'_{11}-w'^{15}$ since $w'_{11}<w'^{14}+w'^{15}$.\\
  
  (iii) If $w'_{11}\geq w'^{14}+w'^{15}$
 then by $(5)-(9)$ we have $w''^{46}=w'^{14}$. Since $w'_{11}-w'^{14}-w'^{15}\geq 0$, we have $\max(w'_{11}-w'^{15},0)=w'_{11}-w'^{15}$. So,  $w''^{46}=\min(w'^{14},w'_{11}-w'^{15})=w'^{14}$ since $w'_{11}\geq w'^{14}+w'^{15}$.\end{proof}
 
  Claim (h): $w''^{14}=\min(w'^{14},\max(w'^{14}+w'^{15}-w'_{11},0))$.
  \begin{proof}
  (i) If $w'^{15}\geq w'_{11}$ then by $(1)-(3)$ we have $w''^{14}=w'^{14}$. Since 
  $w'^{15}-w'_{11}\geq 0$, we have $\max(w'^{14}+w'^{15}-w'_{11},0)=w'^{14}+w'^{15}-w'_{11}$. So, $w''^{14}=\min(w'^{14},w'^{14}+w'^{15}-w'_{11})=w'^{14}$ since $w'^{15}\geq w'_{11}$.\\
  
   (ii) If $w'^{15}<w'_{11}\leq w'^{14}+w'^{15}$ then by (4) and (5) we have $w''^{14}=w'^{14}+w'^{15}-w'_{11}$. Since $w'^{14}+w'^{15}-w'_{11}\geq 0$, we have 
   $\max(w'^{14}+w'^{15}-w'_{11},0)=w'^{14}+w'^{15}-w'_{11}$. So, $w''^{14}=\min(w'^{14},w'^{14}+w'^{15}-w'_{11})=w'^{14}+w'^{15}-w'_{11}$ since $w'^{15}<w'_{11}$.\\
   
   (iii) If $w'^{14}+w'^{15}< w'_{11}$ then by $(6)-(9)$ we have $w''^{14}=0$. Since  $w'^{14}+w'^{15}-w'_{11}<0$, we have $\max(w'^{14}+w'^{15}-w'_{11},0)=0$. So,  $w''^{14}=\min(w'^{14},0)=0$.
   \end{proof}

\begin{figure}[htb]
\begin{center}
\includegraphics[scale=.4]{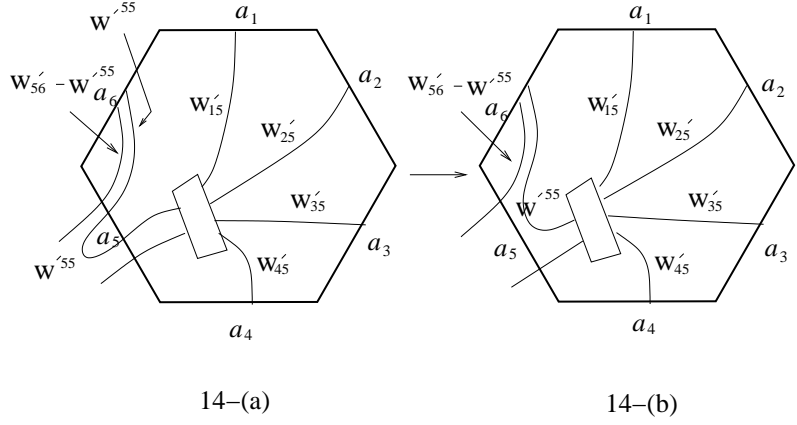}
\end{center}
\caption{The way to make $v^{ii}=0$}
\label{D7}
\end{figure}

Now, by pushing the arcs for $w'^{55}$ across  $a_5$  we  change  $w'^{55}$ to 0. (See Figure~\ref{D7}.)
Let $h_2$ be the isotopy move to push the arcs for $w'^{55}$ across  $a_5$.
We know that $h_2\circ h_1\circ\sigma_1$ is isotopic to $\sigma_1$. Let $h_2(h_1(\sigma_1(\gamma)))=\gamma''$ and $v_{ij}$ and $v^{kl}$ be the weights for $\gamma''$.
 We notice that
pushing the arcs for $w'_{55}$  does not depend on pushing the arcs for $w'^{11}$ since $h_1(\sigma_1(\gamma))$ does not change the weights $w'_{ij}$ if $i\neq j$.
So, by using the diagram which is obtained by switching $H$ and $H^c$, we  have  weight changes of $w'^{kl}$ by $\sigma$ from $\gamma$ to $\gamma''$ which has no bigons. (see Figure~\ref{D7}.) Actually, we have the formulas for $v_{ij}$ from $w''^{ij}$ by replacing $1$ by $5$ and $5$ by $1$.\\

$v_{12}=w'_{12}$, $v_{13}=w'_{13}$, $v_{14}=w'_{14}$,

 $v_{15}=\max(w'_{15}-w'^{55},0)$, 

 $v_{16}=\min(w'^{55},w'_{15})$;
   
    $v_{23}=w'_{23}$, $v_{24}=w'_{24}$,
    
     $v_{25}=\min(w'_{25},\max(w'_{15}+w'_{25}+w'_{35}+w'_{45}-w'^{55},0))$, 
    
    $v_{26}=\min(w'_{25},\max(w'^{55}-w'_{15}-w'_{45}-w'_{35},0))$;
  
   $v_{34}=w'_{34}$,
   
   $v_{35}=\min(w'_{35},\max(w'_{15}+w'_{35}+w'_{45}-w'^{55},0))$,
       
    $v_{36}=\min(w'_{35},\max(w'^{55}-w'_{15}-w'_{45},0))$; 
    
    $v_{45}=\min(w'_{45},\max(w'_{15}+w'_{45}-w'^{55},0))$,
    
     $v_{46}=\min(w'_{45},\max(w'^{55}-w'_{15},0)$;
     
$v_{56}=w'_{56}-(w'^{55})$.

 We  note that $v^{kl}=w''^{kl}$.
 
\end{proof}

Now, we have the following corollary by combining Theorem~\ref{T44} and Theorem~\ref{T45}.

\begin{Cor}\label{T46}
Let $w_{ij}$ and $w^{ij}$ be the weights for $[\gamma]$. Then the following formulas give the weights $v_{ij}$ and $v^{ij}$ for $[\sigma_1(\gamma)]$ which has $v_{ii}=v^{ii}=0$ for all $i\in\{1,2,3,4,5,6\}$.\\

$v_{12}=w_{12}+w_{26}$,

 $v_{13}=w_{13}+w_{36}$,
 
  $v_{14}=w_{14}+w_{46}$, 
  
  $v_{15}=\max(w_{15}+w_{56}-w^{56},0)$, 

 $v_{16}=\min(w^{56},w_{15}+w_{56})$;
   
    $v_{23}=w_{23}$,
    
     $v_{24}=w_{24}$, 
     
     $v_{25}=\min(w_{25},\max(w_{15}+w_{25}+w_{35}+w_{45}+w_{56}-w^{56},0))$, 
    
    $v_{26}=\min(w_{25},\max(w^{56}-w_{15}-w_{56}-w_{45}-w_{35},0))$;
  
   $v_{34}=w_{34}$,
   
    $v_{35}=\min(w_{35},\max(w_{15}+w_{35}+w_{45}+w_{56}-w^{56},0))$,
       
    $v_{36}=\min(w_{35},\max(w^{56}-w_{15}-w_{56}-w_{45},0))$; 
    
    $v_{45}=\min(w_{45},\max(w_{15}+w_{45}+w_{56}-w^{56},0))$,
    
     $v_{46}=\min(w_{45},\max(w^{56}-w_{15}-w_{56},0)$;
     
$v_{56}=w_{16}+w_{26}+w_{36}+w_{46}+w_{56}-(w^{56})$.\\

$v^{52}=w^{52}+w^{26}$,

 $v^{53}=w^{53}+w^{36}$,
 
  $v^{54}=w^{54}+w^{46}$,
  
   $v^{51}=\max(w^{51}+w^{16}-w_{16},0)$,

 $v^{56}=\min(w_{16},w^{51}+w^{16})$; 

$v^{23}=w^{23}$, 

$v^{24}=w^{24}$, 

$v^{12}=\min(w^{12},\max(w^{12}+w^{13}+w^{14}+w^{15}+w^{16}-w_{16},0))$, 

$v^{26}=\min(w^{12},\max(w_{16}-w^{13}-w^{14}-w^{15}-w^{16},0))$; 

$v^{34}=w^{34}$,

 $v^{13}=\min(w^{13},\max(w^{13}+w^{14}+w^{15}+w^{16}-w_{16},0))$,

 $v^{36}=\min(w^{13},\max(w_{16}-w^{14}-w^{15}-w^{16},0))$;

 $v^{14}=\min(w^{14},\max(w^{14}+w^{15}+w^{16}-w_{16},0))$,
 
  $v^{46}=\min(w^{14},\max(w_{16}-w^{15}-w^{16},0))$;
 
$v^{16}=w^{16}+w^{26}+w^{36}+w^{46}+w^{56}-(w_{16})$.

\end{Cor}

Also, we can calculate  weight changes which are affected by $\sigma_1^{-1}$ by using the symmetry as follows.
 $\partial H$ separates $\Sigma_{0,6}$ into two disks $H$ and $H^c$.
Now, we interchange $H$ and $H^c$ while fixing $\sigma_1(\gamma)$.
Then we can get the formulas for the weights $u_{ij}$ and $u^{kl}$ for $[\sigma_1^{-1}(\gamma)]$. In fact, if we switch the upper indices and lower indices, we get the
formulas for the weight changes by $\sigma_1^{-1}$. For example, we get $u^{12}=w^{12}+w^{26}$ from $v_{12}=w_{12}+w_{26}$.\\

Similarly, we  get  weight changes which are effected by $\sigma_i^{\pm 1}$ for $2\leq i\leq 4$ by using a multiple of $60^\circ$ rotation. 
For example, consider $\sigma_3$. First, rotate the hexagon diagram $-120^\circ$ (clockwise) about the center of the hexagon. Let $f$ be the rotation.
 Now apply $\sigma_1$ to $f(\gamma)$ to have the weights $u_{ij}$ and $u^{kl}$ for $[\sigma_1(f(\gamma))]$.
After this, rotate the resulting diagram $+120^\circ$ (couterclockwise) about the center of the hexagaon. Then we obtain the weight change formulas for $[\sigma_3(\gamma)]$.\\

We notice that the permutation $(123456)^{(6-2)}=(153)(264)$ gives the index changes for clockwise $120^\circ$ rotation. From example, $u_{12}=w_{56}$.
Then we can get the weight $v'_{ij}$ and $v'^{kl}$ for $[\sigma_1(f(\gamma))]$ from $u_{ij}$ and $u^{kl}$. Now, we switch the indices to have the weights $v_{ij}$ and $v^{kl}$ for $[\sigma_3(\gamma)]$
by using the permutation $(123456)^{(6-4)}=(135)(246)$ for counterclockwise $120^\circ$ rotation.\\

Now, we can calculate the weights of $G^{-1}F(\partial E_i)$ in the hexagion parameterization by using the formulas given in this section.
We need 30 parameters with integer entries to express a simple closed curve $\gamma$.
We notice that it is possible to have a very long length sequence of five generators of $\pi_1(\Sigma_{0,6})$ to express $\gamma$ if we use the fundamental group argument.
For example, consider a simple closed curve $\gamma$ which has weights $w^{16}=w^{45}=w_{14}=w_{56}=1, w^{15}=w_{15}=20001$ and all the other weights are zero. However, we need a $40004$ length sequence of five generators of $\pi_1(\Sigma_{0,6})$ to express $\gamma$.\\

Despite this benefit, it is difficult to know whether $\gamma$ bounds an essential disk or not from the hexagon prameterization.
So, we will use the \emph{Dehn} parameterization as the follows.

\section{Step 2-1:  Dehn parameterization of $\mathcal{C}$}

Let $\gamma$ be a simple closed curve in $\Sigma_{0,6}$. 
Consider the pair of pants $I:=\partial B^3-\{E_1'\cup E_2'\cup E_3'\}$. \\

\begin{figure}[htb]
\begin{center}
\includegraphics[scale=.32]{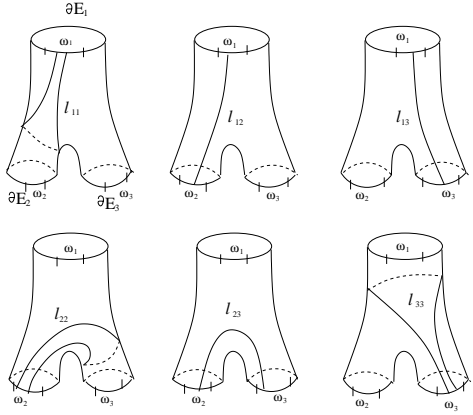}
\end{center}
\caption{Standard arcs $l_{ij}$}
\label{E1}
 \end{figure}

Figure~\ref{E1}  shows $\emph{standard arcs}$ $l_{ij}$ in the pair of pants $I$. We notice that we can isotope $\gamma$ into $\delta$ in $\Sigma_{0,6}$ so that each component of $\delta\cap I$ is isotopic to one of the standard arcs and $\delta\cap \partial E_i\subset \omega_i$. 
Then we say that subarc $\alpha$ of $\delta$ $\emph{is carried by}$ $l_{ij}$ if some component of $\alpha\cap I$ is isotopic to $l_{ij}$. The closed arc $\omega_i\subset \partial E_i$ is called a $\emph{window}$. \\

 Let $I_i=|\delta\cap \omega_i|$. 
Then $\delta$ can have many parallel arcs which are the same type in $I$. Let $x_{ij}$ be the number of parallel arcs of the type $l_{ij}$ which is called the $weight$ of $l_{ij}.$ \\

\begin{figure}[htb]
\begin{center}
\includegraphics[scale=.25]{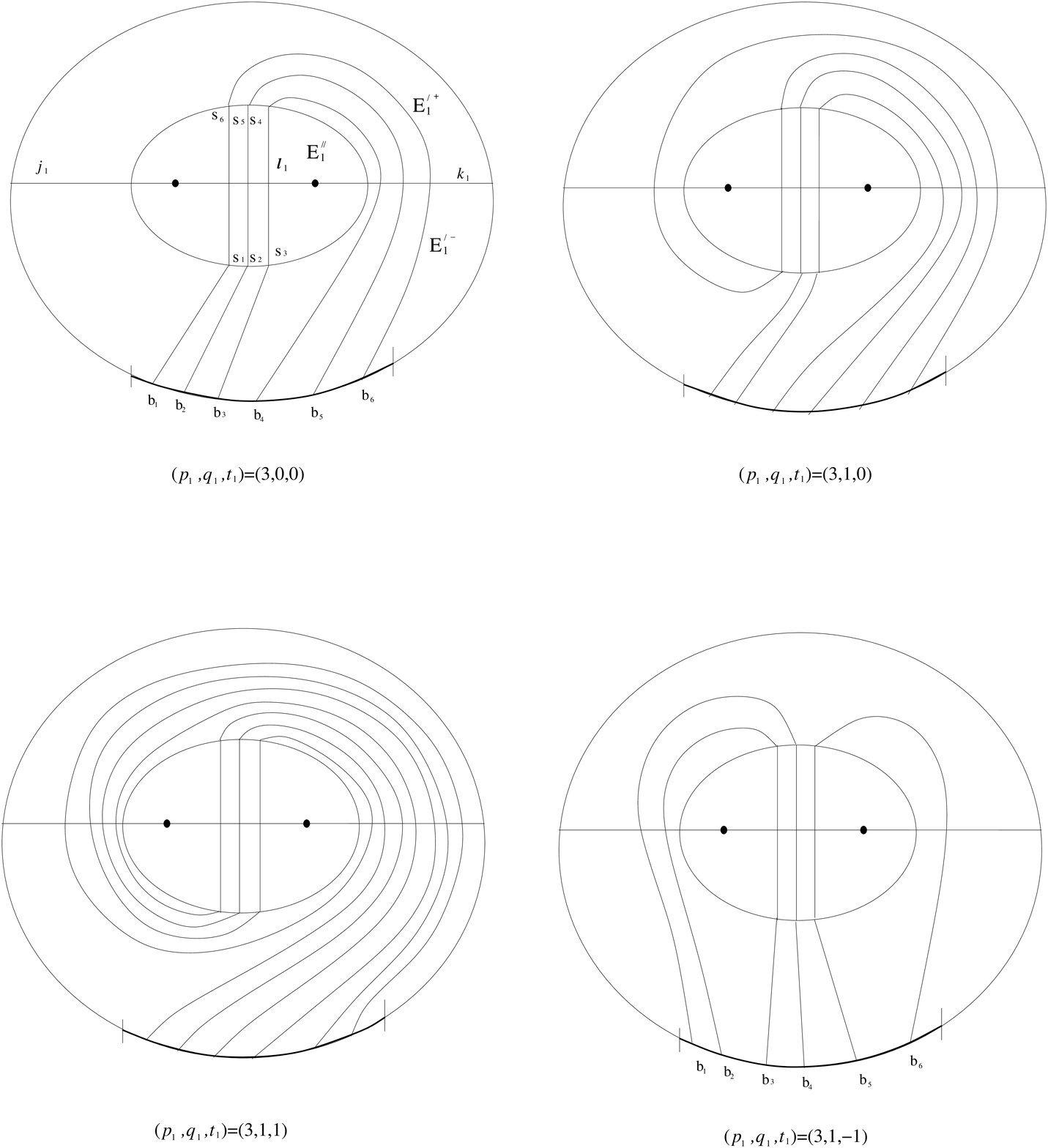}
\end{center}
\caption{}
\label{E2}
 \end{figure}

Now, consider $E_1'$.  Let $j_1$ and $k_1$ be the simple arcs as in Figure~\ref{E2}.
We assume that $\partial E_1'\cup\delta \cup j_1\cup k_1\cup l_1$ has no bigon in $\Sigma_{0,6}$. We note that $j_1\cup k_1\cup l_1$ separates $E_1'$ into two semi-disks ${E_1'}^+$ and ${E_1'}^-$ as in Figure~\ref{E2}.
Let ${u}_1^+$ be the number of subarcs of $\delta$ from $l_1$ to $j_1$ in  ${E_1'}^+$. Also, let ${v}_1^+$ be the number of subarcs of $\delta$ from $l_1$ to $k_1$ in ${E_1'}^+$ and let ${w}_1^+$ be the number of subarcs of $\delta$ from $j_1$ to $k_1$ in ${E_1'}^+$.  
Let $m_1=|\delta\cap j_1|$ and $n_1=|\delta\cap k_1|$ in $E_1'$.  For example, in the third diagram of Figure~\ref{E2}, we have ${u}_1^+=0$, ${v}_1^+=3$, ${w}_1^+=4$, $m_1=4$ and $n_1=7$.\\

 We notice that each component of $\delta\cap E_1'$ meets $l_1$ exactly once. Also, we know that each such component is essential in 
$E_1'-\{1,2\}$.\\

The components of $\delta\cap E_1'$ are determined by three parameters $p_1,q_1,t_1$ as in Figure~\ref{E2},
where $p_1=\min\{|\delta'\cap l_1||\delta'\sim \delta$ in $\Sigma_{0,6}\}$, $q_1\in \mathbb{Z},~0\leq q_1<p_1$. 
In order to define $q_1$ and $t_1$, consider $m_1$ and $n_1$.
 Then we know that $u_1^++v_1^+=p_1$. 
So, $m_1-n_1=(u_1^++w_1^+)-(v_1^++w_1^+)=u_1^+-v_1^+$. Therefore, $-p_1=-u_1^+-v_1^+\leq u_1^+-v_1^+=m_1-n_1=u_1^+-v_1^+\leq u_1^++v_1^+=p_1$.
So, we know $-p_1\leq m_1-n_1\leq p_1$. Now, we define $q_1$ and $t_1$ as follows. If $n_1-m_1=p_1$ then $q_1\equiv m_1$ (mod $p_1$) and $0\leq q_1<p_1$, and $t_1={m_1-q_1\over p_1}$ and if $-p_1\leq n_1-m_1<p_1$ then $q_1\equiv -m_1$ (mod $p_1$) and $0\leq q_1<p_1$, and $t_1={-m_1-q_1\over p_1}$.
Then $t_1$ is called the $twisting$ $number$ in $E_1'$. 
Also, let $(p_1,q_1,t_1)$ be the three parameters to determine the arcs in $E_1'$. Similarly, we have the three parameters $(p_i,q_i,t_i)$ for $E_i'$ ($i=2,3$). 
Then $\gamma$ is determined by a sequence of nine parameters $(p_1,q_1,t_1,p_2,q_2,t_2,p_3,q_3,t_3)$ by Lemma~\ref{T51}.

\begin{figure}[htb]
\begin{center}
\includegraphics[scale=.25]{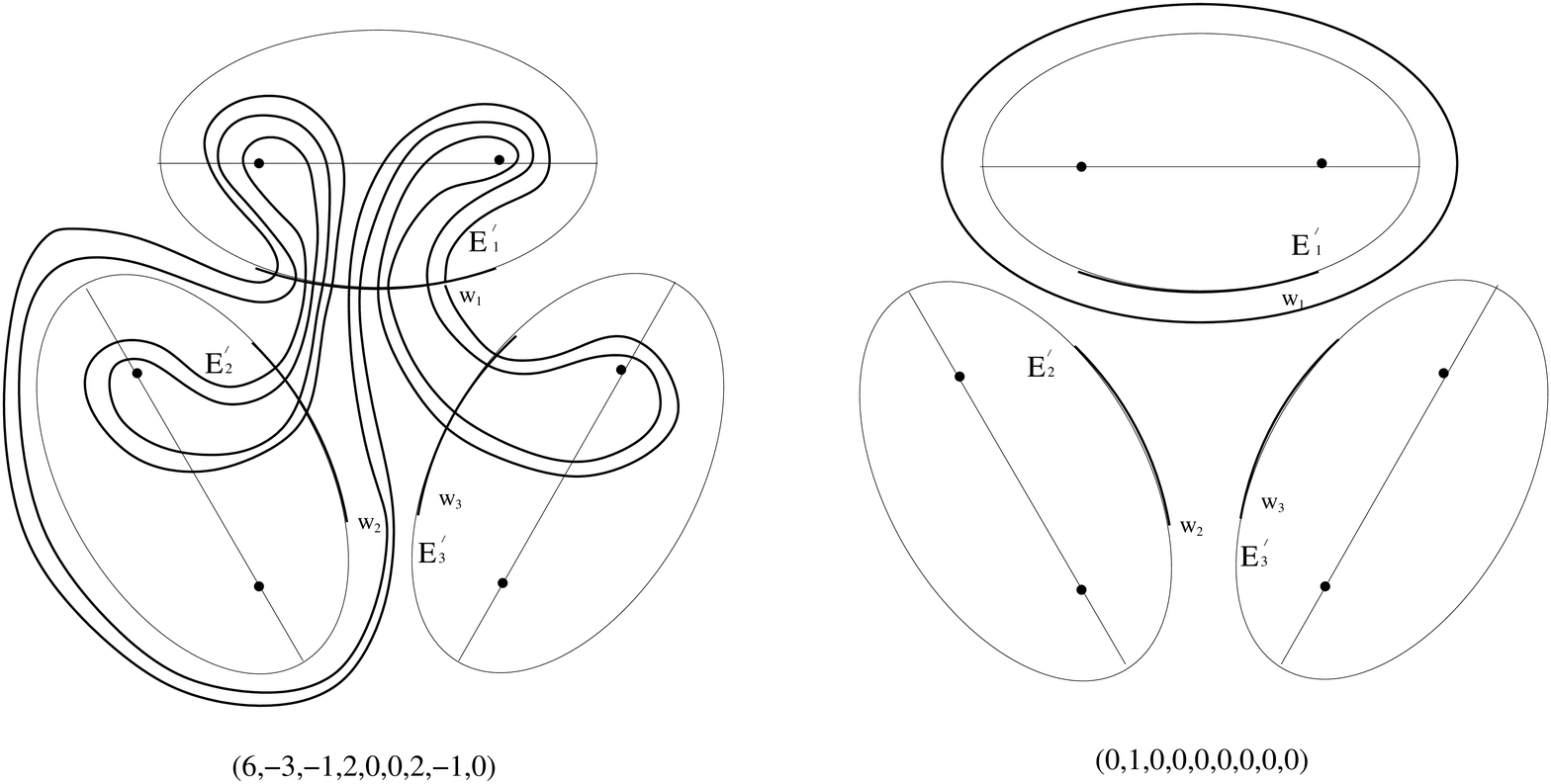}
\end{center}
\caption{Essential curves obtained by a sequece of nine parameters}
\label{E3}
\end{figure}

\begin{Lem}\label{T51}

$I_i$ ($i=1,2,3$) determine the weights $x_{jk}$. ($j,k\in\{1,2,3\}$)
\end{Lem}

\begin{proof} We have two subcases for this.
First, suppose that $I_i<I_j+I_k$ for all distinct $i,j,k\in\{1,2,3\}$. 
We claim that $x_{11}=x_{22}=x_{33}=0$. If not, then $x_{ii}>0$ for some $i$.
We notice that $x_{jj}=x_{kk}=0$. So we have $I_i=2x_{ii}+x_{ij}+x_{jk}$, $I_j=x_{ij}$ and $I_k=x_{ik}$.
This shows that $2x_{ii}+x_{ij}+x_{jk}<x_{ij}+x_{ik}$. This makes a contradiction. So, $x_{11}=x_{22}=x_{33}=0$.
Now, we have $I_i=x_{ij}+x_{ik}$. This implies that
 $x_{ij}={{I_i+I_j-I_k}\over 2}$. \\

Now, suppose that $I_i>I_j+I_k$ for some $i$.
 Then we notice that $\gamma$ has $I_i=2x_{ii}+x_{ij}+x_{ik}$, $I_j=x_{ij}$ and $I_k=x_{ik}$.\\

This implies that $x_{ij}=I_j,~ x_{ik}=I_k$ and $x_{ii}={I_i-I_j-I_k\over 2}$. 

\end{proof}

Recall that $\mathcal{C}$ is the set of isotopy classes of simple closed curves in $\Sigma_{0,6}$.
For a given simple closed curve $\delta$ in a hexagon diagram, we define $p_i$, $q_i$ and $t_i$ in $E_i'$ as above. Then 
let $q_i'=p_it_i+q_i$ for $i=1,2,3$. 
\begin{Thm} [Special case of Dehn's Theorem ]\label{T52}

There is an one-to-one map $\phi:\mathcal{C}\rightarrow\mathbb{Z}^6$ so that $\phi(\delta)=(p_1,p_2,p_3,q_1',q_2',q_3')$. i.e., it classifies isotopy classes of simple closed curves.
 \end{Thm}

When $p_1=p_2=p_3=0$ then $t_i'=1$ if the simple closed curve is isotopic to $\partial E_i'$ and $t_j'=0$ if $j\neq i$.
Refer~\cite{6} to see the general Dehn's theorem.\\

We will use a sequence of nine parameters instead of six parameters for convenience.

\section{ Step 2-2: Hexagon diagram and Dehn diagram}

Let $(p_1,q_1,t_1,p_2,q_2,t_2,p_3,q_3,t_3)$ be the nine parameters of $\gamma$ in $\Sigma_{0,6}$.
Assume that $\gamma$ bounds an essential disk $A$ in $B^3-\epsilon$.
Then we notice that $|\gamma\cap l_i|$ is an even number and  $H_i$ cannot contain a bigon component of the closures of $A-E$.
If $H_i$ contains a bigon $\Delta$, then $\Delta$ will meet $\epsilon_i$. This makes a contradiction that $A$ is an essential disk in $B^3-\epsilon$.\\

To have an essential disk $A$ which is not parallel to one of the $E_i$, at least two components of the closures of $A-E$ are bigon in $P$ since $H_i$ cannot have a bigon.
So, $\gamma$ has  $l_{ii}$ arcs in $I$ for some $i$.
We notice that we cannot have $l_{ii}$ and $l_{jj}(j\neq i)$ at the same time.
 We assume that $x_{11}>0$ and $x_{22}=x_{33}=0$. If not, then we rotate $\gamma$ a multiple of $120^\circ$ about the center of $H$ to have $x_{11}>0$. Note that  the $120^\circ$ rotation  preserves  $\infty$ tangle.
Essential curves obtained by a sequece of nine parameters
The following lemma is very useful to simplify the sequence of parameters of $\gamma$.

\begin{Lem}\label{T61}
 Let $N$ be an essential disk in $B^3-\epsilon$ and $h$ be the clockwise half Dehn twist supported on $N'$ which is the 2-punctured disk in $\Sigma_{0,6}$ so that $\partial N=\partial N'$ . Then  $\gamma$ bounds an essential disk in $B^3-\epsilon$ if and only if  $h(\gamma)$ bounds an essential disk in $B^3-\epsilon$. 
 \end{Lem}
 
  \begin{proof}
Let $B_1$ and $B_2$ be the closures of two components of $B^3-N$. Assume that $B_1$ contains one arc of $\epsilon_i$ for $i=1,2,3$.
Consider an extended homeomorphism $H^{-1}$ of $h^{-1}$ from $(B^3,H(\epsilon))$ to $(B^3,\epsilon)$ so that $H(N)=N$ and $H|_{B_2}=id_{B_2}$. Then, $H^{-1}$ interchanges the endpoints of the properly embedded arc in $B_1$ without changing the tangle type. So, we know that $(B^3,\epsilon)\approx (B^3, H(\epsilon))$.
Now, we know that there exists $i$ so that $E_i'$  contains the two punctures of $N'$ since $N$ is an essential disk in $B^3-\epsilon$.
Let $K_1$ and $K_2$ be the closure of two components of $B^3-E_i$. Actually, $K_1=H_i$ and $K_2=B^3-H_i$.
Let $M_1$ and $M_2$ be the closure of two components of $B^3-N$. We assume that $K_1$ and $M_1$ contains only the same two punctures.
Now, we can construct a homeomorphism $J$ from $(B^3,\epsilon)$ to $(B^3,\epsilon)$ so that $J(K_1)=M_1$, $J(K_2)=M_2$ and $J(\epsilon_i)=\epsilon_i$ for  $i=1,2,3$. 
So, we know that $(B^3,\epsilon)\approx (B^3,J(\epsilon))=(B^3,\epsilon)$. This implies that $(B^3,H^{-1}(\epsilon))\approx (B^3,J(\epsilon)).$
By using Theorem~\ref{T32}, we know that $(H^{-1})^{-1}J(\partial E_i)=H(\gamma )=h(\gamma)$ bounds an essential disk in $B^3-\epsilon$.\\

To see the other direction, we consider $h^{-1}$ which is the counter-clockwise half Dehn twist supported on $N'$. 
\end{proof}

By using this lemma, we notice that a simple closed curve $\gamma'$ which is parameterized by $(p_1,q_1,0,p_2,q_2,0,p_3,q_3,0)$  bounds an essential disk in $B^3-\epsilon$ if only if $\gamma$ does. \\

Now, we will discuss how to modify $\gamma$ into $\gamma'$ which is parameterized by $(p_1,q_1,0,p_2,$ $q_2,0,p_3,q_3,0)$.\\

\begin{figure}[htb]
\begin{center}
\includegraphics[scale=.25]{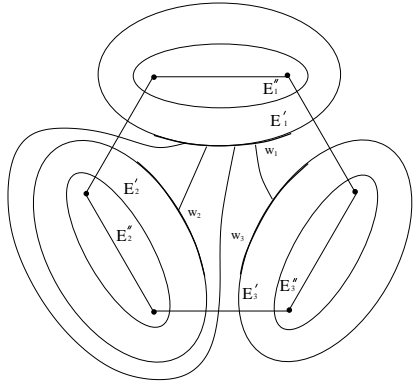}
\end{center}
\caption{Pseudo-hexagon diagram}
\label{F1}
\end{figure}

In $\Sigma_{0,6}$, we choose windows $\omega_i$ and three two punctured disks $E_i'$ to have a pseudo-hexagon diagram as in Figure~\ref{F1}.
We isotope $\gamma$ into $\delta$ in $\Sigma_{0,6}$ so that all components of $\delta\cap I$ are parallel to one of the standard arcs as in Figure~\ref{F1}. i.e., $l_{12}$ and $l_{13}$ lie in $H$ and $l_{11}$ meets the hexagon at exactly two times. \\

Now, consider the graph $(\cup_{j=1}^6a_j)\cup\partial E'\cup \delta$.
Then we assume that $\partial E_i'\cup \delta\cup(E_i'\cap \cup_{j=1}^6a_j)$ has no bigon for $i=1,2,3$. Then let $v_{ij}$ be the number of arcs for $\delta$ from $a_i$ to $a_j$ in $H$ and
let $v^{ij}$ be the number of arcs for $\delta$ from $a_i$ to $a_j$ in $H^c$.

\begin{Lem}\label{T62}
Suppose that $\delta$ is a simple closed curve in pseudo-hexagon diagram for which $x_{11}>0$. Then $v_{11}=v_{33}=v_{44}=v_{55}=v_{35}=0$ and $v^{ii}=0$ for all $i=1,2,3,4,5,6$.
\end{Lem}

\begin{proof}  
Suppose that $v_{11}>0$. Then there is a component $\alpha$ of $\delta\cap H$ with $\partial \alpha\subset a_1$. We notice that $\alpha$ cannot be carried by $l_{11}$ since $l_{11}$ arc meets $a_6$ and $a_4$. If $\alpha$ is carried by $l_{12}$, then $\alpha$ needs to meet $\omega_2$ at least two times. However, $\alpha$ cannot meet $a_4,~a_5$ or $a_6$ in $E_2'$. Therefore, there exists an arc which is parallel in $E_2'$ to a subarc of $\omega_2$. So, we can isotope the arc out of $E_2'$. This makes a bigon in $E_2'$. This contradicts the definition of a pseudo-hexagon diagram.
Therefore, $\alpha$ cannot be carried by $l_{12}$. Similarly, $\alpha$ cannot be carried by $l_{13}$. If it cannot be carried by $l_{11},~l_{12}$ or $l_{13}$ then $\alpha$ is parallel in $E_1'$ to an arc in $a_1$. This contradicts the fact that each component of $\delta\cap E_1'$ meets $a_1$ exactly once.\\

Suppose that $v_{33}>0$. Then there is a component $\alpha$ of $\delta\cap H$ with $\partial \alpha\subset a_3$.
We notice that $a_3\subset E_3''$. So, $\partial \alpha$ is in $E_3''$. If $\alpha$ is carried by $l_{13}$, $\alpha$ needs to pass through $\omega_1$. However, it cannot come back to $\omega_1$ without meet $a_6,~a_1$ or $a_2$ because it is essential in $E_1'-\{1,2\}$.  Therefore, $\alpha$ is not carried by $l_{13}$. If $\alpha$ cannot be carried by $l_{13}$ then $\alpha$ is parallel in $E_3'$ to an arc in $a_3$. This contradicts the fact that each component of $\delta\cap E_3'$ meets $a_3$ exactly once.  With a similar argument, we also can show that $v_{55}=0$. \\

Suppose that $v_{44}>0$. Then there is a component $\alpha$ of $\delta\cap H$ with $\partial \alpha\subset a_4$. If $\alpha$ cannot be carried by $l_{23}$, then $\alpha$ is parallel in either $E_2'$ or $E_3'$ to an arc in $a_5$ or $a_3$ respectively. This contradicts the fact that each component of $\delta\cap E_2'$ or $\delta\cap E_3'$ meets $a_5$ or $a_3$ exactly once.  So, $\alpha$ needs to be carried by $l_{23}$. However, we know that $x_{23}=0$ since $x_{11}>0$. This implies that $v_{44}$ also should be zero. \\

Suppose that $v_{35}>0$. Let $\alpha$ be an arc for $v_{35}$. Then the endpoints lie in both $E_2'$ and $E_3'$. Therefore, $\alpha$ need to be carried by $l_{23}$.
However, $x_{23}=0$ since $x_{11}>0$. This implies that $v_{35}=0$.\\

Suppose that $v^{11}>0$. Then there is a component $\alpha$ of $\delta\cap H^c$ with $\partial \alpha\subset a_1$. We notice that $\alpha$ is parallel in $E_1'$ to an arc in $a_1$. This condradicts the fact that each component of $\delta\cap E_1'$ meets $a_1$ exactly once. 
Therefore, $v^{11}=0$. With a similar argument, we can show that $v^{33}=v^{55}=0$.\\

Suppose that $v^{22}>0$. Then there is a component $\alpha$ of $\delta\cap H^c$ with $\partial \alpha\subset a_2$. We notice that both of the endpoints of $\alpha$ lies in $E_1'$ or $E_3'$ since there is no subarc of $\delta$ from $E_1'$ to $E_2'$, from $E_1'$ to $E_3'$ or from $E_2'$ to $E_3'$ in $H^c$. So, $\alpha$ is parallel in $E_1'$ or $E_3'$ to an arc in $a_2$. So, we can isotope $\alpha$ in $E_1'$ or $E_3'$ to reduce the intersection number of $\delta$ with $\cup_{i=1}^6a_i$. This contradicts that there is no bigon in $E_3'$  in pseudo-hexagon diagram. Therefore, $v^{22}=0$.\\

Similarly, we can show that  $v^{44}=v^{66}=0$.
\end{proof}
This lemma shows that  $v_{22}$ and $v_{66}$ are the only $v_{ii}$ which might be positive integers.\\

\begin{figure}[htb]
\begin{center}
\includegraphics[scale=.25]{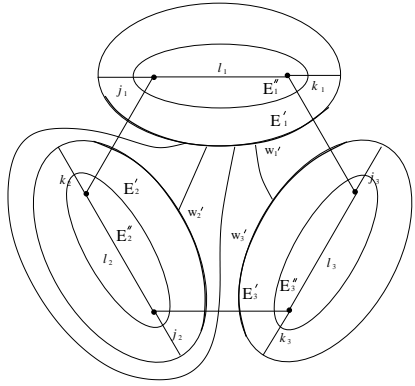}
\end{center}
\caption{Pseudo-hexagon diagram and Hexagon diagram}
\label{F2}
\end{figure}

Now, we take  new windows $\omega'_i$ as in Figure~\ref{F2}.\\

Then, we can isotope all arcs for $v_{22}$ and $v_{66}$ by pushing across $a_2$ and $a_6$ respectively to have a new simple closed curve $\eta$ with windows $\omega'_i$. Let $u_{ij}$ be the number of arcs of $\eta$ from $a_i$ to $a_j$ in $H$. Also, let $u^{ij}$ be the number of arcs of $\eta$ from
$a_i$ to $a_j$ in $H^c$.

\begin{Lem}\label{T63}
Suppose that $\delta$ is a simple closed curve in pseudo-hexagon diagram. Then  $\eta$ has $u_{ii}=u^{ii}=0$ for all $i=1,2,3,4,5,6$.
\end{Lem}

\begin{proof}
We notice that all arcs for $v_{ij}$  are essential in $H$ if $i\neq j$. 
Also, all arc for $v^{ij}$ are essential in $H^c$ if $i\neq j$. By Lemma~\ref{T62}, we know that $v^{ii}=0$ for all $i$ and $v_{22}$ and $v_{66}$ are the only $v_{ii}$ which might be positive integers.\\

Now, let $\alpha_1$ be an arc for $v_{22}$.
Then, we claim that  $\alpha_1$ is carried by $l_{13}$. If $\alpha_1$ is not carried by $l_{13}$ then the endpoints lie in either $E_1'$ or $E_3'$. So, $\alpha_1$ is parallel in $E_1'$ or $E_3'$ to an arc in $a_2$. So, we can isotope $\alpha_1$ in $E_1'$ or $E_3'$ to reduce the intersection of $\delta$ with $a_2$. This contradicts the fact that there is no bigon in $E_i'$ in a pseudo-hexagon diagram. Therefore, $\alpha_1$ is carried by $l_{13}$.  Let $x$ be the endpoint of $\alpha_1$ in $E_1'$ and $y$ be the endpoint of $\alpha_1$ in $E_3'$. Let $\beta_1$ be the component  of $(\delta-\alpha_1)\cap E_1'$ which contains $x$. Then let $\theta_1$ be the component of $\beta_1\cap H^c$ which contains $x$. Then the other endpoint $z_1$ of $\theta_1$ lies on either $a_1$ or $a_6$.  Similarly, Let $\beta_2$ be the component of $(\delta-\alpha_1)\cap E_3'$ which contains $y$ and $\theta_2$ be the component of $\beta_2\cap H^c$ which contains $y$. Then the other endpoint $z_2$ of $\theta_2$ lies on either $a_3$ or $a_4$.\\

Therefore, by pushing $\alpha_1$ across $a_2$ we get an arc for either $u^{13},~u^{14},~u^{63}$ or $u^{64}$. We notice that each arc for $u^{13},~u^{14},~u^{63}$ or $u^{64}$ is essential in $H^c$. \\

Now, let $\alpha_2$ be an arc for $v_{66}$. Then , we know that $\alpha_2$ is carried by $l_{11}$ or $l_{12}$.
 Let $x'$ be the endpoint of $\alpha_2$ in $E_1'$ and $y'$ be the endpoint of $\alpha_2$ in ${E_1'}^c$. Then let $\beta_1'$ be the component  of $(\delta-\alpha_2)\cap E_1'$ 
which contains $x'$ and $\theta_1'$ be the component of $\beta_1'\cap H^c$ which contains $x'$.  Then the other endpoint $z_1'$ of $\theta_1'$ lies on either $a_1$ or $a_2$. Similarly, let  $\beta_2'$ be the component of $(\delta-\alpha_2)\cap (E_1')^c$ which contains $y$ and let $\theta_2'$ be the component of $\beta_2'\cap H^c$ which contains $y$. Then the other endpoint $z_2'$ of $\theta_2'$ lies on either  $a_4$ or $a_5$.\\

Therefore, by pushing $\alpha_2$ across $a_6$ we get an arc for either $u^{14},~u^{15},~u^{24}$ or $u^{25}$. We notice that each arc for $u^{14},~u^{15},~u^{24}$ or $u^{25}$ is essential in $H^c$. It is possible to use an arc $\alpha_3$ for $v^{26}$ to have a new arc $\alpha=\alpha_1\cup\alpha_2\cup\alpha_3$ so that $|\alpha_1\cap \alpha_3|=|\alpha_2\cap\alpha_3|=1$. So, if we push $\alpha_1$ and $\alpha_2$ across $a_2$ and $a_6$ respectively then we have a new arc which might have $\partial \alpha$ in $a_4$. However, we notice that $\alpha\cap E_1'$ is a non-essential arc in $E_1'-\{1,2\}$. This contradicts the fact that each component of $\delta\cap E_1'$ is essential in $E_1'-\{1,2\}$. Therefore, this case cannot be happen. This implies that $u_{ii}=0$ and $u^{ii}=0$ for all $i$. This completes the proof of this lemma.
\end{proof}

Then the following lemma is also true. 

\begin{Lem}\label{T64}
  $u_{ij}$ and $u^{ij}$ ($i,j\in\{1,2,3,4,5,6\}$) are the weights for a hexagon diagram of $\delta$. 
\end{Lem}

\begin{proof}
By Lemma~\ref{T43} and Lemma~\ref{T63}, we get this lemma.
\end{proof}

Recall that $I_i=2(\sum_{k=1}^6 w_{ki})=2p_i$ for $i=1,2,3$.
Then by Lemma~\ref{T51}, we can calculate the weights $x_{ij}$ of $l_{ij}$. We remark that $x_{ij}$ only depends on $p_1$, $p_2$ and $p_3$.\\

It is clear that $p_i={I_i\over 2}$. But, it is difficult to find $q_i$ and $t_i$ together. So, I want to find $q_i$ by making $t_i$ zero for all $i$.
Actually, if $t_i=0$ for all $i$, then we can find $q_i$ as follows.
\begin{Lem}\label{T65} Let $\delta$ be a simple closed curve in a Dehn diagram. Let $w_{ij}$ and $w^{ij}$ be the weights of $\delta$ in a hexagon diagram.\\

Suppose that $t_1=t_2=t_3=0$. Then 
$q_1=w^{26}$, $q_2=w^{46}-x_{11}$ and $q_3=w^{24}$.
\end{Lem}
\begin{proof}
Consider the pseudo-hexagon diagram of $\delta$ with the weights $v_{ij}$ and $v^{ij}$.
Suppose that $t_1=t_2=t_3=0$. We notice that the graph $\partial E_1'\cup (E_1'\cap (a_1\cup a_2\cup a_6))\cup \delta$. Then consider a subarc $C$ of $\delta$ in $H$ so that one of the endpoints of $C$ is on $\omega_1$. Then the other endpoint should be on $a_2$ since $t_1=0$.  This implies that $v_{66}=0$. Similarly, we see that $v_{22}=0$ since $t_3=0$.  So, $v_{ii}=v^{ii}=0$ for all $i=1,2,...,6$ by Lemma~\ref{T62}. Therefore, $v_{ij}=w_{ij}$ and $v^{ij}=w^{ij}$ for $i,j\in\{1,2,3,4,5,6\}$.
By definition of $q_i$, we have $q_1=m_1=v^{26}=w^{26}$ since $t_1=0$.\\

Let $m_2=|\delta\cap j_2|$ and $n_2=|\delta\cap k_2|$ in $E_2'$. Also, let $m_3=|\delta\cap j_3|$ and $n_3=|\delta\cap k_3|$ in $E_3'$.
Then, we also know that $v^{46}-l_{11}=m_2$ and $v^{24}=m_3$.
Then this implies that $q_2=m_2=v^{46}-l_{11}$ and $q_3=m_3=v^{24}$ since $t_2=t_3=0$.

\end{proof}

We say that $\gamma$ is $\textit{right-twisted}$ (or $\textit{left-twisted}$) in $E_i'$ if $t_i>0$ (or $t_i<0)$.
If we know $\delta$ is right-twisted (or left-twisted) in $E_i'$, then we apply a half Dehn twist supported on $E_i'$ to decrease (or increase) the twisting number $t_i$ until the  simple closed curve is not twisted.

\begin{Lem}\label{T66}
$\gamma$ is left-twisted in $E_1'$ if and only if $u_1^+>0$.
\end{Lem}
\begin{proof}
 If $\gamma$ is left-twisted in $E_1'$ then $t_1<0$ by definition. Then $-p_1\leq m_1-n_1<p_1$. We claim that $u_1^+>0$. If $u_1^+=0$, then we know that $v_1^+=p_1$.  So, $n_1-m_1=v_1^+-u_1^+=p_1$. This contradicts that $-p_1\leq n_1-m_1<p_1$. Therefore, $u_1^+>0$.
Now, suppose that $u_1^+>0$ to show the other direction. Then we know that $-p_1\leq n_1-m_1<p_1$.
This implies that $t_1<0$ by the definition of $t_i$. Therefore, 
 $u_1^+>0$ if and only if $\gamma$ is left-twisted in $E_1'$.
\end{proof}
\begin{Lem}\label{T67}
Suppose that $\gamma$ has $x_{11}>0$.
 $\gamma$ is left-twisted in $E_1'$ if and only if $\gamma$ has either
 \begin{enumerate}
 \item $w^{15}+w^{16}>0$ 
 \item $w^{15}=w^{16}=0$, $w^{14}>0$ and $w^{45}+w^{46}< w_{45}+w_{46}+w_{14}$, or
 \item $w^{15}=w^{16}=0$, $w^{14}>0$, $w_{45}+w_{46}+w_{14}\leq w^{45}+w^{46}< w_{45}+w_{46}+w_{14}+w_{24}$ and $w_{26}+w_{25}+( w^{45}+w^{46}-(w_{45}+w_{46}+w_{14}))<w^{12}$.
\end{enumerate}
\end{Lem}

\begin{proof}

$(\Rightarrow)$ Suppose that $\gamma$  has $x_{11}>0$ and $\gamma$ is left-twisted in $E_1'$ with $w^{15}=w^{16}=0$.\\

\begin{figure}[htb]
\begin{center}
\includegraphics[scale=.28]{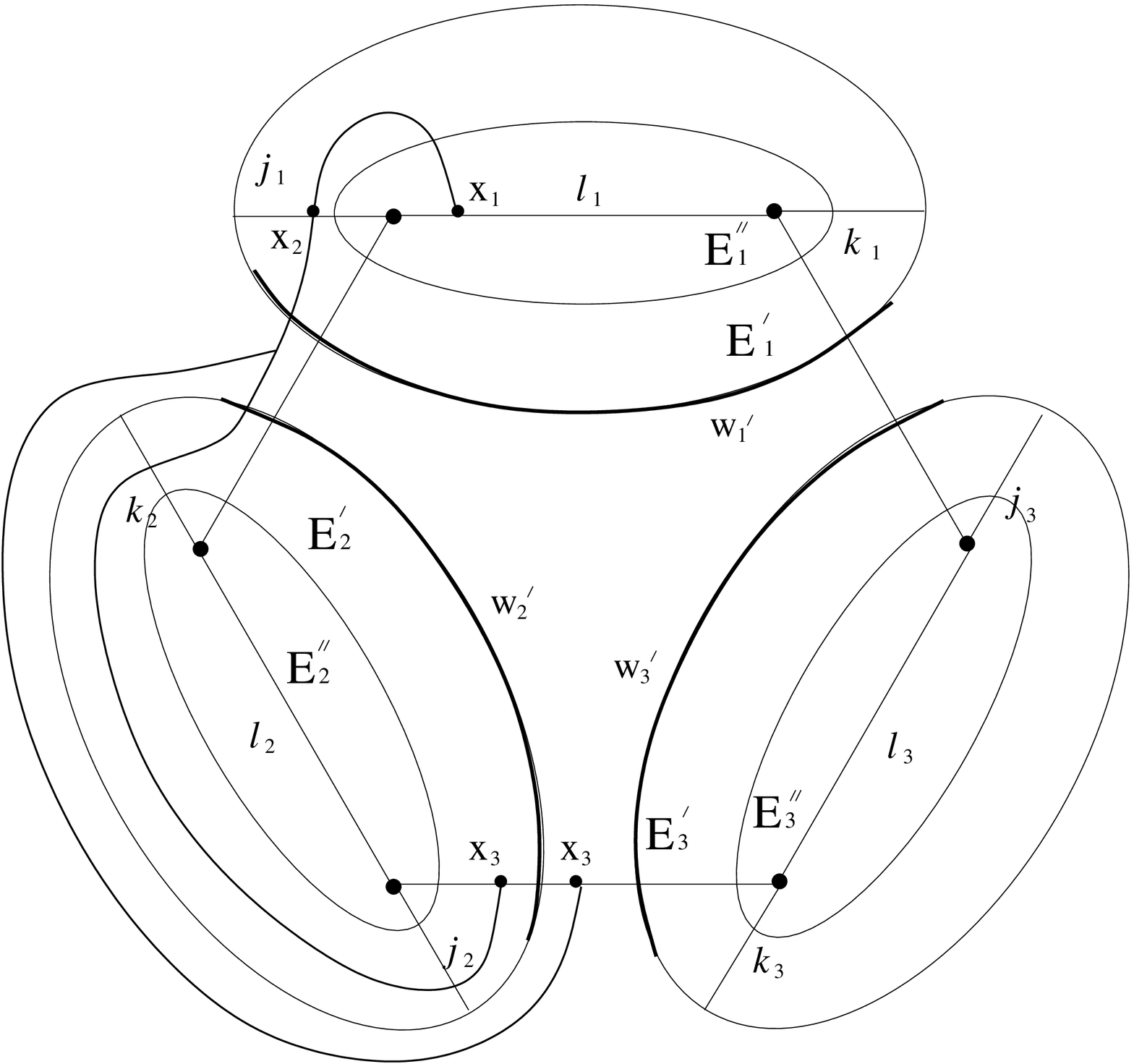}
\end{center}
\caption{}
\label{F3}
\end{figure}
 Then $\gamma$ has $u_1^+>0$ by Lemma~\ref{T66} since $\gamma$ is left-twisted. Let $C$ be the innermost arc for $u_1^+$ with the endpoints $x_1$ and $x_2$ which lie on $l_1$ and $j_1$ respectively. 
Then there is a component $\alpha_1$ of $(\gamma-C)\cap H^c$ so that one of endpoints of $\alpha_1$ is $x_2$. Then, the other endpoint $x_3$ lies on either $a_6$, $a_5$ or $a_4$.
Because $w^{15}=w^{16}=0$, we notice that $x_3$ lies on $a_4$. This implies that $w^{14}>0$. We know that $\alpha_1$ is carried by either $l_{11}$ or $l_{12}$.
We also know that $w^{26}=0$ because of $C\cup \alpha$.
We notice that $C\cup \alpha$ is also the innermost arc for $w^{14}$ since $C$ is the innermost arc for $u_1^+$.\\

Because $x_{23}=0$, we know that $x_3$ goes along a component $\alpha_2$ of $(\gamma-(C\cup\alpha_1))\cap H$ and the other endpoint $x_4$ of $\alpha_2$ lies on either $a_5,~a_6,~a_1$ or $a_2$. This implies that $w^{45}+w^{46}< w_{45}+w_{46}+w_{14}+w_{24}$.
Especially, if $x_4$ lies on $a_2$ then we need a condition that $w_{45}+w_{46}+w_{14}\leq w^{45}+w^{46}< w_{45}+w_{46}+w_{14}+w_{24}$. In this case, $x_4$ needs to go along a component $\alpha_3$ of $(\gamma-(C\cup\alpha_1\cup\alpha_2))\cap H^c$ and the other endpoint $x_5$ of $\alpha_3$ should be on $a_1$ since $w^{26}=0$.
Therefore, we have $w_{26}+w_{25}+w_{21}+( w^{45}+w^{46}-(w_{45}+w_{46}+w_{14}))<w^{12}$.
We notice that $w_{21}=0$ since $w^{12}>0$. Finally, $w_{26}+w_{25}+( w^{45}+w^{46}-(w_{45}+w_{46}+w_{14}))<w^{12}$.\\

$(\Leftarrow)$ Suppose that $\gamma$ is not left-twisted and $w^{16}=w^{15}=0$, $w^{14}>0$ and  $w_{45}+w_{46}+w_{14}\leq w^{45}+w^{46}< w_{45}+w_{46}+w_{14}+w_{24}$.
Since $\gamma$ is not left-twisted, we know that $n_1-m_1=v_1^+=p_1$. \\

\begin{figure}[htb]
\begin{center}
\includegraphics[scale=.3]{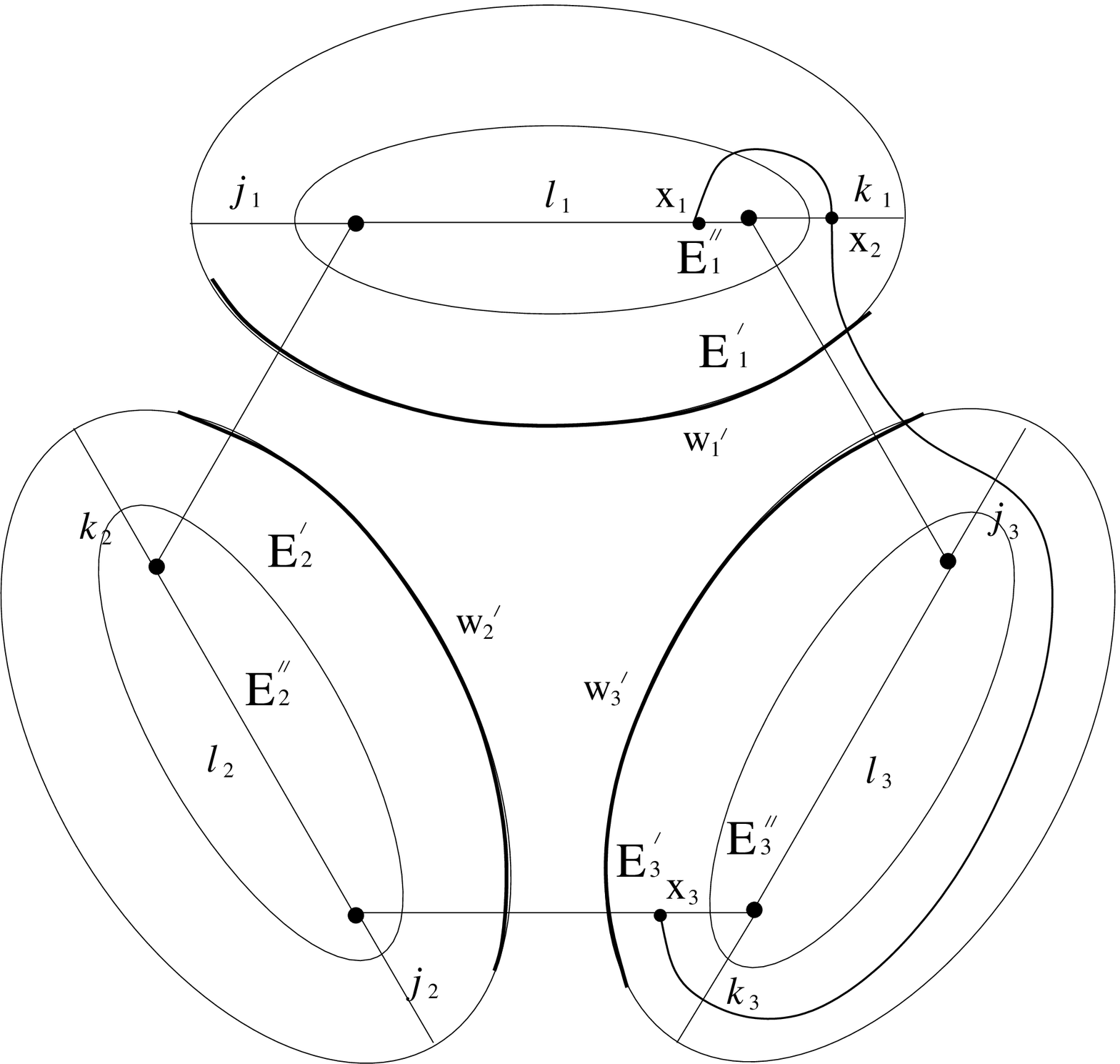}
\end{center}
\caption{}
\label{F4}
\end{figure}
Let $C$ be the outermost arc for $v_1^+$ with the endpoints $x_1$ and $x_2$ as in Figure~\ref{F4}.
Let $\alpha_1$ be the component of $(\gamma-C)\cap H^c$ so that $x_2$ is one of the endpoints of $\alpha_1$.
We notice that $C\cup \alpha_1$ needs to be the outermost arc for $w^{14}$.
Let $x_3$ be the other endpoint of $\alpha_1$. Then $x_3$ should be on $a_4$.\\

 $x_3$ needs to go alone a component $\alpha_2$ of $(\gamma-(C\cup \alpha))\cap H$. Then the other endpoint $x_4$ of $\alpha_2$ should be on $a_2$ 
since $w_{45}+w_{46}+w_{14}\leq w^{45}+w^{46}< w_{45}+w_{46}+w_{14}+w_{24}$.
 After that, $x_4$ continues to go along a component $\alpha_3$ of $(\gamma-(C\cup\alpha_1\cup\alpha_2))\cap H^c$ and the other endpoint $x_5$ of $\alpha_3$ lies on either $a_3$ or $a_4$. Then, we need an equality
 $w_{26}+w_{25}+( w^{45}+w^{46}-(w_{45}+w_{46}+w_{14}))\geq w^{12}$.
This completes the proof of lemma~\ref{T67}.

\end{proof}

If $\gamma$ is left-twisted in $E_1'$, then we apply the counter clockwise half Dehn twist supported on $E_1'$ to $\gamma$.
If $\gamma$ is not left-twisted in $E_1'$, then we notice that $t_1=0$ or $t_1>0$. So, if we apply the clockwise half Dehn twist supported on $E_1'$ to $\gamma$ then
the modified simple closed curve $\gamma'$ is either left-twisted in $E_1'$ or not. If not, then we continue to apply the clockwise half Dehn twist to $\gamma'$.
If $\gamma'$ is left-twisted in $E_1'$ then we notice that $\gamma$ has $t_1=0$. Therefore, this lemma is enough to know if a simple closed curve has $t_1=0$ or not.\\


Now, we need the following lemma to make $t_2$ zero.

\begin{Lem}\label{T68}
Suppose that $x_{11}>0$ and $t_1=0$.
 $\gamma$ is left-twisted in $E_2'$ if and only if $w^{45}>0$.
\end{Lem}
\begin{proof}
We notice that  $u_2^+=w^{45}$ if $t_1=0$. By a similar argumet in lemma~\ref{T67},  we see that $\gamma$ is left-twisted in $E_2'$ if and only if $w^{45}>0$.

\end{proof}

In order to check $\gamma$ is left-twisted in $E_3'$, we  need the following lemma.

\begin{Lem}\label{T69}
Suppose that  $x_{11}>0$ and $t_1=t_2=0$.
 $\gamma$ is left-twisted in $E_3'$ if and only if $w^{13}+w^{14}+w^{63}+w^{64}>0$.
\end{Lem}

\begin{proof}
By using a similar argument in lemma~\ref{T67}, $\gamma$ is left-twisted in $E_3'$ if and only if $u_3^+>0$. 
Let $C$ be an arc for $u_3^+$.\\

$(\Rightarrow)$ Suppose that $w^{13}=w^{14}=w^{63}=w^{64}=0$, then $C$ cannot be a subarc for any $w^{ij}$. Therefore, $\gamma$ is not left-twisted. \\

$(\Leftarrow)$ 
Suppose that $\gamma$ is not left-twisted in $E_3'$ then we know that $w^{13}=w^{14}=w^{63}=w^{64}=0$.

\end{proof}

With the three previous lemmas, we can get $\gamma'$  by applying appropriate half Dehn twists to $\gamma$ so that $\gamma'$ has the same $p_i$ and $q_i$, but $t_i=0$ for all $i$. Also, we know that $\gamma'$ bounds an essential disk in $B^3-\epsilon$ if and only if $\gamma$ does by Lemma~\ref{T61}.

\section{Step 3-1: Pattern diagram of $\gamma'$ in $I'$}

Now, we modify $\gamma'$ into $\gamma_0$ which is in $\emph{standard position}$ to access the main theorem.

\begin{figure}[htb]
\begin{center}
\includegraphics[scale=.45]{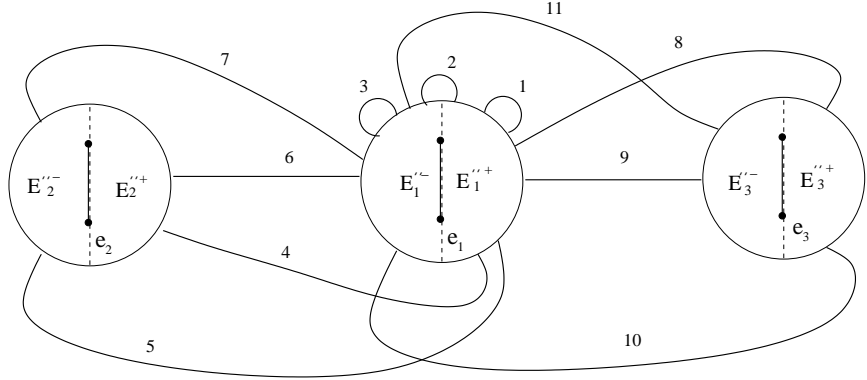}
\end{center}
\caption{Pattern diagram}
\label{G1}
 \end{figure}

Recall that $E_i''$ is  a concentric two punctured disk in the interior of $E_i'$ so that all the components of $\partial A\cap E_i''$ are parallel simple arcs as in Figure~\ref{E2}. Let $I'=S^2-\{E_1''\cup E_2''\cup E_3''\}.$
Now, take equators $e_i$ for each $E_i''$ as in Figure~\ref{G1}. Then we can divide $E_i''$ into $E_i''^{+}$ and $E_i''^{-}$ by $e_i$.\\

Recall the fact that we need to have $x_{ii}>0$ for some $i$ to have an essential disk in $B^3-\epsilon$ which is not parallel to one of the $E_i$, where $x_{ij}$ is the weight of $l_{ij}$ and $l_{ij}$ is the standard arc from the window $\omega_i$ to the window $\omega_j$. So, we assume that $x_{11}>0$ and $x_{22}=x_{33}=0$ without loss of generality.\\

Let $\gamma$ be a simple closed curve which bounds an essential disk in $B^3-\epsilon$ and has $x_{11}>0$.
Now, we define a $pattern$ $diagram$ of $\gamma$ in $I'$ which has 11 types of essential arcs in $I'$ as in Figure~\ref{G1}. \\

 The given number shows each type of arc. For example, type 1 is for an arc from $E_1''^{+}$ to $E_1''^{+}$. These are patterns of connectivity, not isotopy classes of arcs.\\
 
 We will discuss the relation between the \emph{hexagon diagram} and the \emph{pattern diagram} in Section 8

\section{ Step 3-2: Standard diagram of $\gamma_0$ in $I'$}

\begin{figure}[htb]
\includegraphics[scale=.25]{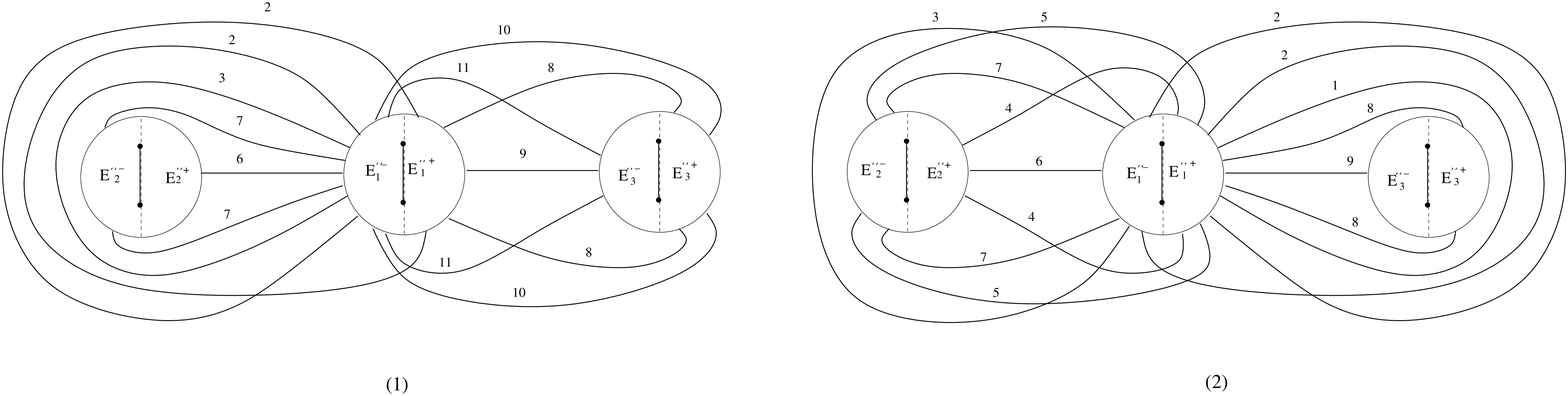}
\caption{}
\label{H1}
 \end{figure}
 
From a simple closed curve $\gamma$ which bounds an essential disk in $B^3-\epsilon$,
we want to construct a new simple closed curve $\gamma_0$ which may be in a different isotopy class in $\Sigma_{0,6}$, but 
$\gamma_0$  bounds an essential disk in $B^3-\epsilon$ if and only if $\gamma$ does.
Especially, every component of $\gamma_0\cap I'$ is isotopic to one of the given  arc types in one of the diagrams in Figure~\ref{H1}.
The two diagrams in Figure~\ref{H1} are called $\emph{standard diagrams}$ and a simple closed curve $\gamma_0$ is in $\emph{standard position}$  if $\gamma_0$ is obtained from one of the standard diagrams by putting weights on these arcs.\\

In this section, we show how $\gamma_0$ is obtained from $\gamma$ by having a certain properly chosen $t_i$ and the same $p_i$ and $q_i$.\\

 \begin{figure}[htb]
 \begin{center}
\includegraphics[scale=.35]{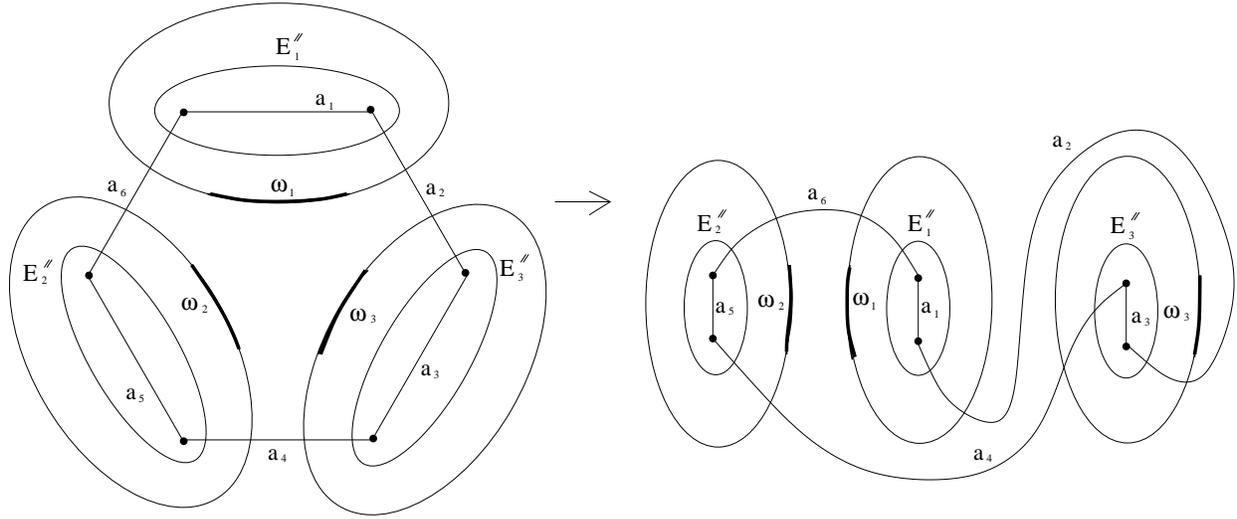}
\end{center}
\caption{Standard diagram}
\label{H2}
 \end{figure}

First, we modify the hexagon diagram into  the new diagram as in Figure~\ref{H2}.\\

We will let $m_k$ be the number of parallel arcs of $\gamma_0\cap I'$ that are isotopic to the arc type $k$ in the standard diagram. Then we say that $m_k$ is the \emph{weight} of the arc type $k$ in the standard diagram. We note that it is possible to have two non-isotopic arcs in $I'$, but they are the same arc type $k$. So, we define $m_{k_1}$ and $m_{k_2}$ of the weights for the two different isotopic arcs with the same arc type $k$ so that $m_{k_1}+m_{k_2}=m_k$ if they exist. Then we have a sequence of weights $m_i$ for $\gamma_0$. It is called a $\emph{standard}$ parameterization. I want to mention the fact that $m_k$ is going to be computed from the Dehn parameters as described in the following pages.\\

Now, we will show the two following lemmas.

\begin{Lem}\label{T81}
Suppose that $\gamma_0$ is a simple closed curve which bounds an essential disk $A$ in $B^3-\epsilon$ and it is in standard position with $x_{11}>0$.
 Then  $m_1+m_3>0$.
\end{Lem} 

\begin{proof}

  \begin{figure}[htb]
  \begin{center}
\includegraphics[scale=.4]{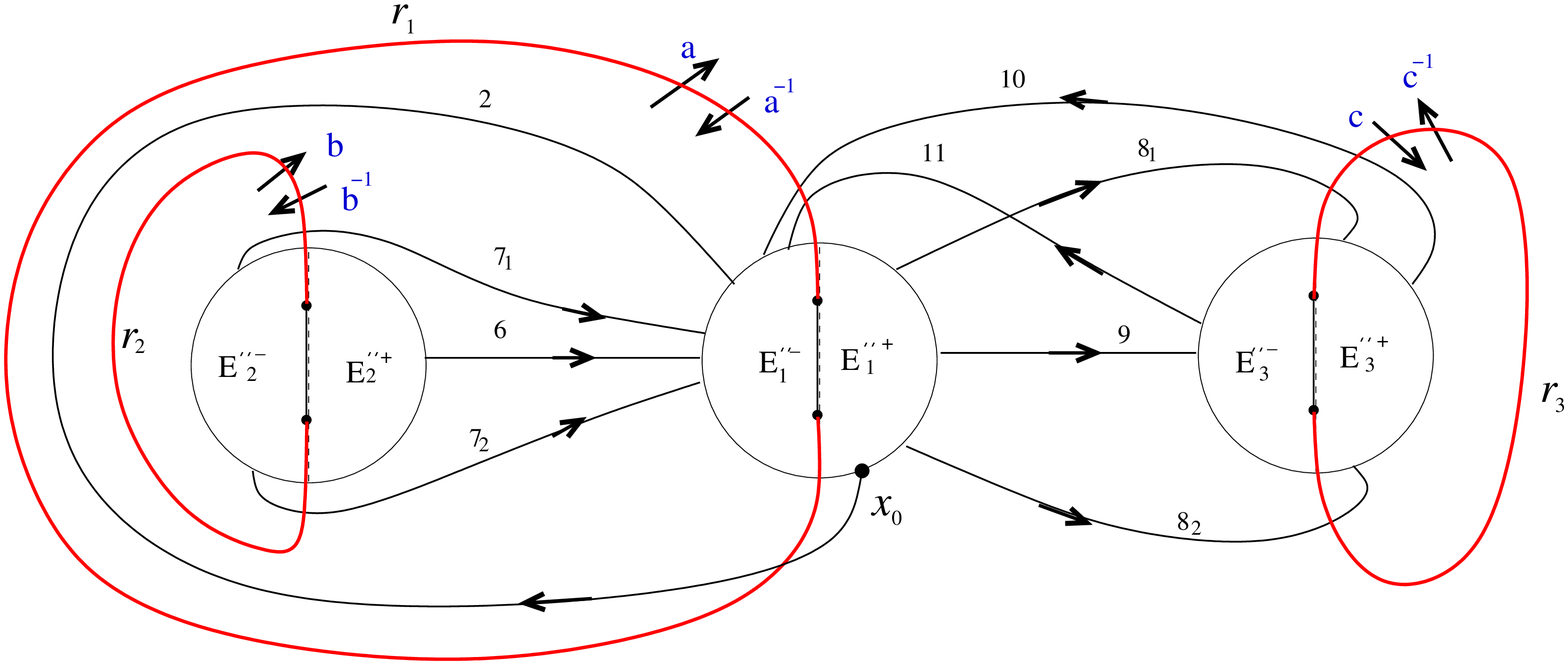}
\end{center}
\caption{}
\label{H3}
\end{figure}

Consider the standard diagram with $m_2>0, ~m_1=m_3=0$ as in Figure~\ref{H3}. By referring to $(1)$ in Figure~\ref{H1}  we can have two possible diagrams to have $m_2>0$. We note that the other diagram can be obtained by reflecting the given diagram about the horizontal axis which is passing through the middles of $E_i''$. Then we have a similar argument to check the other case. So, it is enough to consider the diagram as in Figure~\ref{H3}.\\

 We take a small open neighborhood $N^\circ(\epsilon)$ of $\epsilon$ so that $B^3-N^\circ(\epsilon)$ is a handlebody. Then  $[\partial A]$ is the trivial element of $\pi_1(B^3-N^\circ(\epsilon))$ since $A$ is an essential disk in $B^3-\epsilon$.\\
 
  Let $r_i$ be the simple arcs as in Figure~\ref{H3}. Then we note that the $r_i\cup \epsilon_i$ bound three disjoint disks in $B^3$. So, by considering the intersections between $\partial A$ and $r_i$ we can calculate the element $[\partial A]$ of $\pi_1(B^3-N^\circ(\epsilon),x_0)$ as in Figure~\ref{H3}.\\
  
  Now, we consider  all the possible cases for  the element $[\partial A]$ with respect to the three generators $a,b,c$ of $\pi_1(B^3-N^\circ(\epsilon),x_0)$.\\
  
  Consider the sequence of arc types in the standard diagram which carries $\partial A$. Then, we note that a subarc of $\partial A$ which is carried by one of the arcs in the standard diagram meets at most two times with $r_1\cup r_2\cup r_3$. The type $2$ and $11$ meet only $r_1$, the type $7$  meets only $r_2$ and the type $8$ meets only $r_3$, but the type $10$ meets $r_1$ and $r_3$ to have $(ac^{-1})^{\pm 1}$ with respect to the generators.\\

  We note that the arc type $2$ carries $\partial A$ since $m_2>0$. So, we take the base point $x_0$ in $\partial E_1''$ as in Figure~\ref{H3} and start with a subarc of $\partial A$ which is carried by the arc type $2$.  Then we have $a^{\pm 1}$ with respect to the generators.  Let $i_1$ be the upper arc type $i$ and $i_2$ be the lower arc type $i$ as in Figure~\ref{H3}.
  We note that if we show that $[\partial A]$ with the given orientation is not trivial then $[\partial A]$ with the other orientation is also not trivial.\\
  
 Now,  we give an orientation to each arc type $i$ as in Figure~\ref{H3}. We say an arc type $\bar{i}$ if the arc type $i$ has the opposite oriention of $i$.\\

 We define a $\emph{path}$ $p$ of $\partial A$ so that  $p\subset \partial A-E_1''$ and $|p\cap E_1''|=2$. We note that a path $p$ can break $E_2''$ or $E_3''$.
 A path $p$ is carried by some arc types. So, a path $p$ is represented by an arc type or a sequence of two types. For example, $2$ and $9~\bar{8_1}$ stand for paths. We also note that each path generates non-trivial element with respect to the generators.\\
  
  Now, we consider all the cases for consecutive paths of length 2 to check whether or not there is a cancellation between two paths as follows, where the first element of each pair means  consecutive paths of length 2 and the second element of each pair means the element for the given paths of length 2 with respect to the generators.\\
  
\begin{enumerate}
\item  $(2|2, a^{-1}a^{-1})$, $(2|8_i~\bar{9},a^{-1}c)$, $(2|9~\bar{8_i}, a^{-1}c^{-1})$, $(2|9~10, a^{-1}c^{-1}a^{-1})$ for $i=1,2$
\vskip 10pt
\item  $(\bar{2}|\bar{2},aa)$, $(\bar{2}|\bar{7_i}~6,ab^{-1})$,  $(\bar{2}|\bar{6}~7_i,ab)$ for $i=1,2$
\vskip 10pt
\item  $(\bar{6}~7_i|2,ba^{-1})$, $(\bar{6}~7_i|8_j~\bar{9},bc)$, $(\bar{6}~7_i|9~\bar{8_j},bc^{-1})$, $(\bar{6}~7_i|9~10,bc^{-1}a^{-1})$, $(\bar{6}~7_i|8_2~11,bca^{-1})$ for $i,j=1,2$
\vskip 10pt
\item  $(\bar{7_i}~6|2,b^{-1}a^{-1})$, $(\bar{7_i}~6|8_j~\bar{9},b^{-1}c)$, $(\bar{7_i}~6|9~\bar{8_j},b^{-1}c^{-1})$, $(\bar{7_i}~6|9~10,b^{-1}c^{-1}a^{-1})$, $(\bar{7_i}~6|8_2~11,b^{-1}ca^{-1})$ for $i,j=1,2$
\vskip 10pt
\item  $(8_1~\bar{9}|\bar{2},ca)$, $(8_1~\bar{9}|\bar{6}~7_i,cb)$,  $(8_1~\bar{9}|\bar{7_i}~6,cb^{-1})$, 
$(8_1~\bar{9}|\bar{10}~ \bar{9},cac)$ for $i=1,2$
\vskip 10pt
\item  $(8_2~\bar{9}|\bar{2},ca)$, $(8_2~\bar{9}|\bar{6}~7_i,cb)$,  $(8_2~\bar{9}|\bar{7_i}~6,cb^{-1})$, 
$(8_2~\bar{9}|\bar{10}~ \bar{9},cac)$,  $(8_2~\bar{9}|\bar{10}~ 11,caca^{-1})$, $(8_2~\bar{9}|\bar{11}~ 10,cac^{-1}a^{-1})$,  $(8_2~\bar{9}|\bar{11}~\bar{8_2},cac^{-1})$ for $i=1,2$
\vskip 10pt

\item  $(8_2~ 11|8_2~11,ca^{-1}ca^{-1})$, $(8_2~ 11|9~\bar{8_2}, ca^{-1}a^{-1})$

\vskip 10pt

\item  $(9~\bar{8_1}|\bar{2},c^{-1}a)$, $(9~\bar{8_1}|\bar{6}~7_i,c^{-1}b)$,  $(9~\bar{8_1}|\bar{7_i}~6,c^{-1}b^{-1})$,
$(9~\bar{8_1}|\bar{10}~ \bar{9},c^{-1}ac)$ for $i=1,2$
\vskip 10pt
\item  $(9~\bar{8_2}|\bar{2},c^{-1}a)$, $(9~\bar{8_2}|\bar{6}~7_i,c^{-1}b)$,  $(9~\bar{8_2}|\bar{7_i}~6,c^{-1}b^{-1})$,
$(9~\bar{8_2}|\bar{10}~ \bar{9},c^{-1}ac)$, $(9~\bar{8_2}|\bar{10}~ 11,c^{-1}aca^{-1})$,  $(9~\bar{8_2}|\bar{11}~10,c^{-1}ac^{-1}a^{-1})$,  $(9~\bar{8_2}|\bar{11}~\bar{8_2},c^{-1}ac^{-1})$ for $i=1,2$
\vskip 10pt
\item  $(9~10|8_i~\bar{9},c^{-1}a^{-1}c)$,  $(9~10|9~\bar{8_i},c^{-1}a^{-1}c^{-1})$, $(9~10|9~10,c^{-1}a^{-1}c^{-1}a^{-1})$ for $i=1,2$
\vskip 10pt
\item  $(\bar{10}~ \bar{9}|\bar{2},aca)$, $(\bar{10}~ \bar{9}|\bar{6}~7_i,acb)$, $(\bar{10}~ \bar{9}|\bar{7_i}~6,acb^{-1})$, 
$(\bar{10}~ \bar{9}|\bar{10}~ \bar{9},acac)$, $(\bar{10}~ \bar{9}|\bar{10}~ 11,acaca^{-1})$, $(\bar{10}~\bar{9}|\bar{11}~ 10,acac^{-1}a^{-1})$ for $i=1,2$
\vskip 10pt
\item  $(\bar{10}~ 11|8_2~\bar{9},aca^{-1}c)$,  $(\bar{10}~ 11|9~\bar{8_2},aca^{-1}c^{-1})$, $(\bar{10}~ 11|9~10,aca^{-1}c^{-1}a^{-1})$
\vskip 10pt
\item $(\bar{11}~10|8_2~\bar{9},ac^{-1}a^{-1}c)$,  $(\bar{11}~10|9~\bar{8_2},ac^{-1}a^{-1}c^{-1})$, $(\bar{11}~10|9~10,ac^{-1}a^{-1}c^{-1}a^{-1})$
\end{enumerate}
\vskip 20pt

We note that the path $2$ cannot be the next path of paths $9~10$, $\bar{11}~10$, $8_2~11$ and $\bar{10}~11$.  Otherwise, $\partial A$ has an infinite spiral. Similarly, the paths $\bar{10}~\bar{9}$, $\bar{10}~11$, $\bar{11}~\bar{8_2}$ and $\bar{11}~10$ cannot be the next path of the path $\bar{2}$.\\

 By considering  all the cases, we  note that there is no cancellation between two consecutive paths with respect to the generators $a,b,c$.
This implies that $[\partial A]\neq e$ since $\pi_1(B^3-N^\circ(\epsilon),x_0)$ is a free group. However, $[\partial A]=e$ since $A$ is an essential disk in $B^3-\epsilon$.
This contradicts the assumption that $m_1=m_3=0$, $m_2>0$ and this completes the proof.
\end{proof}
 
\begin{Lem}\label{T82} Suppose that $\gamma'$ is a simple closed curve which is parameterized by $(p_1,q_1,0,p_2,q_2,$ $0,p_3,q_3,0)$.
If $x_{11}>0$,
then we can construct a  simple closed curve $\gamma_0$ which is parameterized by $(p_1,q_1,t_1,p_2,q_2,t_2,p_3,q_3,t_3)$ for  $t_i\in \mathbb{Z}$ as in the table below, and it bounds an essential disk in $B^3-\epsilon$ if  $\gamma'$ does. Moreover, each component of $\gamma_0\cap I'$ is carried by one of the given arc types in one of the standard diagrams.

\begin{enumerate}
\item $q_1+p_1< x_{11}+x_{13}:$ $(t_1,t_2,t_3)=(0,-1,0)$ if $p_2\neq 0$, $(t_1,t_2,t_3)=(0,0,0)$ if $p_2=0$.\\

\item $ q_1+p_1 \geq x_{11}+x_{13}:$ $(t_1,t_2,t_3)=(-1,-1,0)$ if $p_2\neq 0$, $(t_1,t_2,t_3)=(-1,0,0)$ if $p_2=0$.

\end{enumerate}

Moreover, if we have the following condition then $\gamma'$ does not bound an essential disk in $B^3-\epsilon$.
\begin{enumerate}
\item [(3)] $x_{11}+x_{13}\leq q_1+p_1<x_{11}+x_{12}+x_{13}$ and $x_{13}\geq q_1$.

\end{enumerate}
\end{Lem}

 \begin{figure}[htb]
 \begin{center}
\includegraphics[scale=.4]{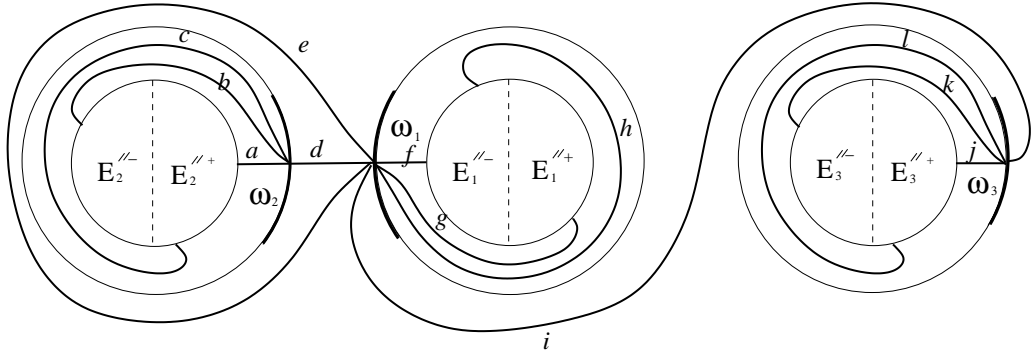}
\end{center}
\caption{The TWP diagram}
\label{H4}
 \end{figure}
 
\begin{proof}

Consider a diagram which is called the \emph{train tracks-window-pattern  diagram} or the \emph{TWP diagram} as in Figure~\ref{H4}. We note that $\gamma'$ meets $\partial E_i'$ only at  the windows $\omega_i$ for $i=1,2,3$ in the TWP diagram. Let $a,b,\cdot\cdot\cdot,l$ be the weights for the train tracks as in Figure~\ref{H4}. We also consider $a,b,\cdot\cdot\cdot,l$ as the types of arcs. Then we can get the 11 types of essential arcs in $I'$ as in the pattern diagram. (See Figure~\ref{G1}.) Then we can realize each connectivity pattern in the pattern diagram by an arc carried by this train track. For instance, the arcs for type 1 are carried by only $g-e-g$ and the arcs for type 3 are carried by $f-e-f$, $h-e-f$ or $h-e-h$.\\

   We note that $|\gamma'\cap\omega_1|=2p_1=2e+d+i$, $q_1=h$, $x_{11}=e$ and $x_{13}=i$.\\

We will now define the standard diagram by modifying the arcs carried by the train track so that they may lie outside the windows.\\

 We set $t_i=0$ if $p_i=0$.\\

\begin{figure}[htb]
\begin{center}
\includegraphics[scale=.26]{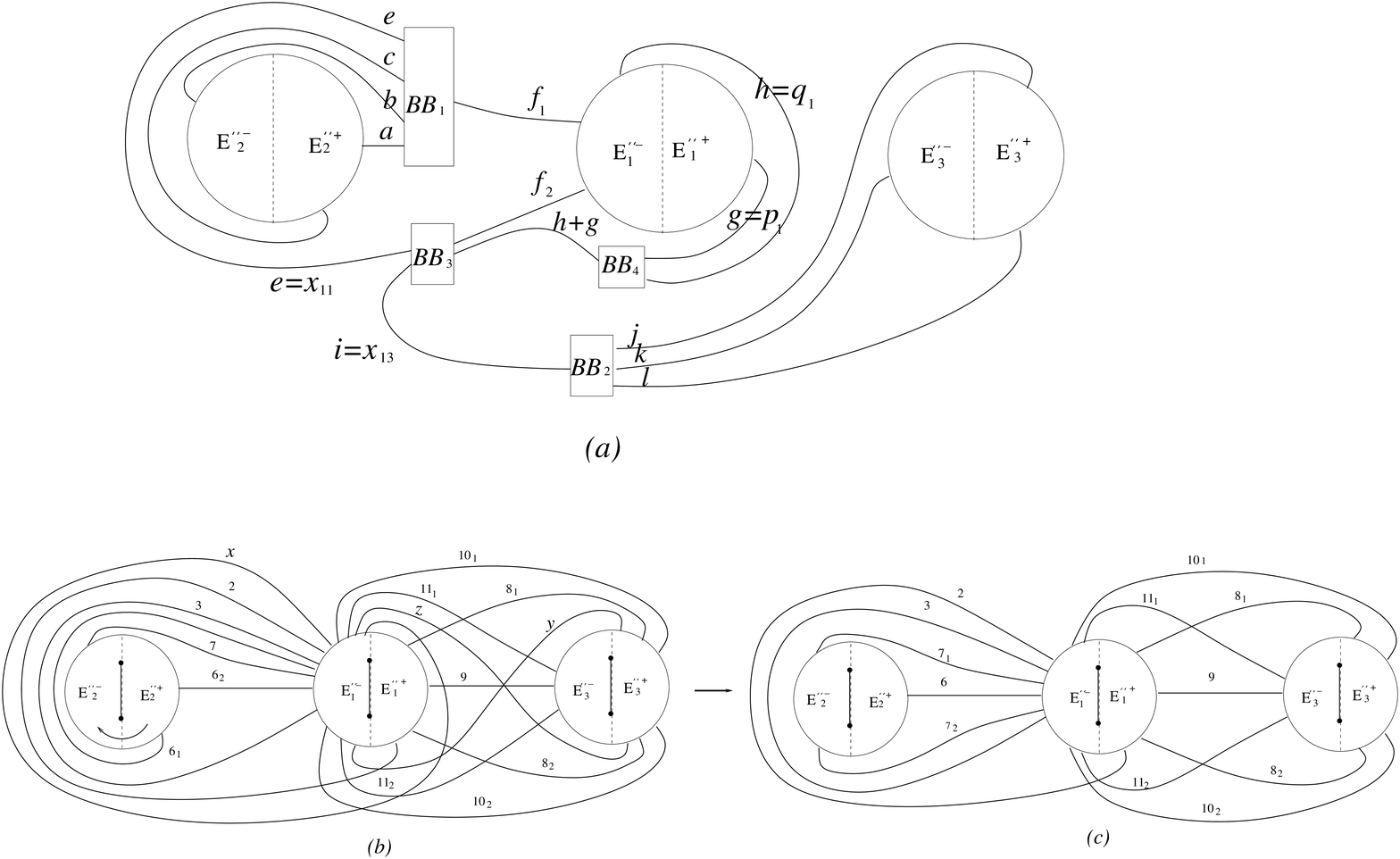}
\end{center}
\caption{}
\label{H5}
 \end{figure}

In order to consider all the possible cases, we modify the diagram of Figure~\ref{H4} into the diagram $(a)$ of Figure~\ref{H5}, Figure~\ref{H6} and Figure~\ref{H7} by the following three subcases.\\

Case 1: $q_1+p_1< x_{11}+x_{13}$. (See the  diagram $(a)$ of Figure~\ref{H5}.) We note that $q_1+p_1=h+g<e+i= x_{11}+x_{13}$.
Because of the given inequality $h+g< e+i$, we can have the black boxes $BB_1$ and $BB_3$ for the connectivities. Then we can consider the two more black boxes $BB_2$ and $BB_4$ for the rest of connectivities of subarcs. We say that the left incoming arc of $BB_i$ is $\emph{input}$ of $BB_i$ and the right incoming arc arc of $BB_i$ is $\emph{output}$ of $BB_i$. We note that $f_2>0$ since $h+g<e+i$. So, we have only one possible output of $BB_1$ which is called $f_1$.  We note that $f_1+f_2=f$. Then the following are all possible connectivities of arcs having arc types from the diagram $(a)$ of Figure~\ref{H5}, where $i_1$ is the upper type $i$ and $i_2$ is the lower type $i$.
Also, we define $m_{k_i}$ as  the weight of the types $k_1$ and $k_2$ if $\gamma$ has two different isotopy types for type $k$. Clearly, we have $m_k=m_{k_1}+m_{k_2}$.\\

We claim that the diagram $(b)$ contains  all the possible arc types which are obtained from the diagram $(a)$.   \\

\begin{enumerate}

\item We start from $E_1''^-$ with $f_1$. Then $f_1-e-(h+g)-h$ gives $x$, $f_1-e-(h+g)-g$ gives $2$, $f_1-e-f_2$ gives $3$,  $f_1-c$ gives $6_1$, $f_1-b$ gives $7$ and $f_1-a$ gives $6_2$.\\

\item We start from $E_1''^-$ with $f_2$. Then $f_2-i-j$ gives $y$, $f_2-i-k$ gives $11_2$ and $f_2-i-l$ gives $10_2$.\\

\item We start from $E_1''^-$ with $g$. Then $g-(h+g)-i-j$ gives $8_1$, $g-(h+g)-i-k$ gives $9$ and $g-(h+g)-i-l$ gives $8_2$.\\

\item We start from $E_1''^-$ with $h$. Then  $h-(h+g)-i-j$ gives $10_1$, $h-(h+g)-i-k$ gives $11_1$ and $h-i-l$ gives $z$.\\

We note that there is no subarc in $I'$ which connects $E_2''$ and $E_3''$.
Now, we exclude some cases as follows.\\

Claim 1: $x$ cannot be realized.

\begin{proof}
 If $m_{x}>0$ then $m_8=m_9=m_{10_2}=m_{11_2}=0$.  Then we note that $|\gamma'\cap (\partial E_1''\cap E_1''^+)|=m_2$. However, $|\gamma'\cap (\partial E_1''\cap E_1''^+)|\geq m_2+2m_{x}$.  Because of the connnectivity in $E_1''$, we have an inequality $m_2\geq m_2+2m_{x}$. This violates the assumption that $m_{x}>0$. Therefore $m_{x}=0$.

\end{proof}

Claim 2: $y$ and $z$ cannot be realized.

\begin{proof}
If there is an arc for $y$ then the arc for type $8_1$ is the only possible arc which connects $E_1''^+$ and $E_3''^+$ and there is no arc to connect $E_1''^+$ and $E_3''^-$. We note that $m_{8_1}<j$ if there is an arc for $y$. We have the equality $j+l=k$ for the connectivity in $E_3''$. We point out that $m_y\neq 0$ implies $m_2=0$, so the only arc entering $E_1''^+$ is $8_1$. Now, we have the inequality $m_{8_1}>l+k$ for the connectivity in $E_1''$. So,  $m_{8_1}>l+k=l+(j+l)=j+2l$. This implies that $m_{8_1}>j$.  This contradicts that $m_{8_1}<j$.\\

If there is an arc for $z$ then we note that the two arcs $2$ and $8_2$ are the only arcs that can enter $E_1''^+$. We note that $m_{8_2}<l$ if there is an arc for $z$. We still have the equality $j+l=k$. Then, we  have the inequality that $m_{8_2}>k+j=(j+l)+j=2j+l$ for the connectivity in $E_1''$. So, $m_{8_2}>l$. This contradicts that $m_{8_2}<l$.
\end{proof}

\end{enumerate}

Now, we set $t_1=0$ and $t_2=-1$ if $p_2\neq 0$ by applying a half Dehn twist supported on $E_2''$ clockwise to have $\gamma_0$ in the diagram $(c)$ of Figure~\ref{H5}. We set $t_2=0$ if $p_2=0$. Also, we set $t_3=0$. We note that the numbers on the diagram $(c)$ does not match with the numbers of the diagram $(b)$ in Figure 28 since the numbers came from the pattern diagram.
Then, we can check that every component of $\gamma_0\cap I'$ is isotopic to one of the arcs in the standard diagram~\ref{H1} $(1)$.\\

We note that  $\gamma'$ bounds a disk in $B^3-\epsilon$ if and only if $\gamma_0$ bounds a disk in $B^3-\epsilon$ by Lemma 6.1.\\

 \begin{figure}[htb]
 \begin{center}
\includegraphics[scale=.26]{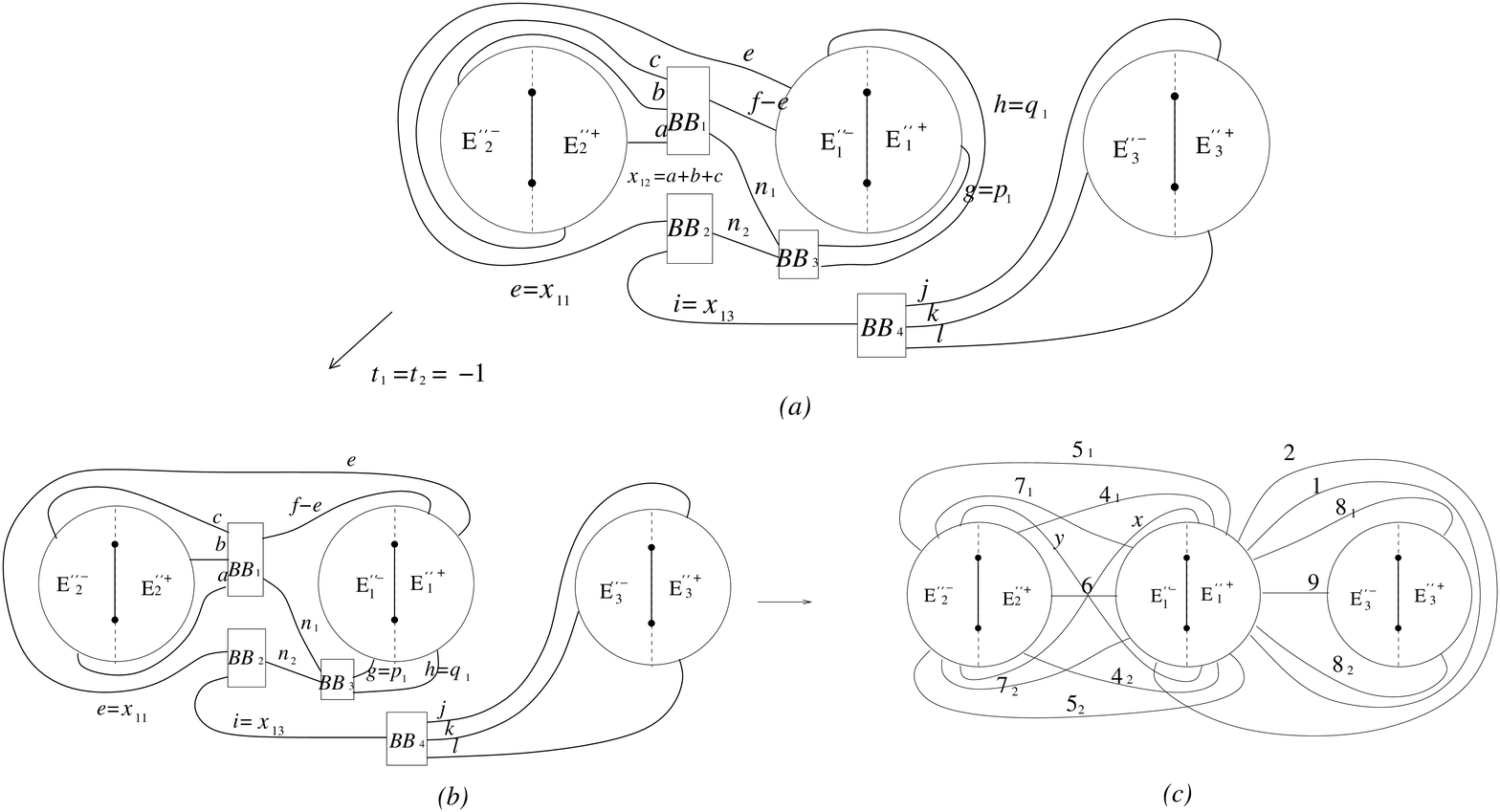}
\end{center}
\caption{}
\label{H6}
 \end{figure}

Case 2: $x_{11}+x_{13}\leq q_1+p_1< x_{11}+x_{12}+x_{13}$. (See the  diagram $(a)$ of Figure~\ref{H6}.) First, in the diagram $(a)$ we note that  $m_1=0$.\\

Because of the given inequalities $x_{11}+x_{13}\leq q_1+p_1< x_{11}+x_{12}+x_{13}$, we need to have the three black boxes $BB_1,BB_2$ and $BB_3$, where $n_1\geq 0$ and $f-e>0$ since $x_{12}>0$. We note that the sum of the inputs of $BB_1$ and $BB_2$ is $x_{11}+x_{12}+x_{13}$. All the output of $BB_2$ should be the input of $BB_3$ since $x_{11}+x_{13}\leq q_1+p_1$. \\

 Now, we  get the diagram $(b)$ of Figure~\ref{H6} by setting $t_1=t_2=-1$ 
 if $p_2\neq 0$. If $p_2=0$ then set $t_2=0$.  Set $t_3=0$.\\

We claim that the diagram $(c)$ of Figure~\ref{H6} contains all the possible arc types which can be obtained from the diagram $(b)$.\\

The following are all possible connectivities of arcs having arc types from the diagram $(b)$ of Figure~\ref{H6}.\\

\begin{enumerate}
\item We start from $E_1''^+$ with $e$. Then  $e-n_2-g$ gives $2$ and $e-n_2-h$ gives $1$.\\

\item We start from $E_1''^+$ with $f-e$. Then   $(f-e)-c$ gives $5_1$, $(f-e)-b$ gives $4_1$ and $(f-e)-a$ gives $x$.\\

\item We start from $E_1''^-$ with $g$. Then $g-n_1-c$ gives $7_1$, $g-n_2-b$ gives $6$ and $g-n_1-a$ gives $7_2$. We note that $g-n_2$ cannot take $i$ since $x_{13}<q_1$.\\

\item We start from $E_1''^+$ with $h$. Then $h-n_1-c$ gives $y$, $h-n_1-b$ gives $4_2$, $h-n_1-a$ gives $5_2$, $h-n_2-i-j$ gives $8_1$, $h-n_2-i-k$ gives $9$ and $h-n_2-i-l$ gives $8_2$.

\end{enumerate}

Claim: $x$ and $y$ cannot be realized.
\begin{proof}
Assume that there is a subarc to realize $x$. Then, there is no arc to connect $E_1''^-$ and $E_2''^+$.
So, $m_{4_1}>m_{5}+m_{7_2}$ for the connectivity in $E_2''$ since  $m_x>0$.
Also, we have an inequality $m_{7_2}+m_{2}\geq m_{2}+m_{4_1}+m_{5}$ for the connectivity in $E_1''$. By combining these two inequalities, we have $m_{4_1}>m_{5}+m_{7_2}\geq m_{5}+(m_{4_1}+m_{5})=m_{4_1}+2m_{5}$.
However, this is impossible to satisfy. Therefore, the assumption fails.\\

Now, assume that there is a subarc to realize $y$. Then, there is no arc to connect $E_1''^-$ and $E_2''^+$.
So, $m_{4_2}>m_{5}+m_{7_1}$ for the connectivity in $E_2''$ since  $m_y>0$.
Also, we note that the arc $7_1$  is the only arc that can enter $E_1''^-$. This makes an inequality $m_{7_1}> m_{4_2}+m_{5}$ for the connectivity in $E_1''$. By combining these two inequalities, we have $m_{4_2}>m_{5}+m_{7_1}\geq m_{5}+(m_{4_2}+m_{5})=m_{4_2}+2m_{5}$.
However, this is impossible to satisfy. Therefore, the assumption fails.

\end{proof}

We note that if $x_{13}\geq q_1$, then  $m_1=0$ because in the diagram $(b)$ every subarc of $\gamma'$ carried by $e$  starts from $E_1''^+$ and ends at $E_1''^-$ in the diagram $(b)$ in Figure~\ref{H6}. Also, $m_3=0$ in the diagram $(b)$ since $m_1=0$ in the diagram $(a)$. Therefore, $m_1+m_3=0$. In this case, $\gamma'$ cannot bound an essential disk in $B^3-\epsilon$ by Lemma~\ref{T81}.\\

Thus, we can check that every component of $\gamma_0\cap I'$ is isotopic to one of the arcs in the standard diagram~\ref{H1} $(2)$.\\

\begin{figure}[htb]
\begin{center}
\includegraphics[scale=.3]{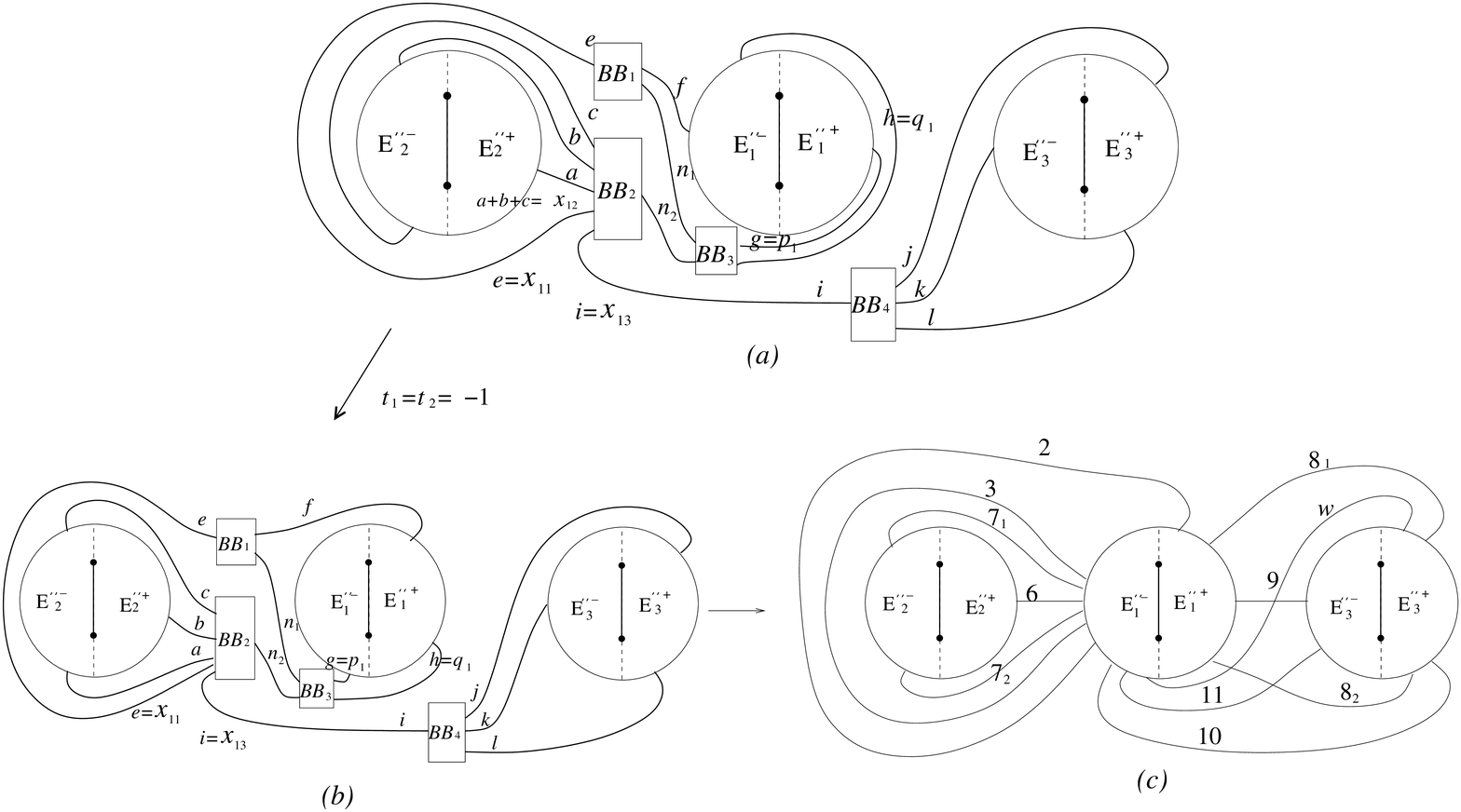}
\end{center}
\caption{}
\label{H7}
 \end{figure}

Case 3:  $x_{11}+x_{12}+x_{13}\leq q_1+p_1$. (See the diagram $(a)$ of Figure~\ref{H7}.)  We have the black box $BB_2$ for the given inequality. We note the sum of the inputs of $BB_2$ is $x_{11}+x_{12}+x_{13}$ and the sum of the outputs of $BB_3$ is $q_1+p_1$. We note that $n_2=x_{11}+x_{12}+x_{13}$.\\

 Then, we set $t_1=-1$,  and we set $t_2=-1$ if $p_2>0$ and $t_2=0$ if $p_2=0$. Also we set $t_3=0$ to have the diagram $(b)$ of Figure~\ref{H7}.\\


Then we have two subcases as follows. We refer to the diagram $(b)$.\\

Subcase 1: If $x_{13}\geq q_1$ then we note that $m_1=0$ since every arc which starts from $E_1''^+$ with $f$ should end at $E_1''^-$. \\

Now, we claim that the diagram $(c)$ of Figure~\ref{H7} contains all the possible arc types which are obtained from the diagram $(b)$.\\

 The following are  the connectivities of arcs having arc types from the diagram $(b)$ of Figure~\ref{H7}.\\
 
 \begin{enumerate}
 \item We start from $E_1''^+$ with $f$. Then $f-e-n_2-g$ gives $2$. We note that $m_1=0$ implies $f-e-n_2-h$ does not occur.
 
 \item We start from $E_1''^-$ with $g$. Then $g-n_1-e-n_2-g$ gives $3$, $g-n_2-c$ gives $7_1$, $g-n_2-b$ gives $6$, $g-n_2-a$ gives $7_2$, $g-n_2-i-j$ gives $w$, $g-n_2-i-k$ gives $11$ and $g-n_2-i-l$ gives $10$. We note that $g-n_1-e-n_2-h$ does not exist since every arc starting from $E_1''^-$ with $g-n_1-e-n_2$ should connect to $g$  since $x_{13}=i\geq h= q_1$.\\

 \item We start from $E_1''^+$ with $h$. Then $h-n_2-i-j$ gives $8_1$, $h-n_2-i-k$ gives $9$ and $h-n_2-i-l$ gives $8_2$.\\
 
  Claim: $w$ cannot be realized.
 \begin{proof}
 Suppose that there is a subarc to realize $m_w>0$. Then, we note that the arcs $2$ and $8_1$  are the only two arcs that can enter $E_1''^+$. This makes an inequality $m_{8_1}+m_2> m_{2}+m_{10}+m_{11}+m_w$ for the connectivity in $E_1''$ since $m_3>0$ by Lemma~\ref{T81}. This implies that $m_{8_1}>m_{10}+m_{11}+m_w$. We also have an equality $m_w+m_{8_1}+m_{10}=m_{11}$ for the connectivity in $E_3''$. By combining this inequality and equality, we have $m_{11}=m_w+m_{8_1}+m_{10}>m_w+m_{10}+m_{11}+m_w+m_{10}>m_{11}$ since $m_w>0$.
However, this is impossible to satisfy. Therefore, the assumption fails.
 \end{proof}
 
 Thus, we can check that every component of $\gamma_0\cap I'$ is isotopic to one of the arcs in the standard diagram~\ref{H1} $(1)$.\\
 \end{enumerate}

\begin{figure}[htb]
\begin{center}
\includegraphics[scale=.35]{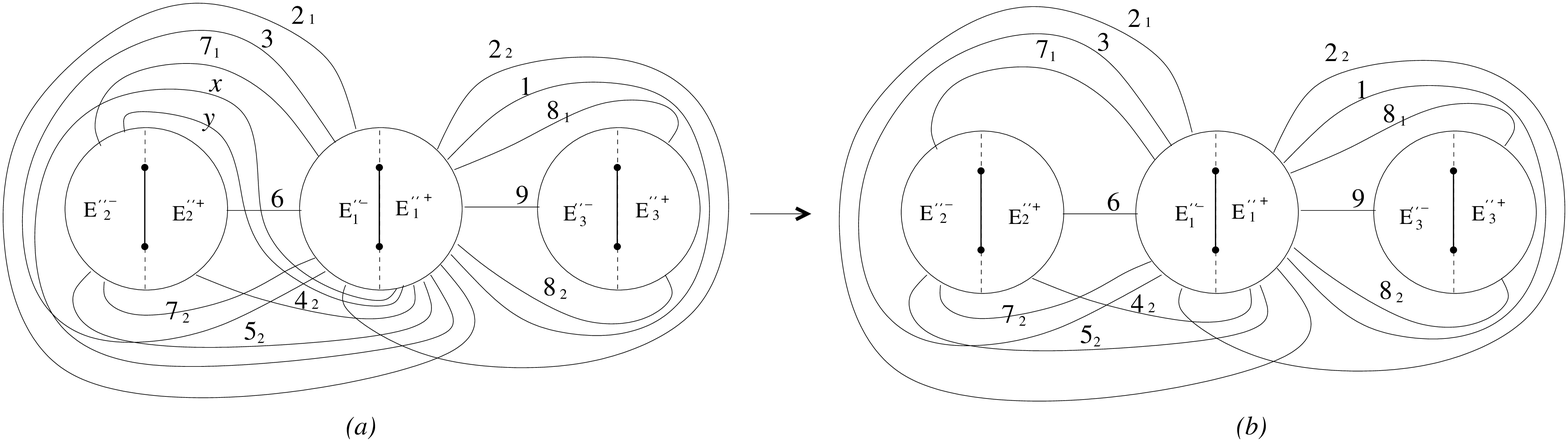}
\caption{}
\label{H8}
\end{center}
\end{figure}

Subcase 2: If $x_{13}<q_1$ then $m_1>0$ since $e>0$ because $x_{11}>0$.\\

The following are the connectivities of arcs having arc types from the diagram $(b)$ of Figure~\ref{H7}. We claim that the diagram $(b)$ of Figure~\ref{H8} contains all the possible arc types.\\

\begin{enumerate}
\item We start from $E_1''^+$ with $f$. Then $f-e-n_2-h$ gives $1$ and $f-e-n_2-g$ gives $2_2$.\\

\item We start from $E_1''^-$ with $g$. Then $g-n_1-e-n_2-h$ gives $2_1$, $g-n_1-e-n_2-g$ gives $3$, $g-n_2-c$ gives $7_1$, $g-n_2-b$ gives $6$ and $g-n_2-a$ gives $7_2$. \\

We note that $g-n_2-i-j$, $g-n_2-i-k$ and $g-n_2-i-l$ do not exist since no arc starting from $E_1''^-$ with $g-n_2$ can go to $i$ since $i<h$.\\

\item We start from $E_1''^+$ with $h$. Then $h-n_1-e-n_2-h$ gives $x$, $h-n_2-c$ gives $y$, $h-n_2-b$ gives $4_2$, $h-n_2-a$ gives $5_2$, $h-n_2-i-j$ gives $8_1$, $h-n_2-i-k$ gives $9$ and $h-n_2-i-l$ gives $8_2$.\\

Claim: $x$ and $y$ cannot be realized.

\begin{proof}
First, assume that there is a subarc to realize $m_x>0$. Then we note that $2_1$ is the only possible arc which can enter to the $E_1''^-$. Because of the connectivity in $E_1''$, we have an inequality $m_{2_1}\geq m_{2_1}+m_x$. This implies that $m_{2_1}>m_{2_1}$ since $x>0$. However, this is impossible to satisfy. Therefore, $m_x=0$.\\

Now, assume that there is a subarc to realize $m_y>0$. Then we note that $4_2$ is the only possible arc which can enter to the $E_2''^+$.  We have an equality $m_{4_2}=m_y+m_{7_1}+m_{5_2}$ for the connectivity in $E_2''$. Also, we have an inequality $m_{7_1}\geq m_y+m_{4_2}+m_{5_2}$ for the connectivity in $E_1''$. Therefore, we have $m_{7_1}\geq m_y+m_{4_2}+m_{5_2}=m_y+(m_y+m_{7_1}+m_{5_2})+m_{5_2}>m_{7_1}$ since $m_y>0$. This is impossible to satisfy. Therefore, $m_y=0$.

\end{proof}
\end{enumerate}
Thus, we can check that every component of $\gamma_0\cap I'$ is isotopic to one of the arcs in the standard diagram~\ref{H1} $(2)$.\\

 By combining cases 1, 2 and 3, we note that  $\gamma'$ can be modified as a  simple closed curve $\gamma_0$  so that all components of $\gamma_0\cap I'$ are isotopic to  arcs of the two diagrams in Figure~\ref{H1}, but $\gamma_0$ bounds an essential disk  if $\gamma'$ does by Lemma 6.1.

\end{proof}

We represent $\gamma_0$ by putting weights on these arcs.\\

Now, we want to calculate $m_i$ of $\gamma_0$ for $i=1,2,...,11$ by using the parameters $p_j,q_j$ of $\gamma'$ for $j=1,2,3$.\\
 
 The following is the table of $m_i$ for each case.\\
 
 \begin{Lem}\label{T83}
 
 \begin{enumerate}
 
 \item $x_{11}+x_{13}>q_1+p_1$:\\
 
  $m_3=p_3-p_2-q_1$, $m_{2}=q_1+p_1-2p_3$, $m_{7_1}=q_2$, $m_{7_2}=p_2-q_2$, $m_6=p_2$, $m_{10_1}=\min(q_1,p_3-q_3)$, $m_{10_2}=\min(0,\max(2p_3-p_1-q_1,0),q_3)$, $m_{11_1}=\max(0,q_1-p_3+q_3)$, $m_{11_2}=\max(0,\max(2p_3-p_1-q_1,0)-q_3)$, $m_{8_1}=\max(0,p_3-q_3-p_1)$, $m_{8_2}=\max(0,q_3-\max(2p_3-p_1-q_1,0))$, $m_9=m_{10}+m_8-m_{11}.$\\
 
 \item $x_{11}+x_{13}\leq q_1+p_1$ and $x_{13}<q_1$:\\
 
 $(a)$ $q_1>x_{13}+x_{11}$: $m_1=p_1-p_2-p_3$, $m_{8_1}=p_3-q_3,~ m_{8_2}=q_3,~m_9=p_3$, $m_{5_1}=\min(q_2,p_2+p_3-q_1)$, $m_{7_1}=q_2-m_{5_1},~ m_{4_1}= p_2+p_3-q_1-m_{5_1}$, $m_{5_2}=\min(p_2-q_2,q_1+p_2-p_1-p_3),~ m_{7_2}=p_2-q_2-m_{5_2},~ m_{4_2}=q_1+p_2-p_1-p_3-m_{5_2}$, $m_6=m_5+m_7-m_4$.\\

$(b)$ $x_{13}<q_1\leq x_{13}+x_{11}$: $m_1=q_1-2p_3, ~m_{8_1}=p_3-q_3, ~m_{8_2}=q_3,~ m_9=p_3$, $m_{2}=p_1+p_3-p_2-q_1$, $m_{5_1}=\min(q_2,p_2+p_3-q_1),~m_{7_1}=q_2-m_{5_1},~ m_{4_1}=p_2+p_3-q_1-m_{5_1}, m_{7_2}=p_2-q_2$, $m_6=m_5+m_7-m_4$. \\

\item $x_{11}+x_{12}+x_{13}\leq q_1+p_1$:\\

\begin{enumerate}

\item   $x_{13}\geq q_1$:\\

$m_3=q_1-p_2-p_3$, $m_2=p_1-q_1$, $m_{7_1}=q_2$, $m_{7_2}=p_2-q_2,$ $m_6=p_2$,
$m_{8_1}=\min(q_1,p_3-q_3)$, $m_{8_2}=\max(q_3-(2p_3-q_1),0),$
$m_{10}=\min(2p_3-q_1,q_3)$,
$m_{11}=\max((2p_3-q_1)-q_3,0)$.\\

\item [(b1)] $x_{13}<q_1$, $p_1\geq 2q_1-2p_3$: $m_{2_2}=p_1+2p_3-2q_1$, $m_{1}=q_1-2p_3$, $m_3=q_1-p_2-p_3$, $m_{7_1}=q_2$, $m_{7_2}=p_2-q_2$, $m_6=p_2$, $m_{8_1}=p_3-q_3$, $m_{8_2}=q_3$ and $m_9=p_3$.\\

\item [(b2)] $x_{13}<q_1$, $p_1< 2q_1-x_{13}$:\\

\begin{enumerate}
\item $q_1>x_{13}+x_{11}$: $m_{2_1}=q_1-p_2-p_3$, $m_{1}=p_1-q_{1}$,  $m_{7_1}=q_2$, $m_{5_2}=\min(p_2-q_2,q_1-x_{11}-x_{13})$, $m_{7_2}=p_2-q_2-m_{5_2}=p_2-q_2-\min(p_2-q_2,q_1+p_2-p_1-p_3)$, $m_{4_2}=q_1+p_2-p_1-p_3-m_{5_2}=q_1+p_2-p_1-p_3-\min(p_2-q_2,q_1+p_2-p_1-p_3)$,  $m_6=p_2-m_{4_2}=p_1+p_3-q_1+\min(p_2-q_2,q_1+p_2-p_1-p_3)$, $m_{8_1}=p_3-q_3$, $m_{8_2}=q_3$ and $m_9=p_3$.\\

\item $q_1\leq x_{13}+x_{11}$: $m_{2_1}=2q_1-p_1-2p_3$, $m_{1}=p_1-q_{1}$, $m_3=p_1+p_3-p_2-q_1$, $m_{7_1}=q_2$, $m_{7_2}=p_2-q_2$, $m_6=p_2$, $m_{8_1}=p_3-q_3$, $m_{8_2}=q_3$ and $m_9=p_3$.
\end{enumerate}

\end{enumerate}
 \end{enumerate}
 
\end{Lem}

\begin{figure}[htb]
\begin{center}
 \includegraphics[scale=.35]{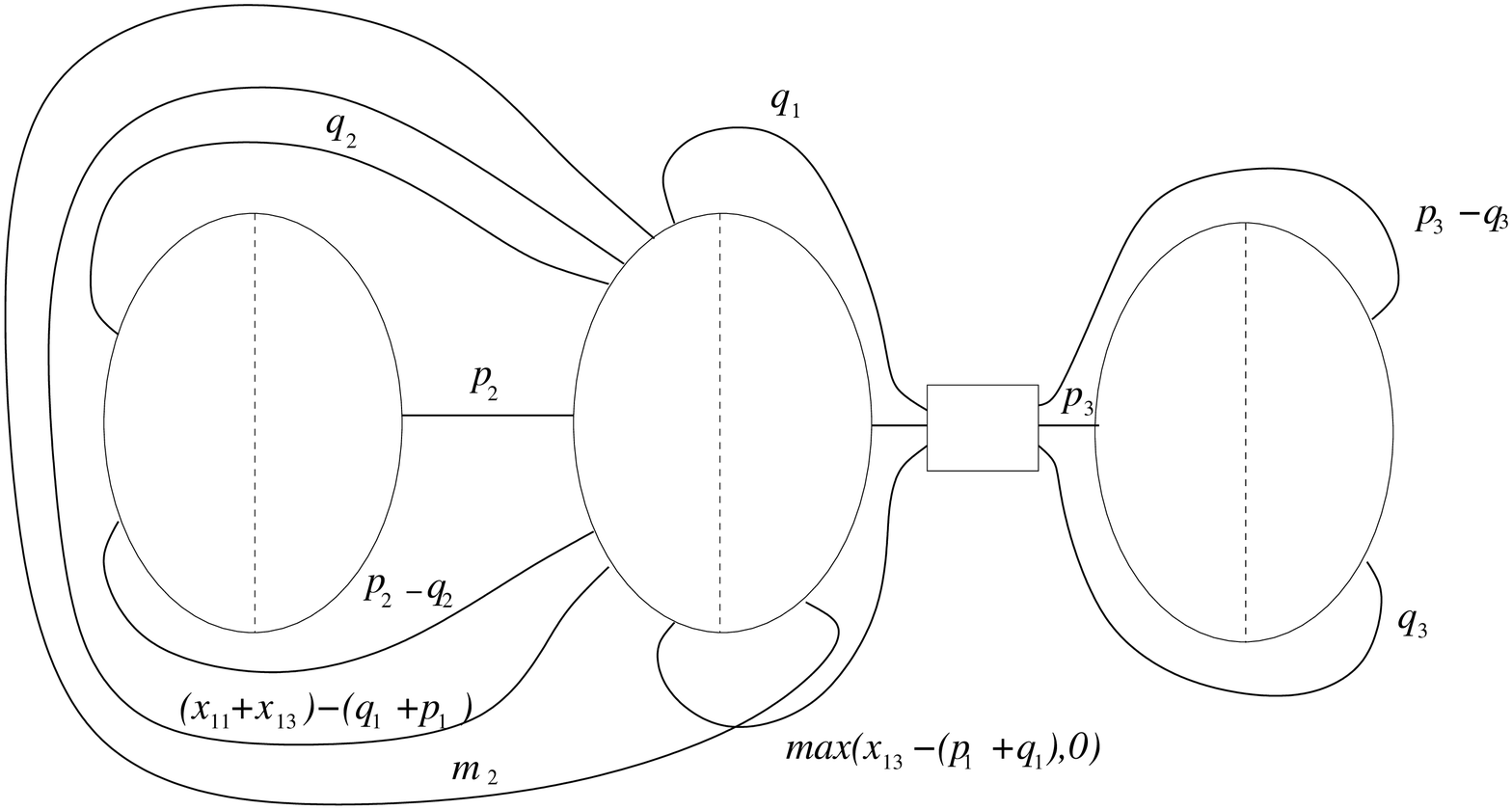}
 \end{center}
 \caption{}
 \label{H9}
 \end{figure}
 
 \begin{proof}
 
 Consider $\gamma_0$ which is in standard position.\\
 
 First, we recall that $x_{ij}=I_j,~ x_{ik}=I_k$ and $x_{ii}={I_i-I_j-I_k\over 2}$ if $x_{11}>0$ by the equations in Lemma~\ref{T51}. Also, we note that $I_1=2p_1,I_2=2p_2$ and $I_3=2p_3$. So, we can have $x_{11}=p_1-p_2-p_3$, $x_{12}=2p_2$ and $x_{13}=2p_3$ from the equations. \\

\underline{Case 1}: $x_{11}+x_{13}>q_1+p_1$.  Refer to the diagrams of Figure~\ref{H5} and Figure~\ref{H9} to get the weights  $m_i$  as follows.\\

First, we note the weight of the bottom arc starting from  $E_1''^-$ is $\max(x_{13}-(p_1+q_1),0)$ as in Figure~\ref{H9} since if $m_2>0$ then the weight is zero and $x_{13}-(p_1+q_2)<0$, and if $m_2=0$ then the weight is $x_{13}-(p_1+q_1)$ because the weight starting from $E_1''^+$ is $p_1$ in this case.\\

 Then we have $m_3=f_2=(e+i)-(h+g)=(x_{11}+x_{13})-(q_1+p_1)$, $m_2=x_{11}-m_3=(q_1+p_1)-x_{13}$, $m_{7_1}=q_2$, $m_{7_2}=p_2-q_2,$ $m_6=p_2$.\\

We note that $m_{10_1}=\min(q_1,p_3-q_3)$, $m_{10_2}=\min(\max(x_{13}-(p_1+q_1),0),q_3)$. Then we have
$m_{11_1}=\max(0,q_1-(p_3-q_3))$ since if $q_1\geq p_3-q_3$ then $m_{11_1}=q_1-(p_3-q_3)$ and if $q_1<p_3-q_3$ then $m_{11_1}=0$. Similarly, we have $m_{11_2}=\max(0,\max(x_{13}-(p_1+q_1),0)-q_3)$. Also, we have
$m_{8_1}=\max(0,(p_3-q_3)-q_1)$ since if $ p_3-q_3\geq q_1$ then $m_{8_1}=(p_3-q_3)-q_1$ and if $p_3-q_3<q_1$ then $m_{18_1}=0$. Similarly, $m_{8_2}=\max(0,q_3-\max(x_{13}-(p_1+q_1),0))$. We note that $m_9=m_{10}+m_8-m_{11}.$\\

In order to have better formulas as in Lemma~\ref{T83}, use  $x_{11}=p_1-p_2-p_3$, $x_{12}=2p_2$ and $x_{13}=2p_3$.\\

\underline{Case 2}: $x_{11}+x_{13}\leq q_1+p_1< x_{11}+x_{12}+x_{13}$ and $x_{13}<q_1$.\\

 Then we need to consider the following two cases in Figure~\ref{H10} which are obtained from the diagram $(b)$ of Figure~\ref{H6}. We note that the first case has $m_2=0$ and the second case has $m_2>0$. By referring to the two diagrams of Figure~\ref{H10}, we can get the weights $m_i$ as follows. We will use a similar argument as in the first case to find $m_{5_1}$ and $m_{5_2}$.\\

\begin{figure}[htb]
\begin{center}
\includegraphics[scale=.25]{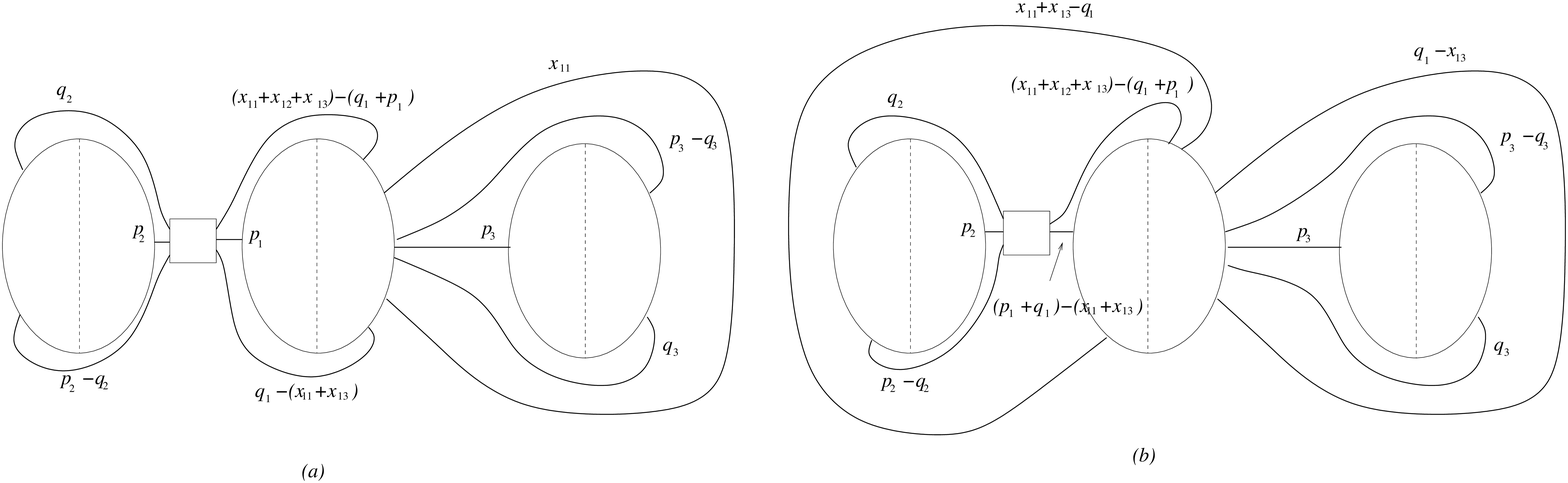}
\end{center}
\caption{}
\label{H10}
\end{figure}

$(a)$ $q_1>x_{13}+x_{11}$: $m_1=x_{11}$, $m_{8_1}=p_3-q_3, m_{8_2}=q_3,m_9=p_3$, $m_{5_1}=\min(q_2,x_{11}+x_{12}+x_{13}-(q_1+p_1))$, $m_{7_1}=q_2-m_{5_1}, m_{4_1}= x_{11}+x_{12}+x_{13}-(q_1+p_1)-m_{5_1}$, $m_{5_2}=\min(p_2-q_2,q_1-(x_{11}+x_{13})), m_{7_2}=p_2-q_2-m_{5_2}, m_{4_2}=q_1-(x_{11}+x_{13})-m_{5_2}$, $m_6=m_5+m_7-m_4$.\\

$(b)$ $x_{13}<q_1\leq x_{13}+x_{11}$: $m_1=q_1-x_{13}, m_{8_1}=p_3-q_3, m_{8_2}=q_3, m_9=p_3$, $m_{2}=x_{11}+x_{13}-q_1$; $m_{5_1}=\min(q_2,x_{11}+x_{12}+x_{13}-(q_1+p_1)),m_{7_1}=q_2-m_{5_1}, m_{4_1}=x_{11}+x_{12}+x_{13}-(q_1+p_1)-m_{5_1}, m_{7_2}=p_2-q_2,$ $m_6=m_5+m_7-m_4$.\\

In order to have better formulas as in Lemma~\ref{T83}, use  $x_{11}=p_1-p_2-p_3$, $x_{12}=2p_2$ and $x_{13}=2p_3$.\\

We recall that if $q_1\leq x_{13}$ then $\gamma'$ does not bound an essential disk in $B^3-\epsilon$.\\

\underline{Case 3}: $x_{11}+x_{12}+x_{13}\leq q_1+p_1$.  We note that $q_1+p_1<2p_1$ since $0\leq q_1<p_1$.
 Then we have two cases for this.\\
 
 \begin{figure}[htb]
 \begin{center}
 \includegraphics[scale=.3]{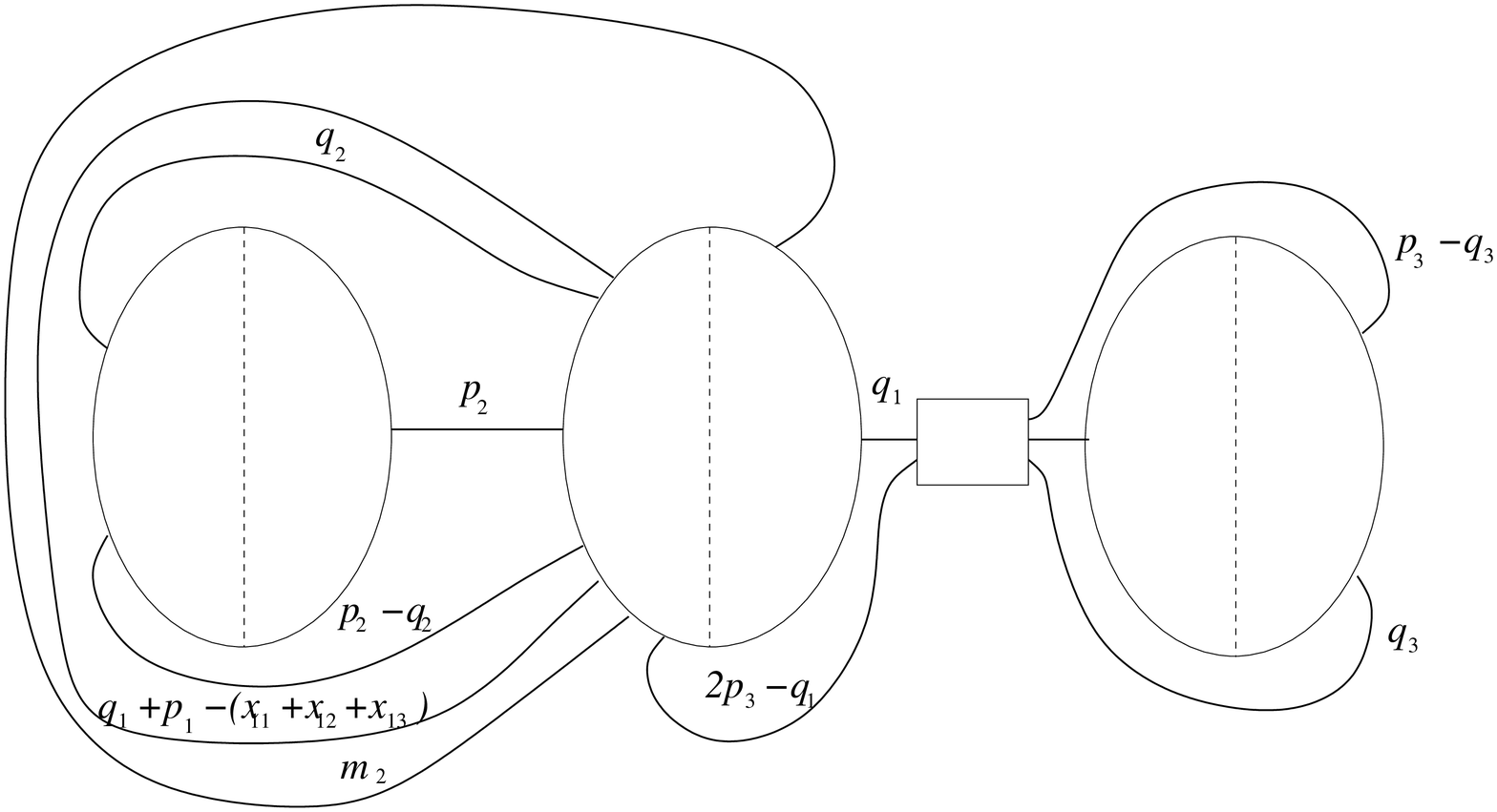}
 \end{center}
 \caption{}
 \label{H11}
 \end{figure}
 
 \begin{enumerate}
 \item[$(a)$] $x_{13}\geq q_1$: Refer to the  diagram $(c)$ of Figure~\ref{H7} and the diagram of Figure~\ref{H11}.\\
 
Then we have the following formulas for $m_j$.  We will use a similar argument as in the first case to find $m_{8_1},m_{8_2},m_{10}$ and $m_{11}$.\\

$m_3=(q_1+p_1)-(x_{11}+x_{12}+x_{13})$, $m_2=x_{11}-m_3=(2x_{11}+x_{12}+x_{13})-(q_1+p_1)$, $m_{7_1}=q_2$, $w_{7_2}=p_2-q_2,$ $w_6=p_2$,
$m_{8_1}=\min(q_1,p_3-q_3)$, $m_{8_2}=\max(q_3-(2p_3-q_1),0)$,
$m_{10}=\min(2p_3-q_1,q_3)$,
$m_{11}=\max((2p_3-q_1)-q_3,0)$.\\

\begin{figure}[htb]
\begin{center}
\includegraphics[scale=.25]{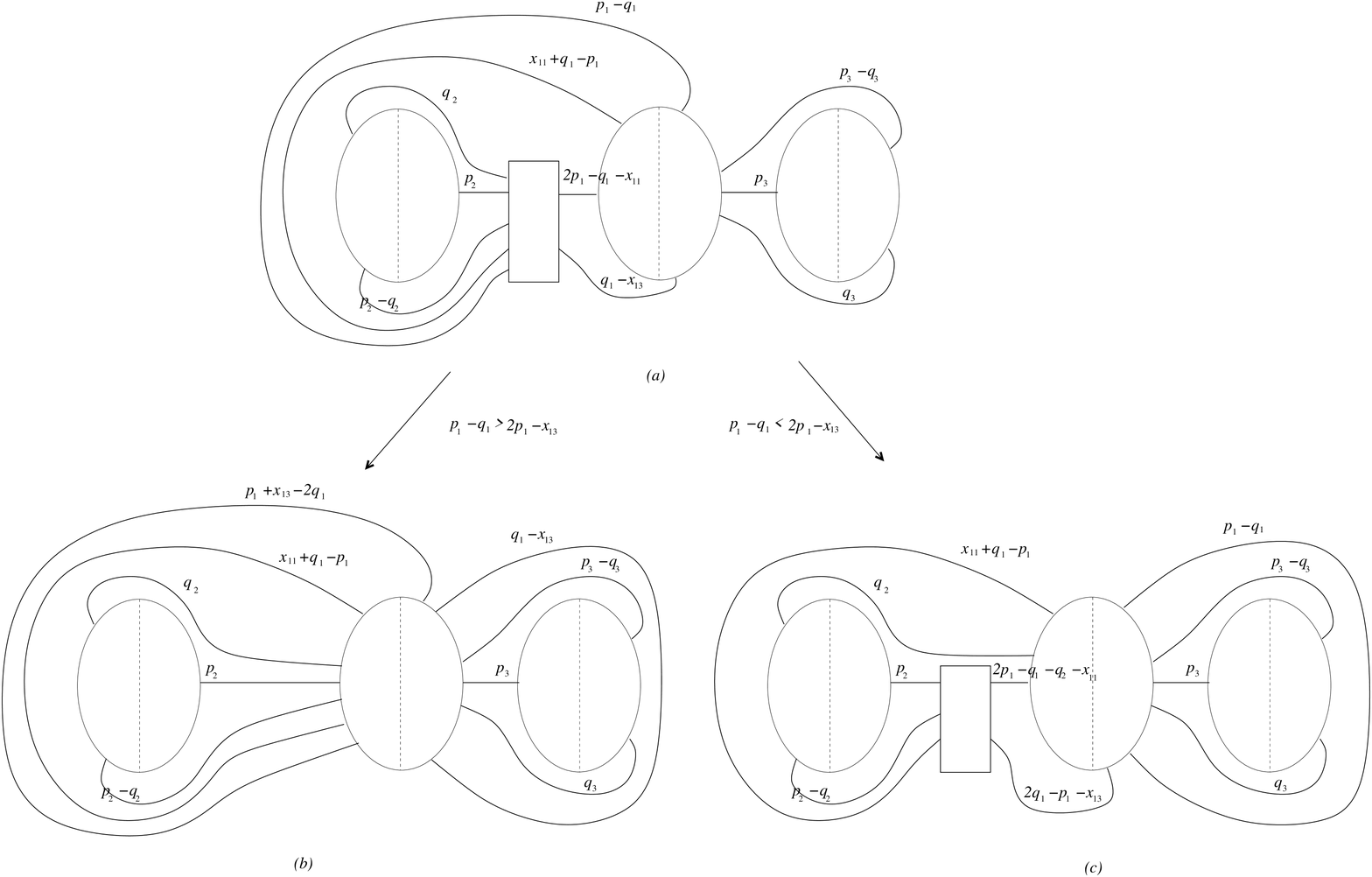}
\end{center}
\caption{}
\label{H12}
\end{figure}
\item[$(b)$] $x_{13}< q_1$:\\

We have the diagram $(a)$ of Figure~\ref{H12} by referring Figure~\ref{H7} and Figure~\ref{H8}.\\

Then we consider the two subcases  $p_1-q_1\geq q_1-x_{13}$ and $p_1-q_1< q_1-x_{13}$.\\

\item[(b1)]: If $p_1-q_1\geq q_1-x_{13}$ then we can directly get the formulas for $m_j$ as follows. (Refer to the diagram $(b)$ of Figure~\ref{H12}.)\\

$m_{2_2}=p_1+x_{13}-2q_1$, $m_{1}=q_1-x_{13}$, $m_3=x_{11}+q_1-p_1$, $m_{7_1}=q_2$, $m_{7_2}=p_2-q_2$, $m_6=p_2$, $m_{8_1}=p_3-q_3$, $m_{8_2}=q_3$ and $m_9=p_3$.\\

\item[(b2)]: If $p_1-q_1<q_1-x_{13}$ then we  have  the diagram $(c)$ of Figure~\ref{H12}.\\

Now, we consider the two subcases for this which are  $(i)~q_1> x_{13}+x_{11}$ or  $(ii)~q_1\leq x_{13}+x_{11}$.\\

\begin{figure}[htb]
\begin{center}
\includegraphics[scale=.25]{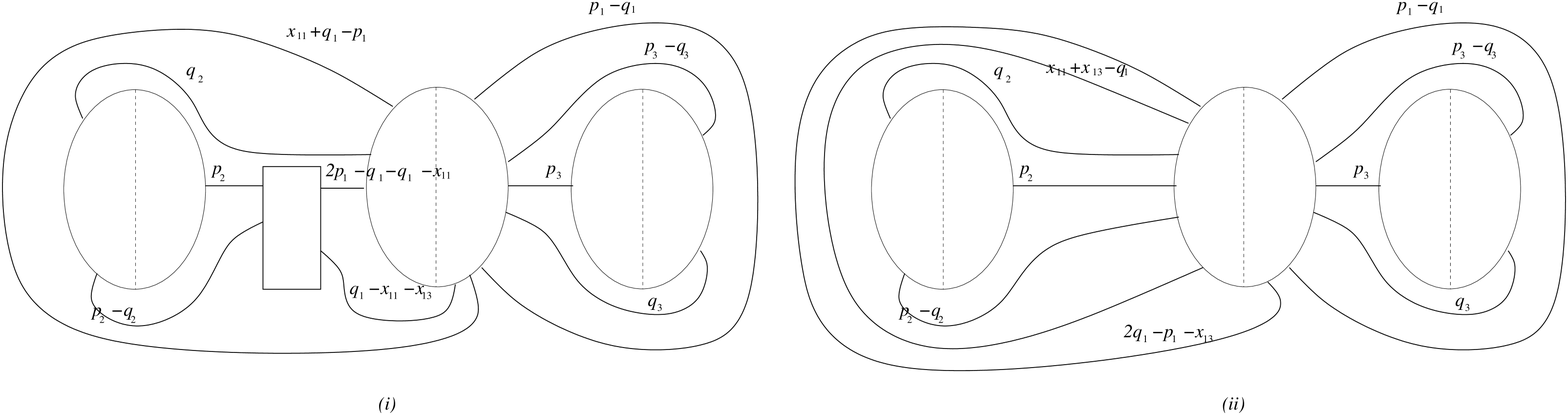}
\end{center}
\caption{}
\label{H13}
\end{figure}

Then we have the two diagrams of Figure~\ref{H13}\\

\begin{enumerate}
\item [$(i)$] $q_1>x_{13}+x_{11}$: (Refer to the diagram $(i)$ of Figure~\ref{H13}.)\\

Then we get the formulas for $m_j$ as follows.\\

$m_{2_1}=x_{11}+q_1-p_1$, $m_{1}=p_1-q_{1}$,  $m_{7_1}=q_2$, $m_{5_2}=\min(p_2-q_2,q_1-x_{11}-x_{13})$, $m_{7_2}=p_2-q_2-m_{5_2}=p_2-q_2-\min(p_2-q_2,q_1-x_{11}-x_{13})$, $m_{4_2}=q_1-x_{11}-x_{13}-m_{5_2}=q_1-x_{11}-x_{13}-\min(p_2-q_2,q_1-x_{11}-x_{13})$,  $m_6=p_2-m_{4_2}=p_2-(q_1-x_{11}-x_{13}-\min(p_2-q_2,q_1-x_{11}-x_{13}))=p_2-q_1+x_{11}+x_{13}+\min(p_2-q_2,q_1-x_{11}-x_{13})$, $m_{8_1}=p_3-q_3$, $m_{8_2}=q_3$ and $m_9=p_3$.\\

\item [$(ii)$] $q_1\leq x_{13}+x_{11}$: (Refer to the diagram $(ii)$ of Figure~\ref{H13}.)\\

Then we get the formulas for $m_j$ as follows.\\

$m_{2_1}=2q_1-p_1-x_{13}$, $m_{1}=p_1-q_{1}$, $m_3=x_{11}+x_{13}-q_1$, $m_{7_1}=q_2$, $m_{7_2}=p_2-q_2$, $m_6=p_2$, $m_{8_1}=p_3-q_3$, $m_{8_2}=q_3$ and $m_9=p_3$.\\

In order to have better formulas as in Lemma~\ref{T83}, use  $x_{11}=p_1-p_2-p_3$, $x_{12}=2p_2$ and $x_{13}=2p_3$.

\end{enumerate}
\end{enumerate}
\end{proof}

\section{Step 4: Main Theorem}

Now, we want to discuss the main theorem that can complete my algorithm. 
In order to do this, we define four homeomorphisms $(\delta_1\delta_2^{-1})^{\pm 1}$ and $\delta_3^{\pm 1}$ as follows.\\

Let $\delta_1$ and $\delta_2$ be the clockwise half Dehn twists supported on two punctured disks $C_1$ and $C_2$ respectively as in Figure~\ref{I1}.
Also, let $\delta_3$ be clockwise half Dehn twists supported on two punctured disk $E_4'$ as in Figure~\ref{I2}.
Then we have the two following lemmas.

\begin{Lem}\label{T91}
$\gamma_0$ bounds an essential disk in $B^3-\epsilon$ if and only if  $(\delta_1\delta_2^{-1})^{\pm 1}(\gamma_0)$  bounds an essential disk in $B^3-\epsilon$. 
\end{Lem}
\begin{figure}[htb]
\begin{center}
\includegraphics[scale=.45]{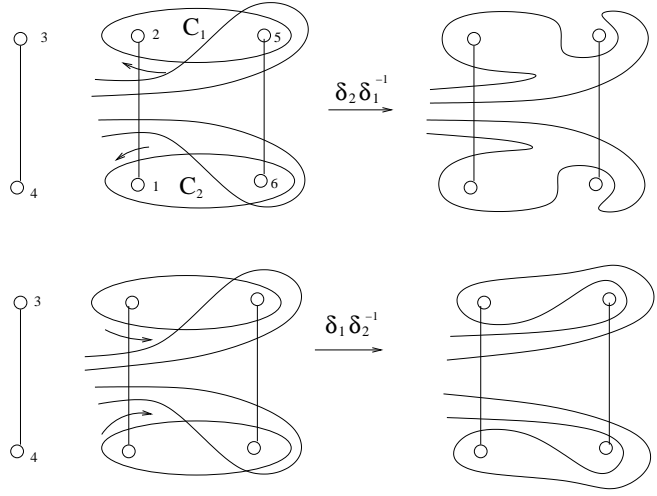}
\end{center}
\caption{The homeomorphisms $\delta_1\delta_2^{-1}$ and $\delta_2\delta_1^{-1}$}
\label{I1}
 \end{figure}
\begin{proof}
We notice that an extension $K$ to $B^3$ of $(\delta_1\delta_2^{-1})^{\pm 1}$ changes the position of two strings. So, it preseves the $\infty$ tangle.
Let $A$ be the essential disk in $B^3-\epsilon$ so that $\partial A=\gamma_0$.
Then, we know that $K(A)$ bounds a disk in $B^3-\epsilon$ and $K(\gamma_0)$ is essential in $\Sigma_{0.6}$. Therefore, $(\delta_1\delta_2^{-1})^{\pm 1}(\gamma_0)$  bounds an essential disk in $B^3-\epsilon$. 

\end{proof}

\begin{Lem}\label{T92}
 $\gamma_0$ bounds an essential disk in $B^3-\epsilon$ if and only if $\delta_3(\gamma_0)$  bounds an essential disk in $B^3-\epsilon$.
\end{Lem}

\begin{proof}
By lemma 6.1, it is trivial.
\end{proof}

\begin{figure}[htb]
\begin{center}
\includegraphics[scale=.45]{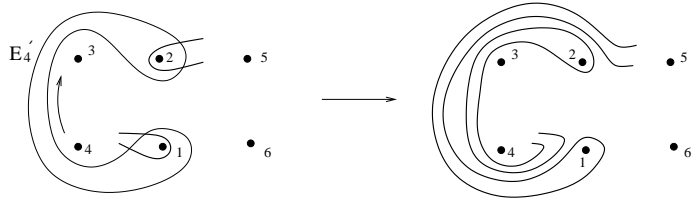}
\end{center}
\caption{The homeomorphism $\delta_3$}
\label{I2}
 \end{figure}
 
Now, let $(p_{11},q_{11},t_{11},p_{21},q_{21},t_{21},p_{31},q_{31},t_{31})$ be the parameters for $h(\gamma_0)$ for $h\in\{(\delta_1\delta_2^{-1})^{\pm 1}, \delta_3^{\pm 1}\}$. 
Also, let $(p_{11},q'_{11},p_{21},q'_{21},p_{31},q'_{31})$ be the Dehn's parameters for $h(\gamma_0)$, where $q'_{i1}=p_{i1}t_{i1}+q_{i1}$ for $i=1,2,3$.
We note that the nine parameters can be obtained from the Dehn's parameters.
Also, we know that $|\gamma_0\cap \partial E|=2(p_1+p_2+p_3)$ and $|\delta_1^{-1}\delta_2(\gamma_0)\cap \partial E|=2(p_{11}+p_{21}+p_{31})$.\\

Now, here is the main theorem.

\begin{Thm}\label{T93}
Suppose that $\gamma_0$ bounds an essential disk in $B^3-\epsilon$ and $\gamma_0$  is in standard position in $I'$ and $m_3>0$.
Then applying one of the homeomorphisms  $(\delta_1\delta_2^{-1})^{\pm 1}$ and $\delta_3^{\pm 1}$ reduces the sum of the $p_i$ for the image of $\gamma_0$. Especially, the following are the formulas for the Dehn's parameter changes for each case.\\

\begin{enumerate}
\item $m_{2_1}, m_3>0, m_1=0$ and $m_3>m_2+1$: $(p_{11},q'_{11},p_{21},q'_{21},p_{31},q'_{31})=
(p_1-2m_2,q'_1+m_2-(m_{10}+m_{11}),p_2,q_2',p_3,q_3'+2(m_{10}+m_{11}))$ by $\delta_3$.\\

\item  $m_{2_2},m_3>0, m_1=0$ and $m_3>m_2+1$: $(p_{11},q'_{11},p_{21},q'_{21},p_{31},q'_{31})=(p_1-2m_2,q'_1-m_2+(m_{10}+m_{11}),p_2,q_2',p_3,q_3'-2(m_{10}+m_{11}))$ by $\delta_3^{-1}$.\\

\item $m_3>m_2>0$:  $(p_{11},q'_{11},p_{21},q'_{21},p_{31},q'_{31})=(p_1-2m_2,q'_1+m_2,p_2,q_2',p_3,q'_3)$ by $\delta_3$.\\

\item $m_1> m_2\geq m_3>0$:  $(p_{11},q'_{11},p_{21},q'_{21},p_{31},q'_{31})=(p_1-2m_2,q'_1+(m_3-m_2),p_2,q_2',p_3,q'_3)$ by $\delta_3$.\\

\item $m_1=m_3=1$ and $m_i=0$ for all $i\neq 1,3$: It bounds an essential disk in $B^3-\epsilon$.\\

\item $m_1=m_2=0$, $m_3\geq 2$.

\begin{enumerate}
\item $m_{11}=0$: $(p_{11},q'_{11},p_{21},q'_{21},p_{31},q'_{31})=(p_1-m_8,q'_1-m_{8_1},p_2,q'_2,p_3-m_8,q'_3+m_8)$ by $\delta_1^{-1}\delta_2$.\\

\item $m_8=0$: $(p_{11},q'_{11},p_{21},q'_{21},p_{31},q'_{31})=(p_1-m_{11},q'_1-m_{11_1},p_2,q'_2,p_3-m_{11},q'_3+m_{11})$ by $\delta_1\delta_2^{-1}$.\\

\item $m_8,m_{11}>0$: \\

\begin{enumerate}
\item $m_{8_1},m_{11_2}>0$: $(p_{11},q'_{11},p_{21},q'_{21},p_{31},q'_{31})=(p_1-(m_{8_1}-m_{11_2}),q'_1-m_{8_1},p_2,q'_2,p_3-(m_{8_1}-m_{11_2}),q'_3+m_{11_2})$ by $\delta_1^{-1}\delta_2$.\\

\item $m_{8_2},m_{11_1}>0$: $(p_{11},q'_{11},p_{21},q'_{21},p_{31},q'_{31})=(p_1-(m_{8_2}-m_{11_1}),q'_1+m_{11_1},p_2,q'_2,p_3-(m_{8_2}-m_{11_1}),q'_3-m_{8_2})$ by $\delta_1^{-1}\delta_2$.
\end{enumerate}
\end{enumerate}

Also, if $\gamma_0$ satisfies  any of the following conditions then $\gamma_0$ does not bound an essential disk in $B^3-\epsilon$.\\

\item $m_1+m_3< 2$ and $m_i>0$ for some $i$.\\

\item $m_2,m_3>0, m_1=0$ and $m_3\leq m_2+1$.\\

\item $m_2\geq m_1,m_3>0$.

\end{enumerate}
\end{Thm}

\begin{proof}
Suppose that $\gamma_0$ is parameterized by $(p_1,q_1,t_1,p_2,q_2,t_2,$ $p_3,q_3,0)$  to have $\gamma_0$ which is in standard position in $I'$.
We note that  $t_2=0$ if $p_2=0$ and $t_2=-1$ if  $p_2\neq 0$.\\

\begin{figure}[htb]
\begin{center}
\includegraphics[scale=.3]{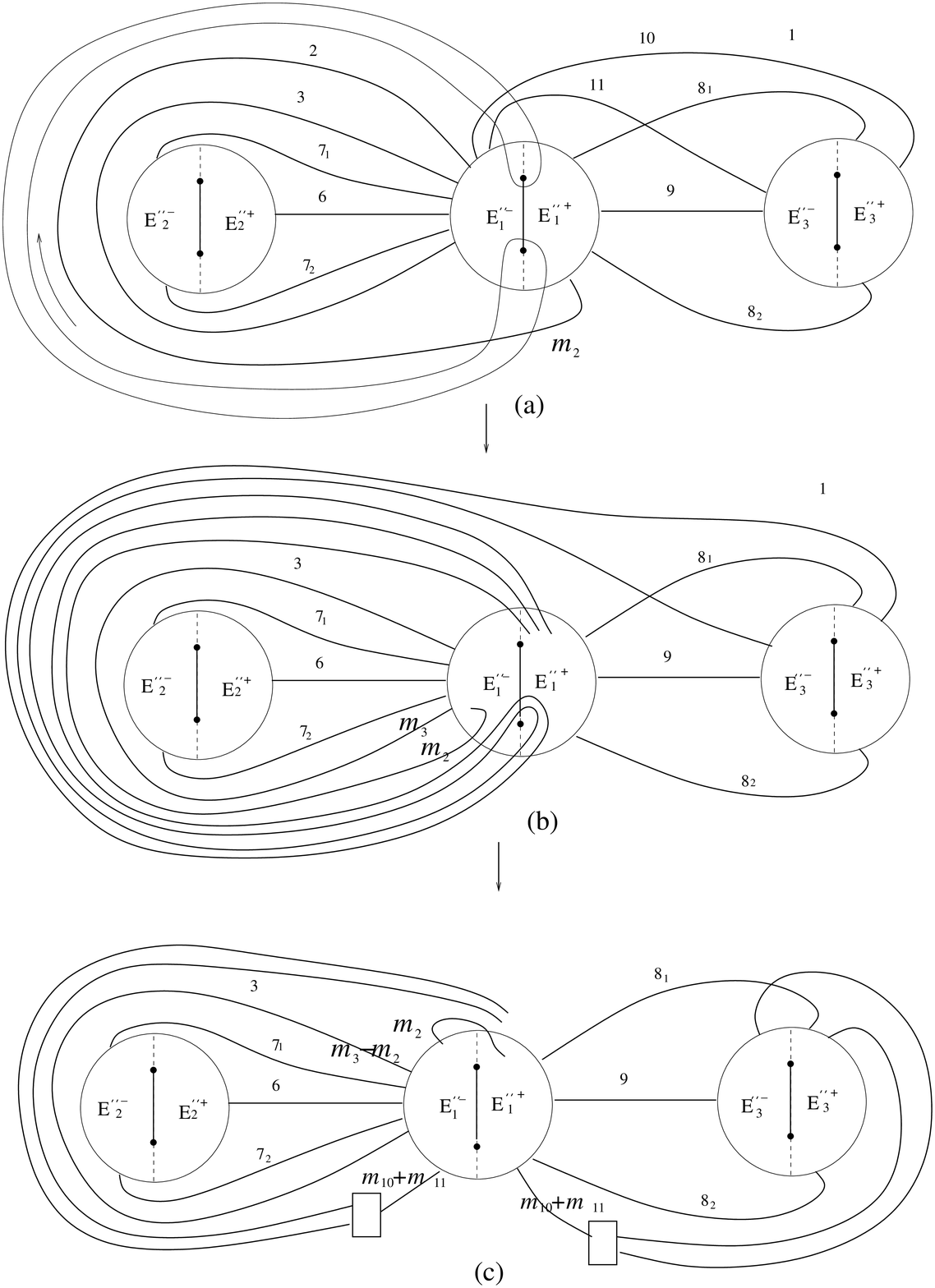}
\end{center}
\caption{}
\label{I3}
\end{figure}

First, we assume that $m_2,m_3> 0$ and $m_1=0$. We note that there are two type 2 in a standard diagram and they cannot coexist. Without loss of generality, we choose the  diagram $(a)$ as in Figure~\ref{I3}.\\

Then we apply $\delta_3$ to $\partial A$ to reduce the minimal intersection number of $\gamma_0$ with $\partial E$. The diagram $(b)$ of Figure~\ref{I3} shows there is no possibility to have $x_{22}>0$ or $x_{33}>0$ for $\delta_3(\gamma_0)$.\\

We note that if $ m_3\leq m_2+1$ then $x_{11}$ of $\delta_3(\gamma_0)$ is less than $2$. Therefore, $\delta_3(\gamma_0)$ does not bound an essential disk in $B^3-\epsilon$ since $x_{11}\geq 2$ if $\delta_3(\gamma_0)$ bounds an essential disk in $B^3-\epsilon$. This implies that $\gamma_0$ also does not bound an essential disk in $B^3-\epsilon$. This makes a contradiction. Therefore, $m_3>m_2+1$. So, we have the diagram $(c)$ of Figure~\ref{I3}.\\

Now, we note that $|\delta_3(\gamma_0)\cap \partial E|=|\gamma_0\cap \partial E|-4m_2$ as the diagram $(c)$ of  Figure~\ref{I3}.\\

First, we note that $(p_{21},q'_{21})=(p_2,q_2')$.  Also, we know that $p_{31}=p_3$. We note that the rightmost arc type coming to $E_3''$ is $8_2$ before taking $\delta_3$, but the  arc type $8_2$ moved around counterclockwise by $2(m_{10}+m_{11})$ after taking $\delta_3$. Therefore, $q'_{31}=q_3+2(m_{10}+m_{11})$.\\
 
We note that $p_{11}=p_1-2m_2$.  We also note that $q'_{11}=q'_1+m_2-(m_{10}+m_{11})$ by considering the incoming of the arc type for $x_{11}$ to the $E_1''^-$. In the diagram $(a)$, the arc type $2$ for $x_{11}$ is coming to the $E_1''^-$ after the types $10$ and $11$. However, after applying $\delta_3$ to $\gamma_0$ the arc type $3$ for $x_{11}$ is the right of some arcs with the weight $m_2$ which are not for $x_{11}$ in the $E_1''^-$.
Therefore, we have the following formula for the parameter changes.\\

 $(p_{11},q'_{11},p_{21},q'_{21},p_{31},q'_{31})=(p_1-2m_2,q'_1+m_2-(m_{10}+m_{11}),p_2,q_2',p_3,q_3'+2(m_{10}+m_{11}))$.
 \\
 
 Also, we check that $p_{11}+p_{21}+p_{31}=p_1-2m_2+p_2+p_3<p_1+p_2+p_3$ since $m_2>0$.\\

We note that if $\gamma_0$ has another type 2 then we need to apply $\delta_3^{-1}$ to reduce the sum of $p_i$ for $\gamma_0$.\\

Actually, the following is the formula for the parameter changes by $\delta_3^{-1}$.\\

 $(p_{11},q'_{11},p_{21},q'_{21},p_{31},q'_{31})=(p_1-2m_2,q'_1-m_2+(m_{10}+m_{11}),p_2,q_2',p_3,q_3'-2(m_{10}+m_{11}))$.\\

The case that $m_1,m_2>0$ and $m_3=0$ is analogous to the previous case.\\

Now, we assume that $m_1,m_2,m_3>0$.\\

\begin{figure}[htb]
\begin{center}
\includegraphics[scale=.27]{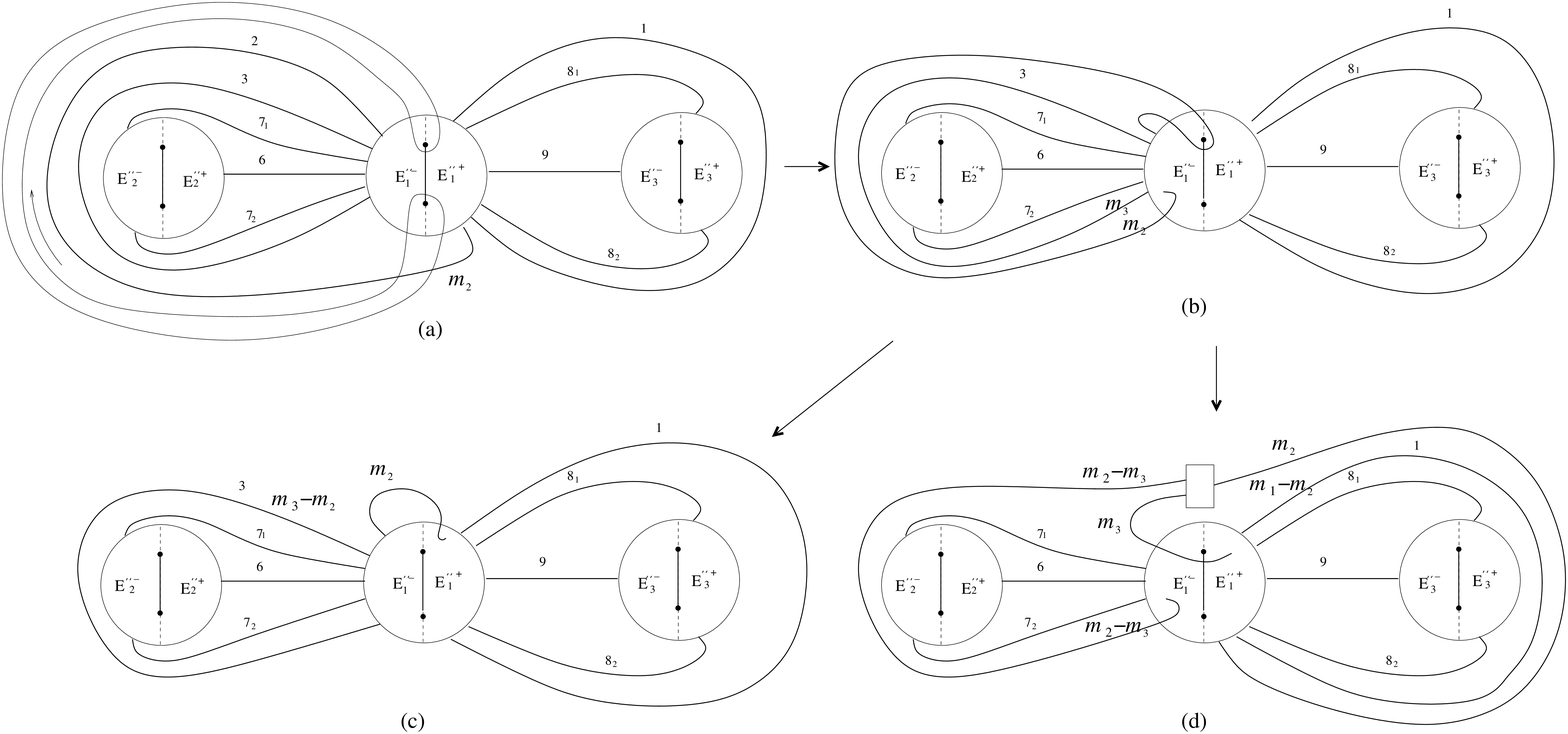}
\end{center}
\caption{}
\label{I4}
\end{figure}

Then we have the diagram $(a)$ of Figure~\ref{I3}. After applying $\delta_3$ we can get the diagram $(b)$ of Figure~\ref{I3}.\\

First of all, we note that if $m_2\geq m_1,m_3$ then $\gamma_0$ does not bound an essential disk in $B^3-\epsilon$ since $\delta_3(\gamma_0)$ has $x_{11}=0$ and in the diagram $(b)$ there is no arc can occur $x_{22}$ and $x_{33}$. So, we consider the two cases $m_2<m_3$ and $m_3\leq m_2<m_1$.\\

\begin{enumerate}
\item First, we assume that $m_2<m_3$. Then we can get the diagram $(c)$ in Figure~\ref{I3}.\\

We note that $(p_{21},q'_{21})=(p_2,q'_2)$ and $(p_{31},q'_{31})=(p_3,q'_3)$. Also, $p_{11}=p_1-2m_2$.  In the diagram $(a)$, the arc type $2$ for $x_{11}$ is coming to the $E_1''^-$ leftmost. However, after applying $\delta_3$ to $\gamma_0$ the arc with the weight $m_2$ which is not for $x_{11}$ is the left of the arc type $2$ in the $E_1''^-$. This implies that $q'_{11}=q'_1+m_2$. So, we have the following formula in this case.\\

 $(p_{11},q'_{11},p_{21},q'_{21},p_{31},q'_{31})=(p_1-2m_2,q'_1+m_2,p_2,q_2',p_3,q'_3)$.
 \\
 
 Also, we can check that  $p_{11}+p_{21}+p_{31}=p_1-2m_2+p_2+p_3<p_1+p_2+p_3$ since $m_2>0$.\\
 
 \item Now, we assume that $m_3\leq m_2<m_1$. Then we can get the diagram $(d)$ of Figure~\ref{I3}.\\
 
 We note that $(p_{21},q'_{21})=(p_2,q'_2)$ and $(p_{31},q'_{31})=(p_3,q'_3)$. We also can check that $p_{11}=p_1-2m_2$. We note that the arc with the weight $m_2-m_3$ in $E_1''^-$ cannot connect to the arc with the weight $m_3$ since $m_1>m_3$. In the diagram $(a)$, the arc type $2$ for $x_{11}$ is coming to the $E_1''^-$ leftmost. However, after applying $\delta_3$ to $\gamma_0$ the arc with the weight $m_3$ which is  for $x_{11}$ is replacing the position of the arc type $2$ in the $E_1''^-$. This implies that  $q'_{11}=q'_1+(m_3-m_2)$. So, we have the following formula in this case.\\
 
  $(p_{11},q'_{11},p_{21},q'_{21},p_{31},q'_{31})=(p_1-2m_2,q'_1+(m_3-m_2),p_2,q_2',p_3,q'_3)$.\\
 
 Also, we can check that  $p_{11}+p_{21}+p_{31}=p_1-2m_2+p_2+p_3<p_1+p_2+p_3$ since $m_2>0$.
 \end{enumerate}

\begin{figure}[htb]
\begin{center}
\includegraphics[scale=.30]{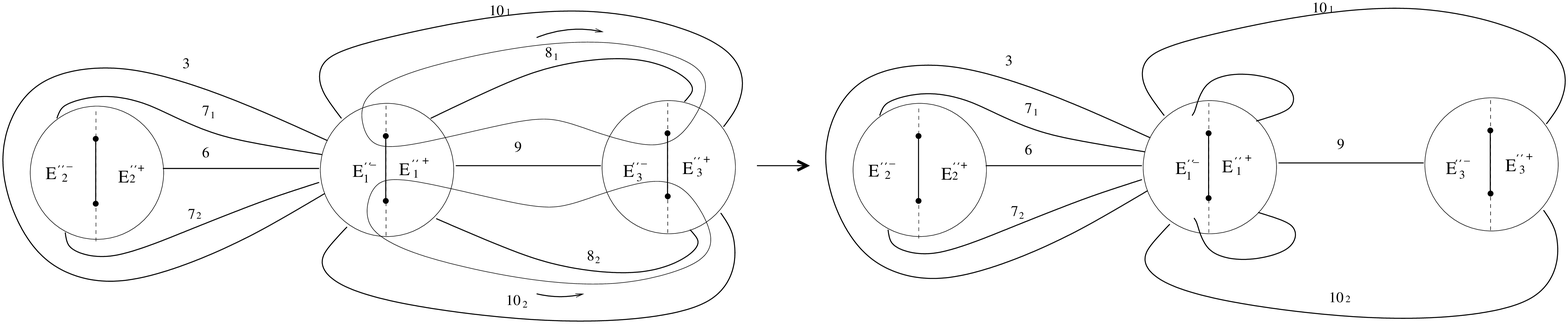}
\end{center}
\caption{}
\label{I5}
\end{figure}

Now, we assume that $m_2=0$. \\

We note that if $m_1,m_3>0$ then $m_1=m_3=1$ and $m_i=0$ for all $i\neq 1,3$. Otherwise, $\gamma_0$ is not a simple closed curve. Moreover, if $m_1=m_3=1$ then $\gamma_0$ bounds an essential disk in $B^3-\epsilon$ and the algorithm stops. So, we may assume that $m_3>0$, $m_1=0$ for the rest of this algorithm.\\

Then, we have consider the following three subcases.\\

\begin{enumerate}
\item $m_{11}=0$: We have the left diagram of Figure~\ref{I5}. Now, apply $\delta_1^{-1}\delta_2$ to $\gamma_0$ to get the right diagram of Figure~\ref{I5}. 
We note that $(p_{21},q'_{21})=(p_2,q'_2).$ Also, $p_{31}=p_3-m_{8}$ and  $q'_{31}=q'_3+m_8$.
Moreover, $p_{11}=p_1-m_8$ and $q'_{11}=q'_1-m_{8_1}$.\\

Therefore, we have the following formula for the Dehn's parameter changes.\\

$(p_{11},q'_{11},p_{21},q'_{21},p_{31},q'_{31})=(p_1-m_8,q'_1-m_{8_1},p_2,q'_2,p_3-m_8,q'_3+m_8).$\\

We can check that $p_{11}+p_{21}+p_{31}=p_1+p_2+p_3-2m_{8}$.\\

\begin{figure}[htb]
\begin{center}
\includegraphics[scale=.30]{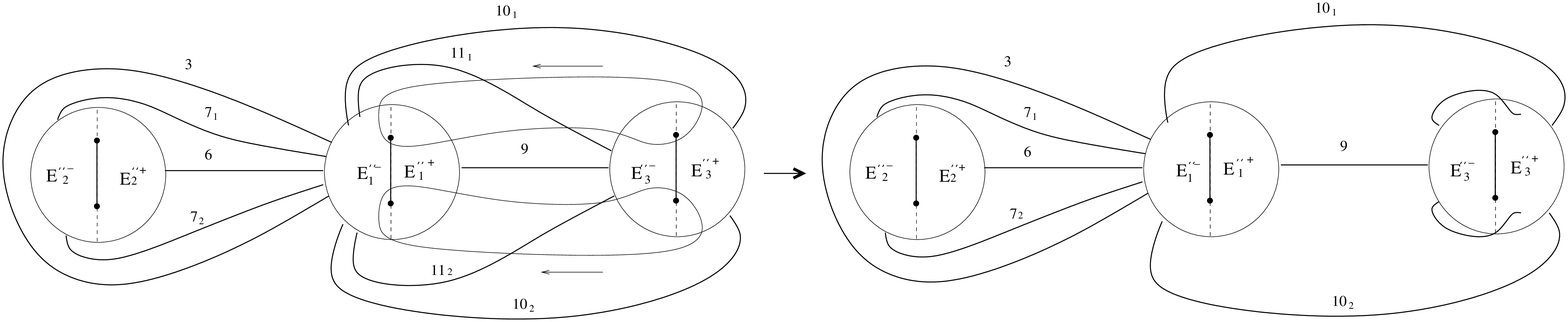}
\end{center}
\caption{}
\label{I6}
\end{figure}

\item $m_{8}=0$: We have the left  diagram of Figure~\ref{I6}. Now, apply $\delta_1\delta_2^{-1}$ to $\gamma_0$ to get the right diagram of Figure~\ref{I6}.
We know that $(p_{21},q'_{21})=(p_2,q'_2).$ Also, $p_{31}=p_3-m_{11}$ and  $q'_{31}=q'_3+m_{11}$.
Moreover, $p_{11}=p_1-m_{11}$ and $q'_{11}=q'_1-m_{11_1}$. So, we have the following formula for the Dehn's parameter changes.\\

$(p_{11},q'_{11},p_{21},q'_{21},p_{31},q'_{31})=(p_1-m_{11},q'_1-m_{11_1},p_2,q'_2,p_3-m_{11},q'_3+m_{11}).$\\

We also can check that $p_{11}+p_{21}+p_{31}=p_1+p_2+p_3-2m_{11}$.\\

\item $m_8,m_{11}> 0$: Since $m_{8_1}$ and $m_{11_1}$ cannot coexist, we assume that $m_{8_1}>0$ without loss of generality.
Then $m_{11_2}>0$ and $m_{8_2}=m_{11_1}=0$. Figure~\ref{I7}  shows this case. \\
\begin{figure}[htb]
\begin{center}
\includegraphics[scale=.30]{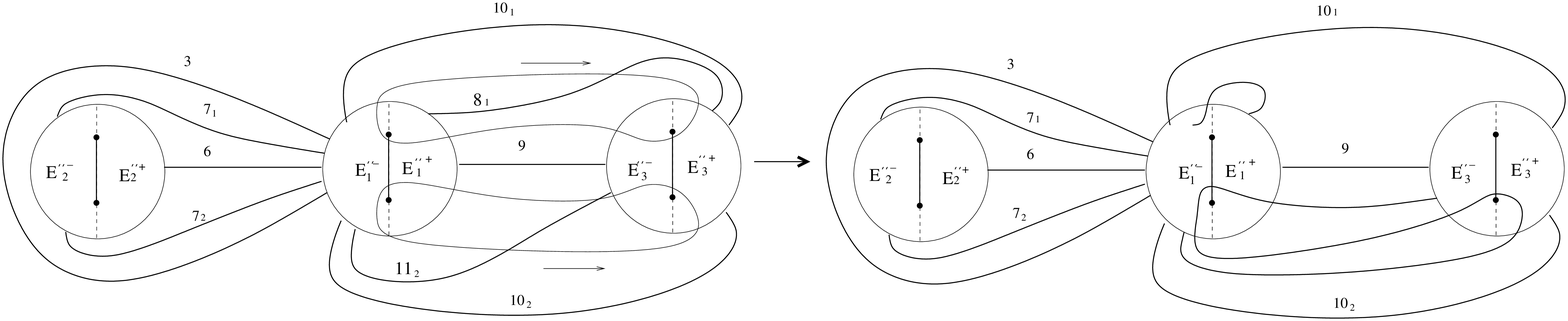}
\end{center}
\caption{}
\label{I7}
\end{figure}

Now, apply $\delta_1^{-1}\delta_2$ to the left diagram of Figure~\ref{I7}. For the connectivities in $E_1''$ and $E_3''$, we have $(m_3+m_7+m_6)+m_{10}+m_{11_2}=m_{8_1}+m_9$ and
$m_9+m_{11_2}=m_{10}+m_{8_1}$. So, we have $(m_3+m_7+m_6)+2m_{11_2}=2m_{8_1}$. This implies that $m_{8_1}>m_{11_2}$ since $m_3>0$.\\

We also note that $m_3+m_{10_1}\geq m_{8_1}+2$ to have at least two type 3 for $\delta_1^{-1}\delta_2(\gamma_0)$.
So, we know that $(p_{21},q'_{21})=(p_2,q'_2)$. Also, we know that $p_{31}=p_3-(m_{8_1}-m_{11_2})$ and $q'_{31}=q'_3+m_{8_1}.$
Moreover, $p_{11}=p_1-(m_{8_1}-m_{11_2})$ and $q'_{11}=q'_1-m_{8_1}$. So, we have the following formula for the Dehn's parameter changes.\\

$(p_{11},q'_{11},p_{21},q'_{21},p_{31},q'_{31})=(p_1-(m_{8_1}-m_{11_2}),q'_1-m_{8_1},p_2,q'_2,p_3-(m_{8_1}-m_{11_2}),q'_3+m_{11_2}).$
Also, we can check that  $p_{11}+p_{21}+p_{31}=p_1+p_2+p_3-2(m_{8_1}-m_{11_2})$.\\

For the case that $m_{8_2}>0$, $m_{11_1}>0$ and $m_{8_1}=m_{11_2}=0$, we  still need to apply $\delta_1^{-1}\delta_2$. to reduce the minimal intersection number of $\gamma_0$ with $\partial E$. Then we can have the following formula for the Dehn's parameter changes.\\

  $(p_{11},q'_{11},p_{21},q'_{21},p_{31},q'_{31})=(p_1-(m_{8_2}-m_{11_1}),q'_1+m_{11_1},p_2,q'_2,p_3-(m_{8_2}-m_{11_1}),q'_3-m_{8_2}).$ Also, we can check that  $p_{11}+p_{21}+p_{31}=p_1+p_2+p_3-2(m_{8_2}-m_{11_1})$.
\end{enumerate}
\end{proof}

If $\gamma_0$ is in standard position in $I'$ with $m_1>0$, then we rotate $\gamma_0$ with $180^\circ$ about the center of $E_1'$ to have a new simple closed curve $\eta$ which is in standard position in $I'$ with $m_3>0$. We note that $\eta$ also bounds an essential disk in $B^3-\epsilon$ since the rotation preserves $\infty$ tangle.\\

If the set of weights $m_i$ for $\gamma_0$ satisfies one of the conditions $(7)-(9)$ of Theorem~\ref{T93}, then we stop the algorithm to say that $\gamma_0$ does not bound an essential disk in $B^3-\epsilon$. If the set of weights $m_i$ for $\gamma_0$ satisfies the condition $(5)$ of Theorem~\ref{T93} then we stop the algorithm to say that $\gamma_0$ does bound an essential disk in $B^3-\epsilon$. If not, i.e., the set of weights $m_i$ for $\gamma_0$ satisifies one of the conditions $(1)-(4),(6)$ of Theorem~\ref{T93}, then we reduce the sum of $m_i$ by using  
 the formulas for the Dehn's parameter changes after applying one of four homeomorphism as in Theorem~\ref{T93}. Then with the new Dehn's parameters we can continue to follow this algorithm until either the data in each step fails to bound an essential disk in $B^3-\epsilon$ or $m_i=0$ for all $i=1,2,...,11$. We note that if $m_i=0$ for all $i$ for $\gamma_0$ then it bounds an essential disk in $B^3-\epsilon$.

\section{Examples of the use of the algorithm}

Example 1: $\infty$ tangle and $T$.\\

Consider $\infty$ tangle as in Figure~\ref{J1}. Then the extension of $\sigma_5\sigma_3\sigma_1\sigma_2^{-1}\sigma_3\sigma_1$ to $B^3$  makes  a rational 3-tangle $T$.
For every strings of $T$, if we choose the other two strings then they are isotopic to a trivial rational 2-tangle in $B^3$. However, $T$ is not isotopic to $\infty$ tangle.\\

\begin{figure}[htb]
\begin{center}
\includegraphics[scale=.25]{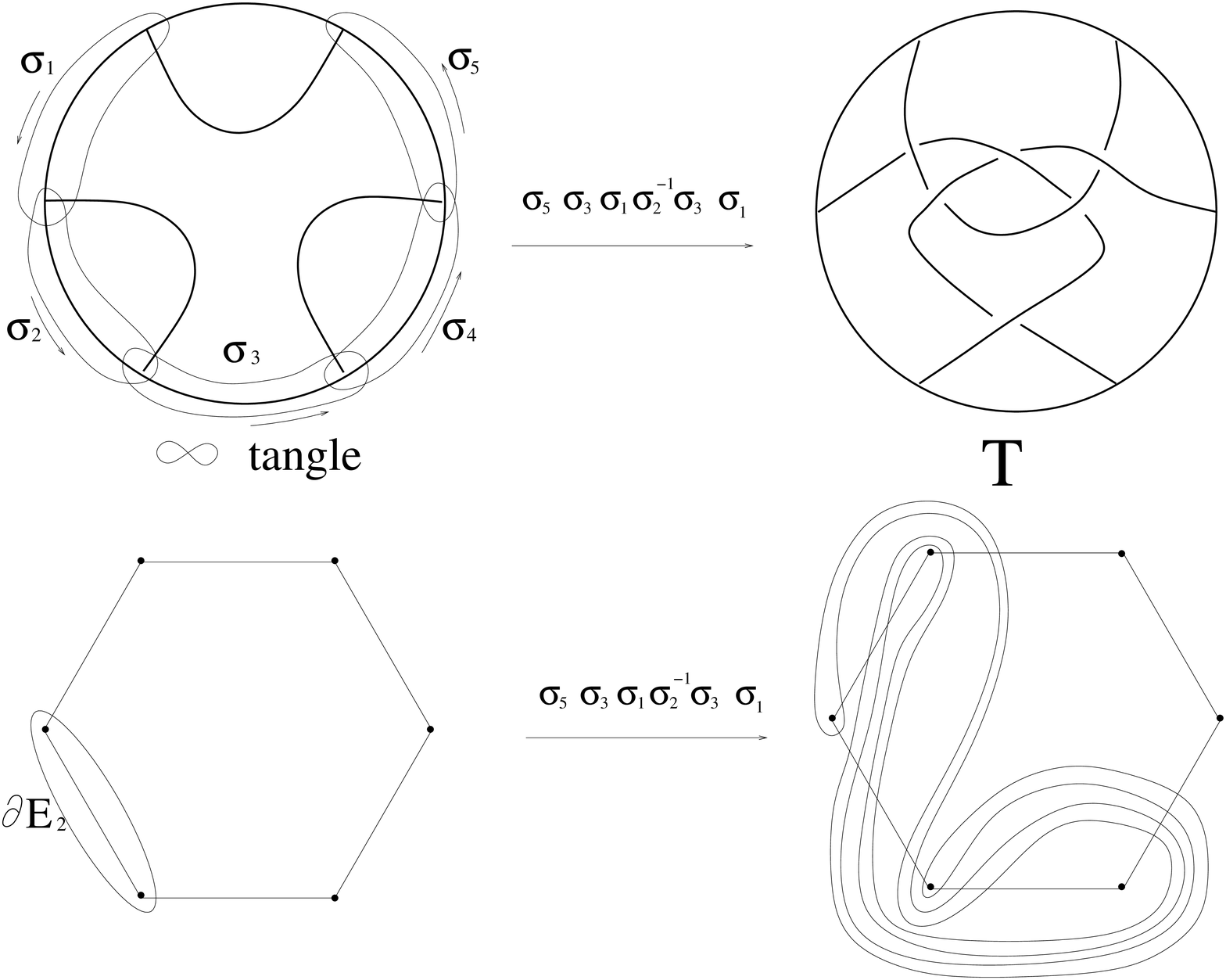}
\end{center}
\caption{Example 1}
\label{J1}
\end{figure}

In order to show this, we consider $\partial E_2$ which bounds an essential disk in $B^3-\epsilon$. We notice that $w_{46}=w^{46}=1$ and all the other weights are zero for $[\partial E_2]$.\\

Now, consider the simple closed curve $\alpha =\sigma_5\sigma_3\sigma_1\sigma_2^{-1}\sigma_3\sigma_1(\partial E_2)$. We will show that $\alpha$ does not bound an essential disk in $B^3-\epsilon$.  Also, let $w_{ij}(f)$  and $w^{ij}(f)$ be the weights of $[f(\partial E_2)]$.\\

Let $f_1=\sigma_1$, $f_2=\sigma_3\sigma_1,f_3=\sigma_2^{-2}\sigma_3\sigma_1$, $f_4=\sigma_1\sigma_2^{-2}\sigma_3\sigma_1$, $f_5=\sigma_3\sigma_1\sigma_2^{-2}\sigma_3\sigma_1$ and $f_6=\sigma_5\sigma_3\sigma_1\sigma_2^{-2}\sigma_3\sigma_1$.\\

 From the weight change formulas, we can get $w_{14}(f_1)=w_{56}(f_1)=w^{16}(f_1)=w^{45}(f_1)=1$ and all the other weights are zero.\\

 From $w_{ij}(f_1)$ and $w^{ij}(f_1)$, we get $w_{15}(f_2)=w_{56}(f_2)=w_{34}(f_2)=w^{35}(f_2)=w^{45}(f_2)=w^{16}(f_2)=1$ and all the other weights are zero.\\
 
 From $w_{ij}(f_2)$ and $w^{ij}(f_2)$, we get $w_{14}(f_3)=w_{46}(f_3)=w_{34}(f_3)=w^{36}(f_3)=w^{46}(f_3)=w^{16}(f_3)=1$ and
 $f_3(w)_{56}=f_3(w)^{45}=2$ and all the other weights are zero.\\
 
 From $w_{ij}(f_3)$ and $w^{ij}(f_3)$, we get $w_{14}(f_4)=w_{15}(f_4)=2,w_{34}(f_4)=1,w_{56}(f_4)=3$ and $w^{15}(f_4)=w^{35}(f_4)=1, w^{16}(f_4)=w^{45}(f_4)=3$ and all the other weights are zero.\\
 
 From $w_{ij}(f_4)$ and $w^{ij}(f_4)$, we get $w_{15}(f_5)=4,w_{34}(f_5)=w_{56}(f_5)=3,w_{35}(f_5)=1$ and $w^{15}(f_5)=1,w^{16}(f_5)=w^{45}(f_5)=3,w^{35}(f_5)=4$ and all the other weights are zero.\\

 Finally, we get $w_{15}(f_6)=4,w_{34}(f_6)=w_{56}(f_6)=3,w_{35}(f_6)=1$ and $w^{15}(f_6)=1,w^{16}(f_6)=w^{45}(f_6)=3,w^{35}(f_6)=4$ and all the other weights are zero.
 We notice that $w_{ij}(f_5)=w_{ij}(f_6)$ and $w^{kl}(f_5)=w^{kl}(f_6)$.\\
 
So, $\alpha$ has $p_1=w_{15}(f_6)=4,p_2=w_{15}(f_6)+w_{56}(f_6)+w_{35}(f_6)=8$ and $p_3=w_{34}(f_6)+w_{35}(f_6)=4$. This implies that $x_{12}=8$, $x_{23}=8$ and all other $x_{ij}=0$.\\

Especially, $x_{11}+x_{22}+x_{33}=0$. Therefore, $\alpha$ does not bound an essential disk in $B^3-\epsilon$.
This implies that $T$ is not isotopic to $\infty$ tangle. We can find the nine parameters for $\alpha$ by using the algorithm to check if $\alpha$ is left-twisted in $E_i'$. We remark that $\alpha$ is parameterized by $(4,0,-1,8,1,-1,4,0,0)$.\\

Example 2: $T$ and $T'$\\

 Now, consider $T'$ which is obtained by reversing all the crossings in $T$. Then we have Figure~\ref{J2}. We want to check whether $T'$ is isotopic to $T$ or not.
Let $f_7=\sigma_5f_6, f_8=\sigma_3f_7,...,f_{12}=\sigma_1 f_{11}=\sigma_1\sigma_3\sigma_2^{-1}\sigma_1\sigma_3\sigma_5 f_6 =(\sigma_5^{-1}\sigma_3^{-1}\sigma_1^{-1}\sigma_2\sigma_3^{-1}\sigma_1^{-1})^{-1}f_6.$\\

\begin{figure}[htb]
\begin{center}
\includegraphics[scale=.25]{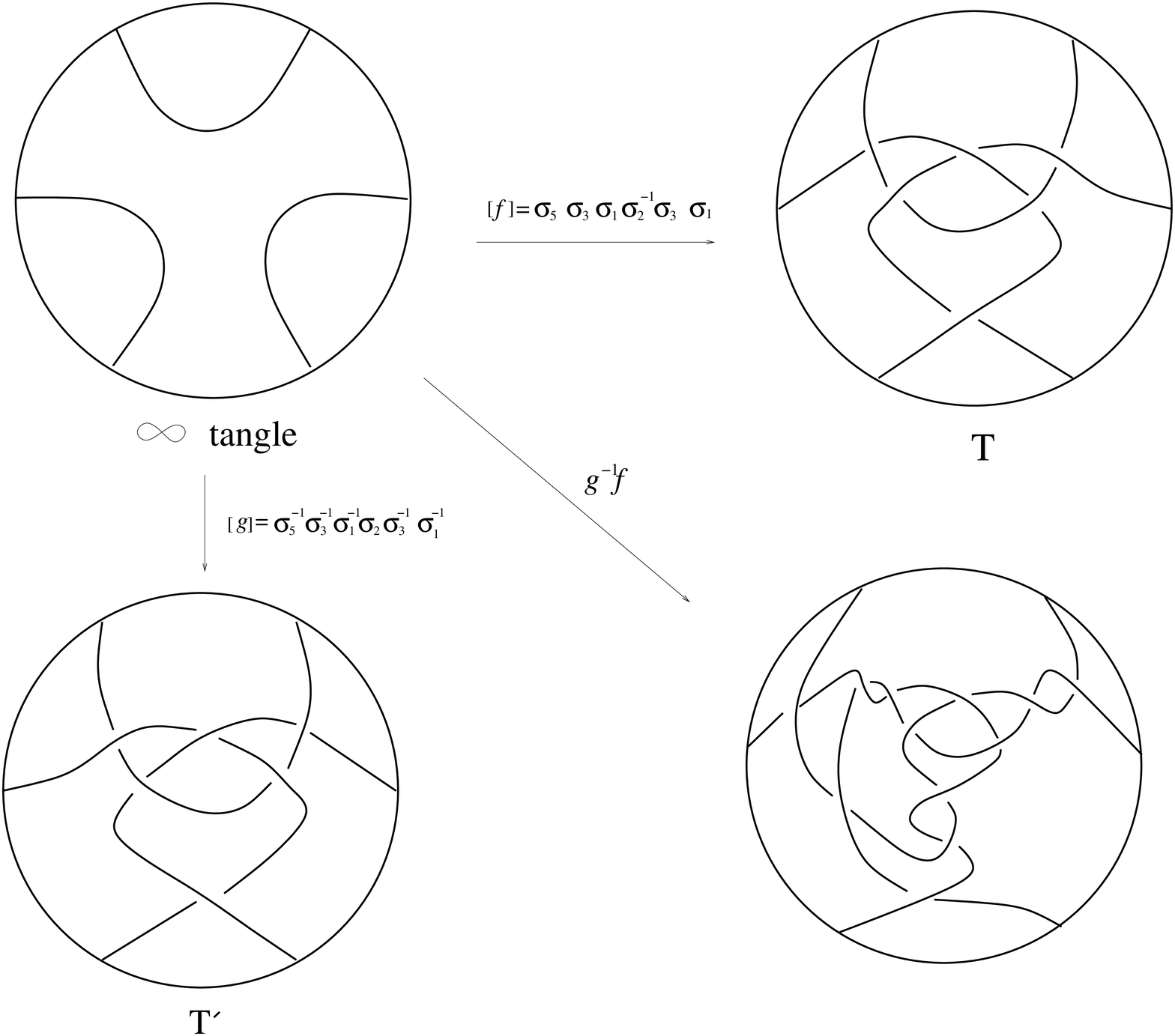}
\end{center}
\caption{Example 2}
\label{J2}
\end{figure}

We notice that if $T$ is isotopic to $T'$ then $f_{12}(\partial E)$ bound essential disks in $B^3-\epsilon$.\\

Consider $\partial E_2$. Then we already got $w_{ij}(f_6)$ and $w^{ij}(f_6)$ in the previous argument.\\

From $w_{ij}(f_6)$ and $w^{ij}(f_6)$, we can get $w_{ij}(f_7)$ and $w^{ij}(f_7)$. Actually, $w_{ij}(f_7)=w_{ij}(f_6)$ and $w^{ij}(f_7)=w^{ij}(f_6)$.\\

From $w_{ij}(f_7)$ and $w^{ij}(f_7)$, we get $w_{34}(f_8)=w_{56}(f_8)=3, w_{15}(f_8)=w_{35}(f_8)=4$ and $w^{15}(f_8)=1, w^{16}(f_8)=w^{45}(f_8)=3, w^{35}(f_8)=7$.\\

From $w_{ij}(f_8)$ and $w^{ij}(f_8)$, we get $w_{34}(f_9)=w_{56}(f_9)=3, w_{35}(f_9)=4, w_{15}(f_9)=7$ and $w^{16}(f_9)=w^{45}(f_9)=3, w^{15}(f_9)=4, w^{35}(f_9)=7$. \\

From $w_{ij}(f_9)$ and $w^{ij}(f_9)$,  we get $w_{34}(f_{10})=w_{14}(f_{10})=7, w_{46}(f_{10})=3, w_{56}(f_{10})=14$ and $w^{16}(f_{10})=w^{36}(f_{10})=7, w^{46}(f_{10})=3, w^{45}(f_{10})=14$.\\

From $w_{ij}(f_{10})$ and $w^{ij}(f_{10})$,  we get $w_{15}(f_{11})=w_{35}(f_{11})=7, w_{34}(f_{11})= w_{56}(f_{11})=17$ and $w^{16}(f_{11})=7,w^{36}(f_{11})=10, w^{35}(f_{11})=14, w^{45}(f_{11})=17$.\\

Then finally, we get $w_{35}(f_{12})=7, w_{34}(f_{12})= w_{56}(f_{12})=17, w_{15}(f_{12})=24$ and $w^{15}(f_{12})=7,w^{16}(f_{12})= w^{45}(f_{12})=17, w^{35}(f_{12})=24$ from $w_{ij}(f_{11})$ and $w^{ij}(f_{12})$.\\

Let $\beta=f_{12}(\partial E_2)$. Then $\beta$ has $p_1=w_{15}(f_{12})=24, p_2=w_{35}(f_{12})+w_{56}(f_{12})+w_{15}(f_{12}) =48$ and $p_3=w_{35}(f_{12})+w_{34}(f_{12})=24$. Therefore,  $x_{12}=48$, $x_{23}=48$ and all other $x_{ij}=0$. Especially, $x_{11}+x_{22}+x_{33}=0$. This implies that $f_{12}(\partial E_2)$ does not bound an essential disk in $B^3-\epsilon$. Therefore, $T$ is not isotopic to $T'$.

\end{document}